\documentclass[11pt,a4paper]{article}

\usepackage{amsmath}
\usepackage{amsthm}
\usepackage{amssymb}
\usepackage{mathtools}
\usepackage{hyperref}
\usepackage[inline]{enumitem}
\usepackage{amsfonts}
\usepackage{tikz}
\usepackage{graphicx}
\usepackage{stmaryrd}
\usepackage{color}
\usepackage[margin=1.5cm]{geometry}
\numberwithin{equation}{section}

\newcommand{\N}{\mathrm{N}}
\newcommand{\K}{\widetilde{K}}

\newcommand{\realR}{\mathbb{R}}
\newcommand{\compC}{\mathbb{C}}
\newcommand{\intZ}{\mathbb{Z}}
\newcommand{\natN}{\mathbb{N}}
\newcommand{\Prob}{\mathbb{P}}

\newcommand{\bigO}{\mathcal{O}}
\newcommand{\totalalpha}{\boldsymbol{\alpha}}

\newcommand{\ii}{\mathrm{i}}
\newcommand{\dd}{\mathrm{d}}
\newcommand{\Id}{\mathbf{1}}
\newcommand{\doubleX}{X\hspace{-8pt}X}

\newcommand{\lambdatilde}{\tilde{\lambda}}
\newcommand{\sigmatilde}{\tilde{\sigma}}

\usetikzlibrary{arrows}
\DeclareRobustCommand{\gueNE}{%
{\mathbin{\text{%
\ensuremath{\begin{tikzpicture}[scale=0.80]
  \draw [-*] (-.4em,.4em) -- (.6em,.4em);
  \draw [-] (.4em,.4em)--(.4em,-.4em);
  \draw [-] (.4em,-.4em)--(-.4em,-.4em);
  \draw [-] (-.4em,-.4em) -- (-.4em,.4em);
\end{tikzpicture}
}}}}}

\DeclareRobustCommand{\GUEupperright}{%
  \mathbin{\text{\rotatebox[origin=c]{0}{$\gueNE$}}}%
}

\DeclareRobustCommand{\GUEupperleft}{%
  \mathbin{\text{\rotatebox[origin=c]{90}{$\gueNE$}}}%
}

\DeclareRobustCommand{\GUElowerright}{%
  \mathbin{\text{\rotatebox[origin=c]{-90}{$\gueNE$}}}%
}

\DeclareRobustCommand{\GUElowerleft}{%
  \mathbin{\text{\rotatebox[origin=c]{180}{$\gueNE$}}}%
}

\DeclareRobustCommand{\Airyupperleft}{%
  \mathbin{\text{\rotatebox[origin=c]{45}{$\gueNE$}}}%
}

\DeclareRobustCommand{\Airyupperright}{%
  \mathbin{\text{\rotatebox[origin=c]{-45}{$\gueNE$}}}%
}

\DeclareRobustCommand{\Airylowerright}{%
  \mathbin{\text{\rotatebox[origin=c]{-135}{$\gueNE$}}}%
}

\DeclareRobustCommand{\Airylowerleft}{%
  \mathbin{\text{\rotatebox[origin=c]{135}{$\gueNE$}}}%
}

\newcommand{\la}{\langle}
\newcommand{\ra}{\rangle}

\renewcommand{\vec}[1]{\mathbf{#1}}

\DeclareMathOperator{\Airy}{Airy}
\DeclareMathOperator{\GUE}{GUE}
\DeclareMathOperator{\Ai}{Ai}
\DeclareMathOperator{\E}{\mathbb{E}}
\DeclareMathOperator{\var}{Var}
\DeclareMathOperator{\Gammadist}{Gamma}
\DeclareMathOperator{\Exp}{Exp}

\DeclareMathOperator{\arccot}{arccot}
\DeclareMathOperator{\outside}{out}
\DeclareMathOperator{\inside}{in}
\DeclareMathOperator{\leftside}{left}
\DeclareMathOperator{\rightside}{right}
\DeclareMathOperator{\PV}{P.V.}
\DeclareMathOperator{\as}{a.s.}
\DeclareMathOperator{\standard}{std}
\DeclareMathOperator{\scaled}{scaled}
\DeclareMathOperator{\Idmatrix}{Id}

\newcommand{\totala}{\mathbf{a}}
\newcommand{\Airya}{\Airy, \totala}
\newcommand{\GUEalpha}{\GUE, \totalalpha}
\newcommand{\Peche}{P\'{e}ch\'{e}}
\newcommand{\Painleve}{Painlev\'{e}}
\newcommand{\weakconv}{\stackrel{d}\longrightarrow}
\newtheorem{thm}{Theorem}
\newtheorem{lem}[thm]{Lemma}
\newtheorem{asm}{Assumption}
\newtheorem{pro}[thm]{Proposition}
\newtheorem{cor}[thm]{Corollary}

\theoremstyle{remark}
\newtheorem{rmk}{Remark}

\newcommand{\reservesigma}{\sigma}
\DeclareMathOperator{\upperintersect}{upper intersect}
\DeclareMathOperator{\lowerintersect}{lower intersect}
\newcommand{\round}{\circ}
\newcommand{\corner}{>}
\newcommand{\zig}{<}

\begin{document}

\allowdisplaybreaks

\begin{center}
\Large\bf
{ Eigenvector distribution in the critical regime of BBP transition} 
\end{center}

\renewcommand{\thefootnote}{\fnsymbol{footnote}}	
\vspace{1cm}
\begin{center}
 \begin{minipage}[t]{0.4\textwidth}
\begin{center}
{\large Zhigang Bao }\footnotemark[3] \\ \vspace{1ex}
\footnotesize{Hong Kong University of Science and Technology}\\
{\it mazgbao@ust.hk}
\end{center}
\end{minipage}
\hspace{8ex}
\begin{minipage}[t]{0.4\textwidth}
\begin{center}
{\large Dong Wang}\footnotemark[2]  \\ \vspace{1ex}
 \footnotesize{National University of Singapore}\\
{\it wangdong@wangd-math.xyz}
\end{center}
\end{minipage}

\footnotetext[3]{Supported by Hong Kong RGC grant GRF16300618, GRF 16301519, and NSFC 11871425}
\footnotetext[2]{Supported by Singapore AcRF grant R-146-000-217-112}

\renewcommand{\thefootnote}{\fnsymbol{footnote}}	

\end{center}
\vspace{1cm}
\begin{center}
  \begin{minipage}{0.8\textwidth}
    In this paper,   we study the random matrix model of Gaussian Unitary Ensemble (GUE) with fixed-rank (aka spiked) external source. We will focus on the critical regime of the Baik-Ben Arous-\Peche\ (BBP) phase transition and establish the distribution of the eigenvectors associated with the leading eigenvalues. The distribution is given in terms of a determinantal point process with extended Airy kernel. Our result can be regarded as an eigenvector counterpart of the BBP eigenvalue phase transition \cite{Baik-Ben_Arous-Peche05}. The derivation of the distribution makes use of the recently re-discovered {\it eigenvector-eigenvalue identity}, together with the determinantal point process representation of the  GUE minor process with external source.
  \end{minipage}
\end{center}

\section{Introduction}

In this paper, we consider the Gaussian Unitary Ensemble (GUE) with fixed-rank external source, also known as the spiked GUE in the literature, denoted by 
\begin{equation} \label{20040301}
  G_{\totalalpha} \equiv G^{(N)}_{\totalalpha} := G + \sum^k_{i=1} \alpha_i \vec{e}_i \vec{e}^*_i,
\end{equation}
where $G=(g_{ij})_{N,N}$ is a standard $N$-dimensional GUE, i.e., $g_{ii}\sim \N(0,1)$ $(1\leq i\leq N)$; $g_{ij}\sim \N(0,\frac12)+\ii \N(0,\frac12)$ $(1\leq i<j\leq N)$ are independent random variables with standard real/complex normal distributions, and $g_{ji}=\overline{g_{ij}}$.  
Here $\totalalpha=(\alpha_1, \ldots, \alpha_k)\in\mathbb{R}^k$ is a deterministic vector with fixed dimension $k$, and $\{\vec{e}_i\}$ is the standard basis of $\mathbb{R}^N$.  The entire discussion in this paper works under the following more general setting
\begin{equation} \label{20040302}
  G_{\totalalpha, \vec{v}}= G+\sum_{i=1}^k \alpha_i \vec{v}_i\vec{v}_i^*
\end{equation}
with any deterministic orthonormal vectors $\vec{v}_i\in \mathbb{C}^N$. Nevertheless, due to the unitary invariance of GUE, it would be sufficient to focus on the model in \eqref{20040301}. 

Throughout the paper, we will be focusing on the critical regime of the well-known Baik-Ben Arous-\Peche\ (BBP) phase transition \cite{Baik-Ben_Arous-Peche05}, and thus make the following assumption on $\alpha_i$'s
\begin{asm} \label{main assum}
  There exist fixed constants $a_1, \ldots, a_k\in \mathbb{R}$ such that 
  \begin{equation}
    \alpha_i = \sqrt{N} + N^{\frac{1}{6}} a_{k - i + 1}, \quad i = 1, \dotsc, k.
  \end{equation}
 \end{asm}
 We emphasize here that  $\alpha_i$'s  are unordered and may be identical. Whenever the ordered parameters are needed in some local discussion, we will use $\alpha_{(j)}$ to denote the $j$-th largest $\alpha_i$.
 
We further denote the ordered eigenvalues of $G_{\totalalpha}$ by \footnote{Throughout this paper, for all the random matrices we consider, the eigenvalues are distinct with probability $1$. Hence we always assume the simplicity of the eigenvalues without further explanation.}
\begin{equation} \label{20040810}
  \reservesigma_1 > \dotsb > \reservesigma_N
\end{equation}
and set
\begin{equation} \label{20040811}
  \vec{x}_i = (x_{i1}, \ldots, x_{iN})^{\top}
\end{equation}
to be the unit eigenvector associated with $\reservesigma_i$ \footnote{Since the eigenvalues are assumed to be distinct, $\vec{x}_i$ are unique up to an angular factor. We ignore the angular factor since we consider only the moduli of the components throughout the paper.}. In this paper, we are primarily interested in the limiting  distribution of $|x_{ij}|^2$'s with bounded $i$, after appropriate normalization, as the dimension $N \to \infty$. Observe that $|x_{ij}|^2$ can be understood as the square of the projection of eigenvector $\vec{x}_i$ onto the direction $\vec{e}_j$. Due to the unitary invariance, our results can also  be applied to the projection $|\la \vec{x}_i^{\vec{v}}, \vec{v}_j\ra|^2$, where we used $\vec{x}_i^{\vec{v}}$ to denote the $i$-th eigenvector of $G_{\totalalpha,\vec{v}}$ in \eqref{20040302}.  Before we state the main results, we first give a brief review of the literature on the  eigenvalue and eigenvector of  random matrices with fixed-rank deformation, in Section \ref{s.1.1}, and then we present the definition of the extended Airy kernel in Section \ref{s.1.2}, with which we will then state our main results  in Section \ref{s.1.3}.

\subsection{Random matrix with fixed-rank deformation} \label{s.1.1}

Our model in \eqref{20040301} is in the category of  the random matrices with fixed-rank deformation, which also includes  the spiked sample covariance matrix and the signal-plus-noise model as  typical examples. A vast amount of work has been devoted to  understanding  the limiting behavior of the extreme eigenvalues and the associated eigenvectors of the deformed models.  Since the seminal work of Baik, Ben Arous and \Peche\ \cite{Baik-Ben_Arous-Peche05}, it is now well-understood that the extreme eigenvalues undergo a so-called BBP phase transition along with the change of the strength of the deformation.  Specifically, there exists a critical threshold such that the extreme eigenvalue of the deformed matrix will stick to the right end point of the limiting spectral distribution if the strength of the deformation is less than or equal to the threshold, and will otherwise jump out of the support of the limiting spectral distribution. In the latter case we often call the extreme eigenvalue as an {\it outlier}. Moreover, the fluctuation of the extreme eigenvalues in different regimes (subcritical, critical and supercritical) are also identified in \cite{Baik-Ben_Arous-Peche05} for the complex spiked covariance matrix.  Particularly,  for the deformed GUE in (\ref{20040301}),  the phase transition takes place on the scale $\alpha_i = \sqrt{N} + \bigO(N^{1/6})$.  Hence, for the deformed GUE, more specifically, the regimes $\alpha_i < \sqrt{N} - N^{1/6 + \varepsilon}$,  $\alpha_i = \sqrt{N} + \bigO(N^{1/6})$ and $\alpha_i > \sqrt{N} + N^{1/6 + \varepsilon}$  will be referred to as subcritical, critical and supercritical, respectively, in the sequel. We also refer to \cite{Bai-Yao12}, \cite{Baik-Silverstein06}, \cite{Benaych_Georges-Nadakuditi11}, \cite{Benaych_Georges-Nadakuditi12}, \cite{Capitaine-Donati_Martin17}, \cite{Ding20}, \cite{Knowles-Yin13}, \cite{Paul07} and the reference therein for the first-order limit of the extreme eigenvalue of various related models. The fluctuation of the extreme eigenvalues of various models has been considered in \cite{Bai-Yao08}, \cite{Bai-Yao12}, \cite{Baik-Wang10a}, \cite{Baik-Wang10}, \cite{Benaych_Georges-Guionnet-Maida11}, \cite{Bertola-Buckingham-Lee-Pierce11}, \cite{Bertola-Buckingham-Lee-Pierce11a}, \cite{Bloemendal-Virag11}, \cite{Bloemendal-Virag11a}, \cite{Capitaine-Donati_Martin-Feral09}, \cite{Capitaine-Donati_Martin-Feral12}, \cite{El_Karoui07}, \cite{Feral-Peche07}, \cite{Knowles-Yin13}, \cite{Knowles-Yin14}, \cite{Paul07}, \cite{Peche05}, \cite{Pizzo-Renfrew-Soshnikov13}, \cite{Wang08}, \cite{Wang11}. 

In parallel to the results of the extreme eigenvalues, there are some corresponding results on eigenvectors in the literature. Suppose $G_{\totalalpha}$ is given in \eqref{20040301} while $\alpha_1, \dotsc, \alpha_k$ are significantly away from the critical threshold, say, $\min_j|\alpha_j-\sqrt{N}|\geq \varepsilon\sqrt{N}$ for some positive constant $\varepsilon$, it is known that
\begin{enumerate*}[label=(\roman*)]
\item
  if $\alpha_{(i)} \leq (1 - \varepsilon)\sqrt{N}$, then $\vec{x}_i$, the eigenvector associated with the $i$-th largest eigenvalue $\reservesigma_i$, is delocalized in the sense $\lVert \vec{x}_i \rVert_{\infty} \leq N^{-1/2 + \delta}$ for any small constant $\delta>0$ with high probability;
\item
  if $\alpha_{(i)} \geq (1 + \varepsilon)\sqrt{N}$, then $\vec{x}_i$ has an order one bias on the direction of the deformation $\vec{e}_{(i)}$. 
\end{enumerate*}
Here we used $\alpha_{(i)}$ to denote the $i$-th largest value of all $\alpha_j$'s and $\mathbf{e}_{(i)}$ is the  canonical basis vector with the corresponding index. 
In \cite{Benaych_Georges-Nadakuditi11}, \cite{Benaych_Georges-Nadakuditi12}, \cite{Capitaine18}, \cite{Ding20}, \cite{Paul07}, the behavior of the extreme eigenvectors has been studied on the level of the first order limit. A detailed discussion of eigenvector behavior in the full subcritical regime  and supercritical regime  can be found in the recent work \cite{Bloemendal-Knowles-Yau-Yin16}, which was done for the spiked sample covariance matrix. Especially, the discussion in \cite{Bloemendal-Knowles-Yau-Yin16} indicates that in case there is an $\alpha_j$ close to the critical threshold, i.e,  $\alpha_j=(1+o(1))\sqrt{N}$, for any fixed $i$ such that $\reservesigma_i$ is not an outlier, $\vec{x}_i$ will have a  bias of small order towards the direction of $\vec{e}_j$.    On the level of the fluctuation, the limiting behavior of the extreme eigenvectors has  not been  fully studied yet.  By establishing a general universality result of the eigenvectors of the sample covariance matrix in the null case, the authors of \cite{Bloemendal-Knowles-Yau-Yin16} are able to establish  the law of the  eigenvectors of the spiked covariance matrices  in the subcritical regime. In this regime, the eigenvector distribution is similar  to (up to appropriate scaling) that of the bulk and edge  regime of Wigner matrices without spikes; see  \cite{Bourgade-Yau17}, \cite{Knowles-Yin13a}, \cite{Tao-Vu12}, \cite{Benigni19} and  \cite{Marcinek-Yau20} for  instance. More specifically, in the subcritical regime, the limiting distribution of the square of eigenvector components (after appropriate scaling) is given by $\chi^2$ distribution, which tells the asymptotic Gaussianity of the eigenvector components themselves (without taking square). Although the result was established for sample covariance matrix only in \cite{Bloemendal-Knowles-Yau-Yin16}, it can be extended to deformed Wigner without essential difference. In the supercritical regime, the fluctuation of the eigenvectors was recently studied in \cite{Bao-Ding-Wang18}, \cite{Bao-Ding-Wang-Wang20}, \cite{Capitaine-Donati_Martin18} for various models with generally distributed matrix entry. For generally distributed deformed Wigner matrix, the leading eigenvector distribution is non-universal in the supercritical regime; see \cite{Capitaine-Donati_Martin18}. However, if one restricts the discussion to the deformed GUE, then  the limiting distribution of the square of eigenvector components (after appropriate centering and scaling) is given by Gaussian. Although the discussion in \cite{Capitaine-Donati_Martin18} has only covered the regime $\alpha_i\geq (1+\varepsilon)\sqrt{N}$, one can use the approach in \cite{Bao-Ding-Wang-Wang20} to extend the result to the full supercritical regime $\alpha_i \geq \sqrt{N} + N^{1/6 + \varepsilon}$. 

The aforementioned works leave the eigenvector distribution in the critical regime undiscussed. In this paper, we will establish the eigenvector distribution in the critical regime, i.e., $\alpha_i = \sqrt{N} + \bigO(N^{1/6})$. Here although we are dealing with the deformed GUE only, our methodology and result reveal certain universality of eigenvector distribution of random matrices with fixed-rank deformation in the critical regime of BBP transition, within the class of unitary invariant ensemble. Especially, our discussion can apply similarly to the fixed-rank deformed Laguerre Unitary Ensemble (LUE) that is also known as spiked Wishart ensemble or spiked sample covariance matrices in statistics and on which the BBP transition is most intensively studied, and the fixed-rank deformed Jacobi Unitary Ensemble (JUE, aka MANOVA ensemble in statistics) by using the corresponding correlation kernel formulas in \cite{Adler-van_Moerbeke-Wang11}. We expect universal asymptotic results in these models. We also remark here that the universality has not yet been proved or disproved for generally distributed Wigner matrices with fixed-rank deformation in the critical regime of BBP transition, even on the eigenvalue level. 

At last, we remark that for both Hermitian type (complex) random matrices and real symmetric type random matrices, the BBP transition may happen under fixed rank deformations. The study of eigenvalues there shows that in the critical regime the two types of random matrices have different universal behaviours and are usually investigated by different methods (except for \cite{Bloemendal-Virag11}, \cite{Bloemendal-Virag11a}), while in the the supercritical and subcritical regimes, the behaviours of the two types of random matrices have more common features and are usually investigated together, by some perturbative approaches which can often reduce the problems to those of the non-perturbed models. The previous research of eigenvectors in BBP transition, which is only in the supercritical and subcritical regimes, generally works for both types of random matrices and yields similar results for them. Our approach in the critical regime, however, is non-perturbative and  works only for the Hermitian type random matrices, because it depends on the determinantal property that is not available for the real symmetric ones. The study of the eigenvectors in the critical regime of BBP transition for real symmetric type random matrices is more challenging and is out of the scope of the current paper.

\subsection{Extended Airy kernel} \label{s.1.2}

In order to state our main results, we need to first introduce the extended Airy kernel in this subsection. Recall the Airy kernel that defines the celebrated Tracy-Widom distribution that is often seen in random matrix theory and interacting particle systems of the Kardar-Parisi-Zhang (KPZ) universality class, see \cite{Anderson-Guionnet-Zeitouni10}, \cite{Corwin11} and \cite{Quastel12}, and references therein,
\begin{equation} \label{eq:defn_Airy_kernel}
  K_{\Airy}(x, y) = \frac{1}{(2\pi \ii)^2} \int_{\sigma} \dd u \int_{\gamma} \dd v \frac{e^{\frac{u^3}{3} - xu}}{e^{\frac{v^3}{3} - yv}} \frac{1}{u - v},
\end{equation}
where the contours $\gamma$ and $\sigma$ are as in Figure \ref{fig:gamma_L}. They are nonintersecting and infinite contours; $\gamma$ goes from $e^{-2\pi \ii/3} \cdot \infty$ to $e^{2\pi \ii/3} \cdot \infty$, and $\sigma$ goes from $e^{-\pi \ii/3} \cdot \infty$ to $e^{\pi \ii/3} \cdot \infty$. We then define the extended Airy kernel depending on real parameters $a_1, a_2, \dotsc$, which is the correlation kernel of a determinantal point process at discrete time $t \in \intZ_{\geq 0}$. For any $m_1, m_2 \in \intZ_{\geq 0}$, we let
\begin{align}
  K^{m_1, m_2}_{\Airya}(x, y) = {}& -\Id(m_1 < m_2) \Id(x < y) \frac{1}{2\pi \ii} \oint \frac{e^{(y - x)w}}{\prod^{m_2}_{j = m_1 + 1} (w - a_j)} \dd w + \K^{m_1, m_2}_{\Airya}(x, y), \label{eq:form_ext_K} \\
  \shortintertext{where}
  \K^{m_1, m_2}_{\Airya}(x, y)  = {}& \frac{1}{(2\pi \ii)^2} \int_{\sigma} \dd u \int_{\gamma} \dd v \frac{e^{\frac{u^3}{3} - xu}}{e^{\frac{v^3}{3} - yv}} \frac{\prod^{m_1}_{j = 1} (u - a_j)}{\prod^{m_2}_{j = 1} (v - a_j)} \frac{1}{u - v}, \label{eq:form_ext_K_tilde}
\end{align}
such that the contour in \eqref{eq:form_ext_K} encloses all the poles $a_{m_1 + 1}, \dotsc, a_{m_2}$ and in \eqref{eq:form_ext_K_tilde} all the poles $a_1, \dotsc, a_{m_2}$ of $v$ are to the left of $\gamma$. We note that in the special $m_1 = m_2 = 0$ case, the correlation kernel $K^{m_1, m_2}_{\Airya}(x, y)$ is reduced to $K_{\Airy}(x, y)$.

The Airy kernel $K_{\Airy}$ defines a determinantal point process, with infinitely many particles, ordered as
\begin{equation} \label{eq:Airy_proc}
  +\infty > \xi_1 > \xi_2 > \dotsb,
\end{equation}
and the $n$-point correlation function
\begin{equation} \label{eq:n_corr_Airy}
  R_n(x_1, \dotsc, x_n) = \det(K_{\Airy}(x_i, x_j))^n_{i, j = 1}. 
\end{equation}
Analogously, for each $m \geq 0$, the extended Airy kernel $K^{m, m}_{\Airya}$ also defines a determinantal point process, with infinitely many particles, ordered as
\begin{equation} \label{20040401}
  +\infty > \xi^{(m)}_1 > \xi^{(m)}_2 > \dotsb,
\end{equation}
with $\xi^{(0)}_i\equiv \xi_i$ in \eqref{eq:Airy_proc}. Furthermore, if we put all the $\xi^{(m)}_i$'s ($m \geq 0$, $i \geq 1$) together, they form a determinantal point process living in space $\realR$ and time $m \in \intZ_{\geq 0}$. The probability meaning of the extended Airy kernel is that it defines a determinantal point process with infinitely many species of particles, $\xi_i^{(m)}$ (particle index $i = 1, 2, \dotsc$, species index $m = 0, 1, \dotsc$), such that the marginal distribution of $m$-species particles is given by the correlation kernels $K^{m, m}_{\Airya}$, and further the mixed correlation function is given by
\begin{equation} \label{eq: mixed corr}
  \begin{split}
     R_n(x_1, m_1; x_2, m_2; \dotsc; x_n, m_n) 
    := {}& \lim_{\Delta x \to 0} \frac{1}{\Delta x^n} \Prob \left(
      \begin{gathered}
        \text{there exists a particle in $[x_i, x_i + \Delta x)$ at time $m_i$ for $i = 1, \dotsc, n$} 
      \end{gathered}
    \right) \\
    = {}& \det \left( K^{m_i, m_j}_{\Airya}(x_i, x_j) \right)^n_{i, j = 1}.
  \end{split}
\end{equation}

\begin{figure}[htb]
  \begin{minipage}[t]{4cm}
    \centering
    \includegraphics{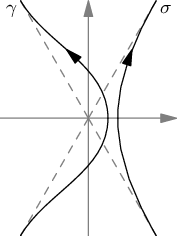}
    \caption{Contours $\gamma$ and $\sigma$.}
    \label{fig:gamma_L}
  \end{minipage}
  \hspace{\stretch{1}}
  \begin{minipage}[t]{4cm}
    \centering
    \includegraphics{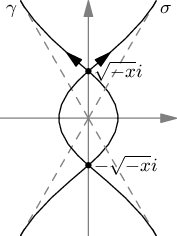}
    \caption{Double contour $\mathsf{X}$ consisting of $\gamma$ and $\sigma$ that are deformed through $\pm\sqrt{-x}i$.}
    \label{fig:X}
  \end{minipage}
  \hspace{\stretch{1}}
  \begin{minipage}[t]{4cm}
    \centering
    \includegraphics{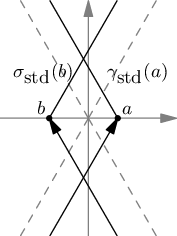}
    \caption{$\gamma_{\standard}(a)$ and $\sigma_{\standard}(b)$ that are standardized deformations of $\gamma$ and $\sigma$.}
    \label{fig:gamma_sigma_standard}
  \end{minipage}
  \hspace{\stretch{1}}
  \begin{minipage}[t]{4cm}
    \centering
    \includegraphics{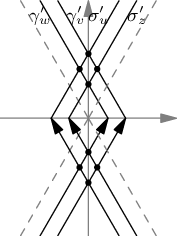}
    \caption{$4$-fold contour $\mathsf{\doubleX}$ consisting of two pairs of $\gamma$ and $\sigma$ in standardized deformation.}
    \label{fig:double_X}
  \end{minipage}
\end{figure}
\begin{rmk}
  It is not easy to check directly that the correlation functions in \eqref{eq: mixed corr} are well-defined, since the kernel functions are generally non-Hermitian. (The necessary and sufficient condition for a Hermitian kernel to define a probabilistic determinantal point process is given in \cite[Theorem 3]{Soshnikov00}.) However, since we know in Lemma \ref{lem:Airy_limit} (see also \cite{Adler-van_Moerbeke-Wang11}) that the extended Airy kernel is the limit of the correlation kernels of the GUE minor process with external source, we conclude thereby that the correlation functions in \eqref{eq: mixed corr} are well-defined. Then it is not hard to see that the rightmost particle exists for each species. Also we have that the point process consisting of finitely many species of particles is simple (by \cite[Remark 4]{Hough-Krishnapur-Peres-Virag06}).
\end{rmk}
 
\subsection{Main results} \label{s.1.3} 

Recall the point process in \eqref{20040401} with any fixed parameter sequence $\totala=(a_1,\ldots, a_k)$. Define for any integers $n>j$ the random variable
\begin{equation} \label{20040410}
  \Xi^{(k)}_j(\totala; n) := n^{\frac{1}{3}} \prod^{j-1}_{i=1} \frac{\xi^{(k)}_j - \xi^{(k-1)}_i}{\xi^{(k)}_j - \xi^{(k)}_i} \prod^n_{i=j+1} \frac{\xi^{(k)}_j - \xi^{(k-1)}_{i-1}}{\xi^{(k)}_j - \xi^{(k)}_i}. 
\end{equation}
Our first result is on the existence of the limit of $ \Xi_{j}^{(k)}(\totala;n)$ as $n\to \infty$. 
\begin{thm} \label{thm:main_thm_1}
  For any $\totala=(a_1, \ldots, a_k)\in \mathbb{R}^k$ with fixed components, and $j$ a fixed positive integer,
  \begin{equation} \label{eq:defn_Xi^infty}
    \Xi_{j}^{(k)}(\totala;\infty):=\lim_{n\to \infty}   \Xi_{j}^{(k)}(\totala;n)
  \end{equation}
  exists almost surely.  
\end{thm}

Our second result is on the distribution of the first $k$ components of the eigenvectors associated with the largest eigenvalues. The theorem states for the first component, and see Remark \ref{rmk:main_thm_1_related} for the $2$-nd, \dots, $k$-th components.
\begin{thm} \label{thm:main_thm_2}
  Under Assumption \ref{main assum}, for any fixed $j$, we have  
  \begin{equation}
    N^{\frac{1}{3}} \left\lvert x_{j1} \right\rvert^2 \weakconv \left( \frac{3\pi}{2} \right)^{\frac{1}{3}} \Xi_{j}^{(k)}(\totala; \infty), \qquad \text{as  }N\to \infty.
  \end{equation}
\end{thm}
\begin{rmk}
  The construction of the  random variable $\Xi^{(k)}_j(\totala; n)$ in (\ref{20040410})  is inspired by the {\it eigenvector-eigenvalue identity} (c.f. (\ref{20041121})) and the weak convergence of the edge eigenvalues of the the spiked GUE (c.f. (\ref{20041120})). This explains heuristically why we shall expect the weak convergence in Theorem \ref{thm:main_thm_2}.  More detailed illustration will be given in Section \ref{subsec:proof_strategy}.
\end{rmk}
\begin{rmk} \label{rmk:main_thm_1_related}
  Since our $\alpha_1, \dotsc, \alpha_k$ are unordered, the first component corresponding to the $\alpha_1 \vec{e}_1 \vec{e}^*_1$ deformation has nothing special compared with the $2$-nd, \dots, $k$-th components, and we state the result for the first component only for notational simplicity. The result of Theorem \ref{thm:main_thm_1} can be adapted for  $x_{jl}$ with any $l = 2, \dotsc, k$ as follows: We consider, instead of $G_{\totalalpha}$, the random matrix $\tilde{G}_{\totalalpha}$ which is a conjugate of $G_{\totalalpha}$ by switching the first row/column and the $l$-th row/column. Then
  \begin{equation}
    x_{jl} = \tilde{x}_{j1},
  \end{equation}
  where $\tilde{x}_{mn}$ means the $n$-th component of the normalized eigenvector of $\tilde{G}_{\totalalpha}$ associated with $\reservesigma_m$ (the $m$-th largest eigenvalue of $G_{\totalalpha}$, which is also the $m$-th largest eigenvalue of $\tilde{G}_{\totalalpha}$).
\end{rmk}
\begin{rmk} We further remark here that the result in Theorem \ref{thm:main_thm_2} holds for any fixed $j$. Particularly, $j$ can be even larger than $k$ (but fixed). Note that $\reservesigma_j$ is not an outlier or a critical spiked eigenvalue in case $j>k$. Hence, the result in Theorem \ref{thm:main_thm_2} shows that a critical spike can even cause a bias of the eigenvectors associated with those non-outliers towards the direction of the spikes, since $x_{j1}$ is typically of order $N^{-1/6}$ here while it would be of order $N^{-1/2}$ in case the deformation $\sum^k_{i = 1} \alpha_i \vec{e}_i \vec{e}^*_i$ was absent, for instance. Such a phenomenon was previously observed in  \cite{Bloemendal-Knowles-Yau-Yin16} when the spike is in the subcritical or supercritical regime, but sufficiently close to the critical regime. As $j \gg k$, this bias eventually peters out, and the decay speed deserves further study.
\end{rmk}
A consequence of Theorem \ref{thm:main_thm_1} and Remark \ref{rmk:main_thm_1_related} is as follows:
\begin{cor} \label{cor:to_thm_2}
  Under Assumption \ref{main assum}, for a fixed $j$, the moduli of components $\lvert x_{j, k + 1} \rvert, \lvert x_{j, k + 2} \rvert, \dotsc, \lvert x_{jN} \rvert$ have the same distribution, and we have, with $l > k$,
  \begin{equation}
    N \lvert x_{j\ell} \rvert^2 \weakconv \frac{1}{2} \chi^2(2) , \qquad \text{as  }N\to \infty, 
  \end{equation} 
  where $\chi^2(2)$ is the $\chi^2$ random variable with parameter $2$, which can be equivalently expressed as $\Gammadist(1,1)$ or $\Exp(1)$.
\end{cor}

Finally, the integrable property of the  distribution of $\Xi_{j}^{(k)}(\totala;\infty)$ will be the subject of further study, and here we only present the first step towards  this direction: the non-degeneracy of the distribution.

\begin{thm}\label{thm.nondegeneracy} Under the assumption of Theorem \ref{thm:main_thm_1}, the  limit $ \Xi_{j}^{(k)}(\totala;\infty)$ is nondegenerate, i.e,  
$$\mathbb{P}( \Xi_{j}^{(k)}(\totala;\infty)=x)<1$$
 for any $x\in \mathbb{R}$. 
\end{thm}

In the end of Section \ref{s.nondegeneracy}, we also present some simulation results, which show some features of the eigenvector distribution numerically.

\subsection{Proof strategy} \label{subsec:proof_strategy}

In this section, we provide a brief description on our proof strategy. We start with Theorem \ref{thm:main_thm_1}. From Lemma \ref{lem:interlacing_Airy}, it will be clear that the two species $\xi_i^{(k)}$'s and $\xi_{i}^{(k-1)}$'s are interlacing. However, in light of the definition of $\Xi_j^{(k)}(\totala;n)$ (c.f. (\ref{20040410})), to  prove Theorem \ref{thm:main_thm_1}, we need to know the interlacing more quantitatively, namely, we shall know more precisely how a particle of the $k-1$ species sits between two neighboring particles of the $k$ species. To this end, we turn to study the logarithm of  $\Xi_{j}^{(k)}(\totala;n)$ in \eqref{20040410}, which can be further written as an integral of $1/(\xi_j^{(k)}-x)$ against a random measure $\mu$,  with the lower limit of the integral given by $\xi_n^{(k)}$; see \eqref{20041001} and \eqref{19100302}. Specifically,  the measure $\mu$ has density $\phi$ taking $1$ on intervals $(\xi_i^{(k)}, \xi_{i-1}^{(k-1)}]$ for all $i$ and $0$ elsewhere. Hence, the measure $\mu$ characterizes the interlacing more quantitatively. It is also clear that $\xi_n^{(k)}\to -\infty$ almost surely as $n\to\infty$. Hence, in order to show that the left tail of the integral in \eqref{19100302} is negligible, i.e., the integral is convergent, it suffices to study the property of the measure $\mu((x, \infty))$ when $x\to -\infty$.  A key technical step in this part is Proposition \ref{prop:Airy_measure}, where we show that the random measure $\mu$ behaves like a halved Lebesgue measure when it acts on the interval $(x, \infty)$ with $x\to -\infty$. Heuristically,  Proposition \ref{prop:Airy_measure} can be interpreted as: when $i$ is large, $\xi_{i}^{(k)}$ is typically sitting neutrally between $\xi_{i}^{(k-1)}$ and $\xi_{i-1}^{(k-1)}$ and does not favor either side. 

Technically, the complementary distribution function, $\mu((x, +\infty))$, can be expressed (approximately) in terms of the difference between two linear statistics: one of the species $\{\xi_i^{(k)}\}$ with the test function $h_x(t)=(-t+x)\mathbf{1}(t>x)$, and the other of the species $\{\xi_i^{(k-1)}\}$ with the same test function; see \eqref{eq:mu_S_N_Airy}. The computation of the linear statistics of determinantal point processes, especially those from random matrix models, is an extensively studied area, and results are abundant. The tricky part in our case is that up to the leading term, the asymptotics of the two linear statistics are the same. We take the advantage of our model that both the mean and variance of the difference of the two linear statistics have an exact formula; see \eqref{eq:int_over_X_and_bar} and \eqref{eq:var_3parts_Airy}. For our purpose, it suffices  to take limits of the mean and variance formulas. This is done in the proof of Proposition \ref{prop:Airy_measure} via a rather delicate saddle point analysis involving a series of contour deformations. We also refer to the recent work \cite{Erdos-Schroder18}  for a study on the fluctuation of the difference of linear eigenvalue statistics of a Wigner matrix and its minor, where two strongly correlated linear statistics have a significant cancellation.  Although the result in \cite{Erdos-Schroder18} is more on the bulk regime of random matrices, while here we are discussing two interlacing species of the extended Airy process,  both works indicate that exploiting the true size of the difference between the linear statistics of two interlacing point processes is normally more delicate than that of a single linear statistic.

For the proof of Theorem \ref{thm:main_thm_2}, we start from the celebrated {\it eigenvector-eigenvalue identity}. We refer the interested readers to the recent survey \cite{Denton-Parke-Tao-Zhang19} and reference therein for a detailed discussion and a history of this identity. Here we cite the identity directly from \cite{Denton-Parke-Tao-Zhang19} with slight modification as the following proposition.
\begin{pro} \label{thm. eigenvector-eigenvalue}
  Let $A\in \mathbb{C}^{n\times n}$ be an Hermitian matrix and let $M_j\in \mathbb{C}^{(n-1)\times (n-1)}$ be its minor obtained by deleting the $j$-th column and row from $A$.  Denote by $\lambda_1(A) > \dotsb > \lambda_n(A)$ the ordered eigenvalues of $A$ and  by  $\lambda_1(M_j) > \dotsb > \lambda_{n-1}(M_j)$ the ordered eigenvalues of $M_j$. Furthermore, let $\vec{u}_i=(u_{i1},\ldots, u_{in})^{\top}$ be the eigenvector of $A$, associated with $\lambda_i(A)$. Then one has 
  \begin{equation}
    \lvert u_{ij} \rvert^2 = \frac{\prod^{n-1}_{\ell=1} (\lambda_i(A) - \lambda_\ell(M_j))}{\prod_{\ell \in \{ 1, \dotsc, n \} \setminus \{ i \}} (\lambda_i(A) - \lambda_\ell(A))}.
  \end{equation}
\end{pro}
Applying the eigenvector-eigenvalue identity to our random matrix $G_{\totalalpha} = G^{(N)}_{\totalalpha}$ and its minor $G^{(N - 1)}_{\totalalpha}$ which is constructed by removing the first column and first row of $G_{\totalalpha}$, we can write 
\begin{equation} \label{eq:alter_expr}
  \lvert x_{j1} \rvert^2 = \prod^{j-1}_{i = 1} \frac{\reservesigma_j - \lambda_i}{\reservesigma_j - \reservesigma_i} \prod^N_{i = j + 1} \frac{\reservesigma_j - \lambda_{i - 1}}{\reservesigma_j - \reservesigma_i},
\end{equation}
where $\reservesigma_1 > \reservesigma_2 > \dotsb > \reservesigma_N$ are eigenvalues of $G_{\totalalpha}$, $\lambda_1 > \dotsb > \lambda_{N - 1}$ are eigenvalues of $G^{(N - 1)}_{\totalalpha}$, and $x_{j1}$ is the first component of the normalized eigenvector of $G_{\totalalpha}$ associated with $\reservesigma_j$. Equivalently, we have
\begin{equation} \label{20041121}
  N^{\frac{1}{3}} \lvert x_{j1} \rvert^2 = N^{\frac{1}{3}} \prod^{j - 1}_{i = 1} \frac{\reservesigma_j - \lambda_i}{\reservesigma_j - \reservesigma_i} \prod^N_{i = j + 1} \frac{\reservesigma_j - \lambda_{i-1}}{\reservesigma_j - \reservesigma_i},
\end{equation}
which resembles formally $\Xi^{(k)}_{j}(\totala; N)$ in \eqref{20040410}. 

First, for the fixed $i$ terms in \eqref{20041121}, we apply the result of GUE minor process with external source in \cite{Adler-van_Moerbeke-Wang11} which shows that the two species of point process $\{\reservesigma_i\}$ and $\{\lambda_i\}$ together form a determinantal point process (see also \cite{Ferrari-Frings12}). We first show in Lemma \ref{lem:Airy_limit} that the edge scaling limit of this point process is given by a determinantal point process with extended Airy kernel defined in \eqref{eq: mixed corr}, under Assumption \ref{main assum}. Particularly, for $\tilde{\sigma}_{i} = N^{1/6}(\sigma_i - 2\sqrt{N})$ and $\tilde{\lambda}_{i} = N^{1/6}(\lambda_i - 2\sqrt{N})$, one can use Lemma \ref{lem:Airy_limit} to show that $(\tilde{\sigma}_{1}, \dotsc, \tilde{\sigma}_{L}, \tilde{\lambda}_{1}, \dotsc, \tilde{\lambda}_{L})$ converges weakly to $(\xi^{(k)}_1, \dotsc, \xi^{(k)}_L, \xi^{(k - 1)}_1, \dotsc, \xi^{(k - 1)}_L)$ for any fixed $L$ as $N \to \infty$.  Consequently, we have 
\begin{align}
\frac{\reservesigma_j-\lambda_i}{\reservesigma_j-\reservesigma_i} \weakconv \frac{\xi_j^{(k)}-\xi^{(k-1)}_i}{\xi_j^{(k)}-\xi_i^{(k)}} \label{20041120}
\end{align}
for any fixed $i\neq j$. This gives an indication of the connection between $N^{1/3} \lvert x_{j1} \rvert^2$ and the limit $\Xi^{(k)}_j(\totala;\infty)$. However, the weak convergence in \eqref{20041120} cannot be applied directly to the double limit case when $i = i(N) \to \infty$ as $N\to \infty$. Hence, besides Lemma \ref{lem:Airy_limit}, we need to analyze the product of the large $i$ terms in \eqref{20041121}. To this end, we mimic the idea of the proof of Theorem \ref{thm:main_thm_1}.  Again, we turn to consider the logarithm of the product in \eqref{20041121} (but over the large $i$ terms only), which can be written as an integral of $1/(\reservesigma_j-x)$ against a random measure $\mu_N$ with the lower limit of the integral given by $\reservesigma_N$; see \eqref{20041101} and \eqref{eq:GUE_P_n_integral}. Specifically, the measure $\mu_N$ has the density $\phi_N$ taking $1$ on $(\reservesigma_i, \lambda_{i-1}]$ for all $i =  2, \dotsc, N$ and $0$ elsewhere. Since the domain of the integral \eqref{eq:GUE_P_n_integral}, i.e., $[\reservesigma_N, \reservesigma_L]$ contains both the edge and bulk regimes of the semicircle law, and the point process $\{\xi^{(k)}_i\}$  only approximates the matrix eigenvalues in the edge regime (with an effective extension to certain order of intermediate regime), we need to further decompose the domain $[\reservesigma_N, \reservesigma_L]$ into two parts: $[\reservesigma_N, \reservesigma_L]= [\reservesigma_N, \reservesigma_{N_0}) \cup [\reservesigma_{N_0}, \reservesigma_L]$, with $N_0:= \lfloor 2/(3\pi)N^{1/10} \rfloor$.  We will then show that under appropriate scaling, the random measure $\mu_N$ can be well approximated by $\mu$ on the domain $[\reservesigma_{N_0}, \reservesigma_L]$. Furthermore, a detailed analysis of the measure $\mu_N$ on  $ [\reservesigma_N, \reservesigma_{N_0}]$ shows that the integral $1/(\reservesigma_j-x)$ over this domain is approximately deterministic, and thus does not contribute to the randomness of the limit of $N^{1/3} \lvert x_{j1} \rvert^2$. 

Especially, our result Theorem \ref{thm:main_thm_2} and its proof show that the randomness of $N^{1/3} \lvert x_{j1} \rvert^2$ essentially  depends only on the local edge regime of the eigenvalues of $G_{\totalalpha}$ and $G_{\totalalpha}^{(N-1)}$,  and the bulk eigenvalues  contribute a deterministic factor. Similar phenomenon has also shown up in the limiting theorem of some other eigenvalue statistics  in the literature, which inspires our current work; for instance, see \cite{Landon-Sosoe19},  \cite{Wu19} for some related discussions. 

For the proof of the nondegeneracy of the distribution of $ \Xi_{j}^{(k)}(\totala;\infty)$, i.e.,  Theorem \ref{thm.nondegeneracy}, we will apply the weak convergence in Theorem \ref{thm:main_thm_2} to translate the question to $N^{1/3}|x_{j1}|^2$. An advantage of the latter is that the eigenvalue distribution admits a log-gas representation which can facilitate our analysis. More specifically, we will show that a bounded truncation (from both below and above) of $\log (N^{1/3}|x_{j1}|^2)$ has a lower bound for variance, uniformly in $N$. This will finally lead to the conclusion of Theorem \ref{thm.nondegeneracy}.

Finally, we remark here that in this paper we do not consider the deformed GOE since the corresponding results on the minor process used for GUE here is not available for GOE so far. But many discussions in this paper work for the deformed GOE as well.

 \subsection{Organization and notation}
 
The rest of the paper is organized as follows.  In Section \ref{s.pre}, we state some preliminary results which will be used in the later sections. The main results of the paper, Theorems \ref{thm:main_thm_1} and \ref{thm:main_thm_2} are proved respectively in Sections \ref{s.airy limit} and \ref{s. GUE limit}, based on the key technical estimates in Propositions \ref{prop:Airy_measure} and \ref{prop:large_small_deviation_mu^k_N}, whose proofs will be stated respectively in Sections \ref{pf.lem_Airy}  and \ref{pf.lem_GUE}.  In addition, Corollary \ref{cor:to_thm_2} is an easy consequence of Theorem \ref{thm:main_thm_2} and Remark \ref{rmk:main_thm_1_related}, and its proof will be stated at the end of Section \ref{s. GUE limit}. Section \ref{s.nondegeneracy} is then devoted to the proof of Theorem \ref{thm.nondegeneracy}. Finally, some proofs of technical lemmas are collected in Appendix \ref{sec:proofs_in_sec_2}.

Throughout the paper, we will use $\bigO(\cdot)$ and $o(\cdot)$ for the standard big-O and little-o notations. We use $C, C', C_1 $, etc. to denote positive constants (independent of $N$). For any set $\mathcal{I}$, the notation $|\mathcal{I}|$ stands for its cardinality. We use the shorthand notation $\llbracket a, b \rrbracket = [a, b] \cap \intZ$, for all $a, b \in \realR$.
 
\section{Collection of results for spiked GUE  and extended Airy process} \label{s.pre}

In this section, we first review the determinantal point process representation of the eigenvalue distribution of $G_{\totalalpha}$ and its minors, and then state some preliminary results to be used in the later sections. Some lemmas in this section are direct consequences of existing results in the literature, while for others, only the proof strategy exists in the literature. We give the proofs and/or the references of the lemmas in Appendix \ref{sec:proofs_in_sec_2}.
 
\subsection{Determinantal point process representation for eigenvalues of $G_{\totalalpha}$}

In this subsection, we re-denote the eigenvalues of $G_{\totalalpha} = G^{(N)}_{\totalalpha}$ and $G_{\totalalpha}^{(N-1)}$ by 
\begin{equation}
  \lambda_i^{(N)}:= \reservesigma_i,\quad \lambda_\ell^{(N-1)}:=\lambda_\ell, \quad \text{where} \quad i \in \llbracket 1, N \rrbracket, \quad \ell \in \llbracket 1, N - 1 \rrbracket 
\end{equation}
and further we denote by $\lambda_1^{(N-j)}> \dotsb > \lambda_{N-j}^{(N-j)}$ the ordered eigenvalues of the matrix $G_{\totalalpha}^{(N-j)}$, which is obtained from $G_{\totalalpha}$ by deleting its first $j$  rows and  columns. 

For each $j$, the eigenvalues of $G^{(N - j)}_{\totalalpha}$ form a determinantal point process. In addition, the eigenvalues of all $G^{(N - j)}_{\totalalpha}$, $j=0, \ldots, k$ together form a determinantal point process, called the \emph{GUE minor process with external source} \cite{Adler-van_Moerbeke-Wang11}. The correlation kernel of these eigenvalues is given as follows: For $j_1, j_2 \in \llbracket 0\ldots, k \rrbracket$, 
\begin{align}
  K^{j_1, j_2}_{\GUEalpha}(x, y) = {}& -\Id(j_1 > j_2) \Id(x < y)\frac{1}{2\pi \ii} \int_{\Gamma} \frac{e^{(y - x)w}}{\prod^{j_1}_{i = j_2 + 1} (w - \alpha_j)} \dd w + \K^{j_1, j_2}_{\GUEalpha}(x, y), \label{eq:kernel_GUE} \\
  \shortintertext{where}
  \K^{j_1, j_2}_{\GUEalpha}(x, y)  = {}& \frac{1}{(2\pi \ii)^2} \int_{\Sigma} \dd z \int_{\Gamma} \dd w \frac{e^{\frac{z^2}{2} - xz} z^{N - k}}{e^{\frac{w^2}{2} - yw} w^{N - k}} \frac{\prod^k_{i = j_1 + 1} (z - \alpha_i)}{\prod^k_{i = j_2 + 1} (w - \alpha_i)} \frac{1}{z - w}. \label{20041005}
\end{align}
Here $\Gamma$ is a circular contour wrapping $0$ and $\alpha_1, \dotsc, \alpha_k$ positively, and $\Sigma$ is an infinite contour from $-\ii \cdot \infty$ to $\ii \cdot \infty$, to the right of $\Gamma$; see Figure \ref{fig:Gamma_Sigma}.

The relation between the correlation kernel $K^{j_1, j_2}_{\GUEalpha}(x, y)$ and the correlation kernel $K^{k_1, k_2}_{\Airya}(x, y)$ defined in (\ref{eq:form_ext_K}) is as follows:
\begin{lem} \label{lem:Airy_limit}
  Let $\alpha_1, \dotsc, \alpha_k$ be given in Assumption \ref{main assum}. If we denote, with $j_i = k - k_i \in \llbracket 0, k \rrbracket$ ($i = 1, 2$),
  \begin{equation}
    K^{j_1, j_2}_{N, \scaled}(x, y) = N^{-\frac{1}{6}} N^{\frac{j_1 - j_2}{6}} e^{N^{\frac{1}{3}}(x - y)} K^{j_1, j_2}_{\GUEalpha}(2\sqrt{N} + N^{-\frac{1}{6}}x, 2\sqrt{N} + N^{-\frac{1}{6}}y),
  \end{equation}
  then
  \begin{enumerate}
  \item
    For all $x, y$ in a compact subset of $\realR$, uniformly 
    \begin{equation}
      \lim_{N \to \infty} K^{j_1, j_2}_{N, \scaled}(x, y) = K^{k_1, k_2}_{\Airya}(x, y).
    \end{equation}
  \item
    Let $\epsilon > \max(a_1, \dotsc, a_k)$. For any $C \in \realR$, as operators on $L^2([C, +\infty))$, the ones with kernel $e^{\epsilon(x - y)} K^{j_1, j_2}_{N, \scaled}(x, y)$ converge, in trace norm, to the one with kernel $e^{\epsilon(x - y)} K^{k_1, k_2}_{\Airya}(x, y)$. 
  \end{enumerate}
 Consequently, the above statements imply the weak convergence of the joint distribution of $\{N^{1/6}(\lambda^{(N - j)}_i - 2\sqrt{N}) \mid 0 \leq j \leq k, \, 1 \leq i \leq K \}$ to that of $\{\xi^{(k - j)}_i \mid 0 \leq j \leq k, \, 1 \leq i \leq K \}$ for any fixed positive integer $K$, where $\xi^{(k - j)}_i$'s are particles in the determinantal point process given by \eqref{eq: mixed corr}.
\end{lem}
We note that although we discussed the joint distribution of $\lambda^{(N - j)}_i$ for all $j\in\llbracket 0, k \rrbracket$, to prove Theorem \ref{thm:main_thm_2},  we only need the $j = 0, 1$ cases. Hence we recycle the notations  $\reservesigma_i$ and  $\lambda_i$ for $\lambda^{(N)}_i$ and $\lambda^{(N-1)}_i$ respectively in the sequel, for simplicity.

\subsection{Useful estimates}

By the well-known Weyl's inequality, we have that
\begin{equation} \label{eq:interlacing_GUE}
  \reservesigma_1 > \lambda_1 > \reservesigma_2 > \lambda_2 > \dotsb > \lambda_{N - 1} > \reservesigma_N, \quad \text{and more generally} \quad \lambda^{(N - j)}_1 > \lambda^{(N - j - 1)}_1 > \lambda^{(N - j)}_2 > \dotsb > \lambda^{(N - j)}_{N-j}
\end{equation}
for all $j = 0, 1, \dots, N-1$ \footnote{ Hereafter we ignore the probability $0$ event that some $\lambda^{(N - j - 1)}_i$ is identical to $\lambda^{(N - j)}_i$ or $\lambda^{(N - j)}_{i + 1}$.}. In parallel, for the particles in the extended Airy process defined in (\ref{eq: mixed corr}), we  also have the following interlacing property. 
\begin{lem} \label{lem:interlacing_Airy}
  Let $j = 1, \dotsc, k$. With probability $1$, the particles $\xi^{(j)}_i$ and $\xi^{(j - 1)}_i$ ($i = 1, 2, \dotsc$) whose joint distribution is given by \eqref{eq: mixed corr} satisfy the interlacing inequality
  \begin{equation} \label{eq:interlacing_Airy}
    +\infty > \xi^{(j)}_1 > \xi^{(j - 1)}_1 > \xi^{(j)}_2 > \xi^{(j - 1)}_2 > \dotsb.
  \end{equation}
\end{lem}

Next, we list some rigidity result for the particles in the extended Airy process and also the eigenvalues of $G_{\totalalpha}$.
\begin{lem} \label{lem:non-asy_est_Airy}
  Let $\xi_j^{(k)}$, $j=1,\dotsc$ be the particles defined in \eqref{eq: mixed corr}. Then we have the following estimates. 
  \begin{enumerate}
  \item \label{enu:lem:non-asy_est_Airy:1} 
    For any fixed $j$, there exists a numerical $C > 0$, such that for all $t > 0$
    \begin{equation}
      \Prob(\xi^{(k)}_j > t) < C e^{-t/C}, \quad \Prob(\xi^{(k)}_j < -t) < C e^{-t/C}.
    \end{equation}
  \item \label{enu:lem:non-asy_est_Airy:2}
    For all $n \geq 2$, there exists a numerical $c > 0$, such that
    \begin{equation} \label{eq:ineq:xi_fluc}
      \Prob \left( \left\lvert \xi^{(k)}_n + \left( \frac{3\pi n}{2} \right)^{2/3} \right\rvert > n^{\frac{3}{5}} \right) \leq c n^{-\frac{6}{5}} \log n.
  \end{equation}
  \end{enumerate}
\end{lem}

\begin{lem} \label{lem:non-asy_est_GUE}
  Suppose $N > k$ is large enough, then:
  \begin{enumerate}
  \item \label{enu:lem:non-asy_est_GUE:1}
    For $j \in  \llbracket1, k \rrbracket$, there exists a numerical $C > 0$, such that
    \begin{align}
      \Prob(\reservesigma_j > 2\sqrt{N} + t N^{-\frac{1}{6}}) < {}& C e^{-t/C}, & \Prob(\reservesigma_N < -2\sqrt{N} - t N^{-\frac{1}{6}}) < {}& C e^{-t/C}, && \text{for all $t > 0$}, \label{eq:largest_right_tail} \\
      \Prob(\reservesigma_j < 2\sqrt{N} - t N^{-\frac{1}{6}}) < {}& C e^{-t/C}, & \Prob(\reservesigma_N > -2\sqrt{N} + t N^{-\frac{1}{6}}) < {}& C e^{-t/C}, && \text{for all $2 \leq t \leq 2N^{2/3}$}. \label{eq:largest_left_tail}
    \end{align}
  \item \label{enu:lem:non-asy_est_GUE:2}
    For any positive constant $C>0$, there exists a numerical $c > 0$, such that for all $2 \leq n \leq CN^{1/10}$,
    \begin{equation} \label{eq:ineq:xi_fluc_GUE}
      \Prob \left( \left\lvert N^{\frac{1}{6}} (\reservesigma_n - 2\sqrt{N}) + \left( \frac{3\pi n}{2} \right)^{2/3} \right\rvert > n^{\frac{3}{5}} \right) \leq c n^{-\frac{6}{5}} \log n.
    \end{equation}
  \item \label{enu:lem:non-asy_est_GUE:3}
    Given any (small) $c \in (0, 1)$, for all $i\in \llbracket 1, (1 - c)N \rrbracket $,  for any (small) constant $\varepsilon>0$ and (large) constant $D>0$ 
    \begin{equation} \label{eq:rigidity_Yau}
      \Prob \left( \lvert \reservesigma_i - \Upsilon_i \rvert \geq N^{-\frac{1}{6} + \varepsilon} i^{-\frac{1}{3}} \right) \leq N^{-D}.
    \end{equation}
    Here $\Upsilon_i$ is the scaled quantile of semicircle law defined by 
    \begin{equation} \label{eq:defn_Upsilon}
      \int^{2}_{\Upsilon_i/\sqrt{N}} \rho_{sc}(x) \dd  x= \frac{N-i+\frac12}{N}, \quad \rho_{sc}(x)=\frac{1}{2\pi}\sqrt{(4-x^2)_+}. 
    \end{equation}
  \end{enumerate}
\end{lem}

All the lemmas stated in this section are proved in Appendix \ref{sec:proofs_in_sec_2}.

\section{Existence of the limit} \label{s.airy limit}

In this section, we prove Theorem \ref{thm:main_thm_1}. Recall \eqref{20040410}. For any $1 \leq j \leq k < m < n$, we set the quantity 
\begin{equation} \label{eq:tilde_P_Airya}
  \Xi_{j}^{(k)}(\totala; n) = n^{\frac{1}{3}} \prod_{i=1}^{j-1}\frac{\xi^{(k)}_j - \xi^{(k - 1)}_i}{\xi^{(k)}_j - \xi^{(k)}_i} \prod^m_{i = j + 1} \frac{\xi^{(k)}_j - \xi^{(k - 1)}_{i - 1}}{\xi^{(k)}_j - \xi^{(k)}_i} \mathcal{A}^{(k)}_{j, m}(\totala; n), \quad \text{where} \quad \mathcal{A}^{(k)}_{j, m}(\totala; n) := \prod^n_{i = m + 1} \frac{\xi^{(k)}_j - \xi^{(k - 1)}_{i - 1}}{\xi^{(k)}_j - \xi^{(k)}_i}.
\end{equation}
In order to prove Theorem \ref{thm:main_thm_1}, it suffices to show that for any  given $m > k$, $\log \mathcal{A}^{(k)}_{j, m}(\totala; n)+\frac{1}{3} \log n$ converges almost surely as $n\to \infty$, or equivalently, by Cauchy's criterion,
\begin{equation} \label{eq:technical_main_thm_1}
  \limsup_{m \to \infty} \sup_{n > m}  \left\lvert \log \mathcal{A}^{(k)}_{j, m}(\totala; n) + \frac{1}{3} \log \frac{n}{m} \right\rvert = 0, \quad \as. 
\end{equation}
In order to study 
\begin{equation} \label{eq:log_A_expr}
  \log \mathcal{A}_{j,m}^{(k)}(\totala;n)  = \sum^n_{i = m + 1}\Big( \log(\xi^{(k)}_j - \xi^{(k - 1)}_{i - 1}) - \log(\xi^{(k)}_j - \xi^{(k)}_i)\Big), 
\end{equation}
we define a random measure $\mu = \mu^{(k)}$ on $\realR$ that is absolutely continuous with respect to the Lebesgue measure, and is given by a random density function $\phi(x)$ 
\begin{equation} \label{20041001}
  \phi(x) :=
  \begin{cases}
    1, & \xi^{(k)}_i < x \leq \xi^{(k - 1)}_{i - 1} \text{ for all } i > 1, \\ 
    0, & \text{otherwise}.
  \end{cases}
\end{equation}
As we mentioned earlier in Section \ref{subsec:proof_strategy}, the measure $\mu$ describes the interlacing of two species $\xi^{(k)}_i$'s and $\xi^{(k-1)}_i$'s quantitatively. Since $\xi^{(k)}_1 < +\infty$ almost surely, we have that
\begin{equation} \label{eq:defn_M(x)}
  M(x) := \mu((x, +\infty)) = \int^{\infty}_x \phi(t) \dd t
\end{equation}
is a well defined random, continuous and piecewise linear function of $x \in \realR$. 
Then we write
\begin{equation} \label{19100302}
  \log \mathcal{A}_{j,m}^{(k)}(\totala;n) = \int^{\xi^{(k - 1)}_m}_{\xi^{(k)}_n} \frac{-1}{\xi^{(k)}_j - x} \dd \mu(x) = \int^{\xi^{(k - 1)}_m}_{\xi^{(k)}_n} \frac{1}{\xi^{(k)}_j - x} \dd M(x)=\int^{\xi^{(k)}_m}_{\xi^{(k)}_n} \frac{1}{\xi^{(k)}_j - x} \dd M(x),
\end{equation}
where in the last step we used the fact that $\dd M(x)=0$ on $(\xi^{(k - 1)}_m, \xi^{(k)}_m]$ by definition. 
We also define the deterministic function
\begin{equation} \label{eq:defn_F(x)}
  F(x) := \E M(x).
\end{equation}
It is clear that as $x \to +\infty$, $M(x) \to 0$ almost surely, and $F(x) \to 0$. Furthermore, since it is clear that $\xi^{(k)}_n\to -\infty$ almost surely as $n\to \infty$ (c.f. (\ref{eq:ineq:xi_fluc})), in light of  (\ref{19100302}), the key to show the convergence of $\log \mathcal{A}_{j,m}^{(k)}(\totala;n)$ is the asymptotic behaviour of the measure $M(x)$ when $x\to -\infty$. We have the following estimate of $F(x)$ and $M(x)$, whose proof will be given in Section \ref{pf.lem_Airy}.
\begin{pro} \label{prop:Airy_measure}
  As the negative parameter $x \to -\infty$, we have
  \begin{equation} \label{eq:mean_variance_measure_Airy}
    F(x) = \frac{-x}{2} + \bigO(\lvert x \rvert^{\frac{1}{2}}), \quad \var M(x) = \E M^2(x) - F^2(x) = \frac{2}{\pi} \sqrt{-x} + \bigO(\lvert x \rvert^{\frac{1}{4}}).
  \end{equation}
\end{pro}

Note that \eqref{eq:technical_main_thm_1} is equivalent to
\begin{equation} \label{080701}
  \lim_{m \to \infty} \sup_{n_2 > n_1 > m} D_{n_1, n_2} = 0, \quad \as, \quad \text{where} \quad D_{n_1, n_2} = \left\lvert \int^{\xi^{(k)}_{n_1}}_{\xi^{(k)}_{n_2}} \frac{1}{\xi^{(k)}_j - x} \dd M(x) + \frac{1}{3} \log \frac{n_2}{n_1} \right\rvert.
\end{equation}
Or else, we can use $F(x)$ to rewrite \eqref{080701} as
\begin{equation}
  \lim_{m \to \infty} \sup_{n_2 > n_1 > m} \tilde{D}_{n_1, n_2} = 0, \quad \as, \quad \text{where} \quad \tilde{D}_{n_1, n_2} = \left\lvert \int^{\xi^{(k)}_{n_1}}_{\xi^{(k)}_{n_2}} \frac{1}{\xi^{(k)}_j - x} \dd M(x) - \int^{-\left( 3\pi n_1/2 \right)^{2/3}}_{-\left( 3\pi n_2/2 \right)^{2/3}} \frac{1}{0 - x} \dd F(x) \right\rvert,
\end{equation}
in light of the first identity in \eqref{eq:mean_variance_measure_Airy}. It is equivalent to show: for any $\epsilon_1, \epsilon_2 > 0$, there exists $m_{\epsilon_1, \epsilon_2}$, such that for all $m > m_{\epsilon_1, \epsilon_2}$, $\Prob \left( \sup_{n_2 > n_1 > m} \tilde{D}_{n_1, n_2} > \epsilon_1 \right) < \epsilon_2$. Furthermore, we see that it suffices to show, under the same assumption,
\begin{equation} \label{eq:alt_technical_main_thm_1}
  \Prob \left( \sup_{n > m} \tilde{D}_{m, n} > \epsilon_1 \right) = \Prob \left( \sup_{n > m} \left\lvert \int^{\xi^{(k)}_m}_{\xi^{(k)}_n} \frac{1}{\xi^{(k)}_j - x} \dd M(x) - \int^{-\left( 3\pi m/2 \right)^{2/3}}_{-\left( 3\pi n/2 \right)^{2/3}} \frac{1}{0 - x} \dd F(x) \right\rvert > \epsilon_1 \right) < \epsilon_2.
\end{equation}
Here we note that by Lemma \ref{lem:non-asy_est_Airy}, the constants $0$, $-(3\pi m/2)^{2/3}$ and $-(3\pi n/2)^{2/3}$ approximate the values of $\xi^{(k)}_j$, $\xi^{(k)}_m$ and $\xi^{(k)}_n$ respectively.

Therefore, in order to prove Theorem \ref{thm:main_thm_1}, it suffices to show  \eqref{eq:alt_technical_main_thm_1} in the sequel.
\begin{rmk} \label{rmk:Airy_ingredients}
  Here we remark that the proof of \eqref{eq:alt_technical_main_thm_1} relies only on the following three ingredients: 
  \begin{enumerate*}[label=(\roman*)]
  \item 
    Proposition \ref{prop:Airy_measure} on the properties of $F(x)$ and $M(x)$,
  \item
    Lemma \ref{lem:non-asy_est_Airy} on the fluctuation (rigidity) of $\xi^{(k)}_j$ and $\xi^{(k)}_n$, and
  \item
    the property that both $\dd F(x)$ and $\dd M(x)$ are dominated by the Lebesgue measure.
  \end{enumerate*}
 Our proof below could be potentially simplified. Nevertheless, for coherence,  we keep the current presentation, since it can be easily adapted in the later proof of \eqref{eq:transformed_to_G_jL_alt2}.
\end{rmk}

\begin{proof}[Proof of \eqref{eq:alt_technical_main_thm_1}]
  We carry out the proof in three steps. To facilitate the proof, for $j\in \llbracket 1, k\rrbracket$ and $m > k$, we denote by $\Omega^{(k)}_{j,m}$ the events that $\xi^{(k)}_m$ and $\xi^{(k)}_j $ satisfy the following rigidity properties
  \begin{equation} \label{022102}
    \Omega^{(k)}_{j, m} := \left\{ \omega \; \middle\vert \; \lvert \xi^{(k)}_j \rvert \leq  m^{\frac{3}{5}}, \; \left\lvert \xi^{(k)}_m + \left( 3\pi m/2 \right)^{\frac{2}{3}} \right\rvert \leq m^{\frac{3}{5}} \right \},
  \end{equation}
  and also denote $\Omega^{(k)}_{j, m, n} = \Omega^{(k)}_{j, m} \cap \Omega^{(k)}_{j, n}$.
  \begin{enumerate}
  \item
    We first note that given any $\epsilon_1 > 0$, there exists $m_{\epsilon_1}$ such that for all $n > m > m_{\epsilon_1}$,
    \begin{equation} \label{eq:statement_1_proof_Airy}
      \begin{split}
        &  \left\lvert \int^{\xi^{(k)}_m}_{\xi^{(k)}_n} \frac{1}{\xi^{(k)}_j - x} \dd F(x) - \int^{-\left( 3\pi m/2 \right)^{2/3}}_{-\left( 3\pi n/2 \right)^{2/3}} \frac{1}{0 - x} \dd F(x) \right\rvert \Id(\Omega^{(k)}_{j,m,n}) \\
        \leq {}& \left\lvert \int^{\xi^{(k)}_m}_{\xi^{(k)}_n} \frac{1}{\xi^{(k)}_j - x} \dd F(x) - \int^{\xi^{(k)}_m}_{\xi^{(k)}_n} \frac{1}{0 - x} \dd F(x) \right\lvert \Id(\Omega^{(k)}_{j,m,n}) \\
        & + \left\lvert \int^{\xi^{(k)}_m}_{\xi^{(k)}_n} \frac{1}{0 - x} \dd F(x) - \int^{-\left( 3\pi m/2 \right)^{2/3}}_{-\left( 3\pi n/2 \right)^{2/3}} \frac{1}{0 - x} \dd F(x) \right\rvert \Id(\Omega^{(k)}_{j,m,n}) \\
        < {}& \frac{\epsilon_1}{3} + \frac{\epsilon_1}{3} = \frac{2\epsilon_1}{3}.
      \end{split}
    \end{equation}
  \item
    We next show that given any $\epsilon_1, \epsilon_2 > 0$, there exists sufficiently large positive integer $m'_{\epsilon_1, \epsilon_2}$ such that for all $m > m'_{\epsilon_1, \epsilon_2}$ 
    \begin{equation} \label{19101501}
      \Prob \left( \sup_{n > m} \left\lvert \int^{\xi^{(k)}_m}_{\xi^{(k)}_n} \frac{1}{\xi_j^{(k)} - x} \dd (M(x) - F(x)) \Id(\Omega^{(k)}_{j,m,n}) \right\rvert > \frac{\epsilon_1}{3} \right) < \frac{\epsilon_2}{2}.
    \end{equation}
    To see it, using integration by parts, we write
    \begin{equation} \label{eq:IbP}
      \int^{\xi^{(k)}_m}_{\xi^{(k)}_n} \frac{1}{\xi_j^{(k)} - x} \dd (M(x) - F(x)) = \frac{M(\xi^{(k)}_m) - F(\xi^{(k)}_m)}{\xi^{(k)}_j - \xi^{(k)}_m} - \frac{M(\xi^{(k)}_n) - F(\xi^{(k)}_n)}{\xi^{(k)}_j - \xi^{(k)}_n} - \int^{\xi^{(k)}_m}_{\xi^{(k)}_n} \frac{M(x) - F(x)}{(\xi^{(k)}_j - x)^2} \dd x.
    \end{equation}
    So to prove \eqref{19101501}, we only need to show that for a large enough $m'_{\epsilon_1, \epsilon_2}$, if $m > m'_{\epsilon_1, \epsilon_2}$,
    \begin{align}
      \Prob \left( \left\lvert \frac{M(\xi^{(k)}_m) - F(\xi^{(k)}_m)}{\xi^{(k)}_j - \xi^{(k)}_m} \Id(\Omega^{(k)}_{j,m}) \right\rvert > \frac{\epsilon_1}{9} \right) < {}& \frac{\epsilon_2}{6}, \label{eq:BC_app_1} \\
       \Prob \left( \sup_{n > m} \left\lvert \int^{\xi^{(k)}_m}_{\xi^{(k)}_n} \frac{M(x) - F(x)}{(\xi^{(k)}_j - x)^2} \dd x \Id(\Omega^{(k)}_{j,m,n}) \right\rvert > \frac{\epsilon_1}{9} \right) < {}& \frac{\epsilon_2}{6}. \label{eq:BC_app_2}
    \end{align}
    \begin{enumerate}
    \item 
      For \eqref{eq:BC_app_1}, we note that since $(\xi^{(k)}_j - \xi^{(k)}_m)/(3\pi m/2)^{2/3} = 1 + \bigO(m^{-1/15})$ on $\Omega^{(k)}_{j, m}$, it suffices to show
      \begin{equation} \label{eq:BC_app_1_simple}
        \Prob \left( \frac{\lvert M(\xi^{(k)}_m) - F(\xi^{(k)}_m) \rvert}{(3 \pi m/2)^{2/3}} \Id(\Omega^{(k)}_{j, m}) > \frac{\epsilon_1}{9} \right) < \frac{\epsilon_2}{6}.
      \end{equation}
      To prove \eqref{eq:BC_app_1_simple}, we first set
      \begin{equation} \label{eq:defn_ell}
        \ell_i = \left\lceil \frac{2}{3\pi} i^{3/2} \right\rceil.
      \end{equation}
      Then for any $m \in \llbracket \ell_i, \ell_{i + 1} \rrbracket$, $\lvert M(\xi^{(k)}_m) - F(\xi^{(k)}_m) \rvert \leq A_i + B_i(m) + C_i(m)$, where
      \begin{equation} 
        A_i = \lvert M(-i) - F(-i) \rvert, \quad B_i(m) = \lvert M(\xi^{(k)}_m) - M(-i) \rvert, \quad C_i(m) = \lvert F(\xi^{(k)}_m) - F(-i) \rvert.
      \end{equation}
      By the estimate in Proposition \ref{prop:Airy_measure} and Markov's inequality, we have that if $i$ is big enough, then for any $\epsilon > 0$
      \begin{equation} \label{eq:appl_Markov_1}
        \Prob \left( \lvert M(-i) - F(-i) \rvert \geq \epsilon i \right) < \epsilon^{-2} i^{-3/2},
      \end{equation}
      and we conclude that there exists a sufficiently large positive integer $N_{\epsilon_1,\epsilon_2}$ such that
      \begin{equation} \label{eq:est_A_alone}
        \Prob \left( \sup_{i > N_{\epsilon_1,\epsilon_2}} \frac{A_i}{(3\pi \ell_i/2)^{2/3}} > \frac{\epsilon_1}{27} \right) < \frac{\epsilon_2}{6}.
      \end{equation}
      Using the properties that $\lvert F(x_1) - F(x_2) \rvert \leq \lvert x_1 - x_2 \rvert$ and $\lvert M(x_1) - M(x_2) \rvert \leq \lvert x_1 - x_2 \rvert$, we have that there exists $N'_{\epsilon_1, \epsilon_2}$ such that if $i > N'_{\epsilon_1, \epsilon_2}$, then for all $m \in \llbracket \ell_i, \ell_{i + 1} \rrbracket$, we have that on $\Omega^{(k)}_{j, m}$,
      \begin{equation} \label{eq:est_B_and_C}
        \frac{\lvert \star_i(m) \rvert}{(3\pi m/2)^{2/3}} \leq \frac{\lvert \xi^{(k)}_m - (-i) \rvert}{(3\pi m/2)^{2/3}} \leq \frac{\lvert [-(3\pi \ell_{i + 1}/2)^{2/3} - \ell^{3/5}_{i + 1}] - [-(3\pi \ell_i/2)^{2/3}] \rvert}{(3\pi m/2)^{2/3}} < \frac{\epsilon_1}{27}, \quad \text{$\star = B$ or $C$}.
      \end{equation}
      Hence we have that if $i > \max(N_{\epsilon_1, \epsilon_2}, N'_{\epsilon_1, \epsilon_2})$ and $m \in \llbracket \ell_i, \ell_{i + 1} \rrbracket$, then by \eqref{eq:est_B_and_C}
      \begin{equation}
        \frac{\lvert M(\xi^{(k)}_m) - F(\xi^{(k)}_m) \rvert}{(3\pi m/2)^{2/3}} \Id(\Omega^{(k)}_{j, m}) \leq \frac{A_i}{(3\pi \ell_i/2)^{2/3}} + \frac{2}{27}\epsilon_1.
      \end{equation}
      Then by \eqref{eq:est_A_alone}, we prove \eqref{eq:BC_app_1_simple}, which further implies \eqref{eq:BC_app_1}.
    \item 
      To show \eqref{eq:BC_app_2} holds, we note that since on $\Omega^{(k)}_{j, m, n}$, $|(\xi^{(k)}_j - x)/x| = 1 + \bigO(m^{-1/15})$ if $x \geq \xi^{(k)}_m$, it suffices to show that if $m$ is large enough,
      \begin{equation} \label{eq:eq:BC_app_2_middle}
        \Prob \left( \sup_{n > m} \int^{\xi^{(k)}_m}_{\xi^{(k)}_n} \frac{\lvert M(x) - F(x) \rvert}{x^2} \dd x \Id(\Omega^{(k)}_{j,m,n})  > \frac{\epsilon_1}{9} \right) < \frac{\epsilon_2}{6}.
      \end{equation}
      Also, since $\xi^{(k)}_m$ goes to $-\infty$ monotonically as $m \to \infty$, $\xi^{(k)}_m < -(\frac{3}{2}\pi m)^{2/3} + m^{3/5}$ on $\Omega^{(k)}_{j, m}$, and the integrand of \eqref{eq:eq:BC_app_2_middle} is non-negative, we find that it suffices to prove that there exists $K_{\epsilon_1, \epsilon_2}$, such that for all positive integers $k_2 > k_1 > K_{\epsilon_1, \epsilon_2}$,
      \begin{equation} \label{eq:part_3_as_Airy_simplify}
        \Prob \left( \sup_{k_2 > k_1}   \int^{-k_1}_{-k_2} \frac{\lvert M(x) - F(x) \rvert}{x^2} \dd x \geq \frac{\epsilon_1}{9} \right) < \frac{\epsilon_2}{6}.
      \end{equation}
      We define the auxiliary functions with integer parameters $k_1 < k_2$
      \begin{equation} \label{eq:defn_ABC_k_1k_1}
        \begin{gathered}
          A_{k_1, k_2} = \sum^{k_2}_{i = k_1} \frac{\lvert M(-i) - F(-i) \rvert}{(i - 1)i}, \quad B_{k_1, k_2} = \sum^{k_2}_{i = k_1} \frac{\lvert M(-i) - F(-i) \rvert}{i(i + 1)}, \\
          C_{k_1, k_2} = 2 \sum^{k_2 - 1}_{i = k_1 + 1} \frac{F(-i)}{(i - 1)i(i + 1)} - \frac{F(-k_1)}{k_1(k_1 + 1)} + \frac{F(-k_2)}{(k_2 - 1)k_2}.
        \end{gathered}
      \end{equation}
      Since for $x \in [-i, -(i - 1)]$, we have, by the monotonicity of $F(x)$ and $M(x)$,
      \begin{equation}
        M(-(i - 1)) - F(-i) \leq M(x) - F(x) \leq M(-i) - F(-(i - 1)),
      \end{equation}
      and the inequality
      \begin{equation} \label{eq:ineq:ABC}
        \begin{split}
          \int^{-k_1}_{-k_2} \frac{\lvert M(x) - F(x) \rvert}{ x^2} \dd x \leq {}& \sum^{k_2}_{i = k_1 + 1} \int^{-(i - 1)}_{-i} \frac{\lvert M(-i) - F(-(i - 1)) \rvert}{x^2} + \frac{\lvert M(-(i - 1)) - F(-i) \rvert}{x^2} \dd x \\
          = {}& \sum^{k_2}_{i = k_1 + 1} \frac{\lvert M(-i) - F(-(i - 1)) \rvert}{(i - 1)i} + \sum^{k_2 - 1}_{i = k_1} \frac{\lvert M(-i) - F(-(i + 1)) \rvert}{i(i + 1)} \\
          \leq {}& (A_{k_1 + 1, k_2} + C_{k_1, k_2}) + (B_{k_1, k_2 - 1} + C_{k_1, k_2}) \\
          \leq {}& A_{k_1, k_2} + B_{k_1, k_2} + 2C_{k_1, k_2}.
        \end{split}
      \end{equation}
      Hence to prove \eqref{eq:part_3_as_Airy_simplify}, it suffices to show that there exists $K_{\epsilon_1, \epsilon_2}$ such that for $ k_1 > K_{\epsilon_1, \epsilon_2}$,
      \begin{equation} \label{eq:part_3_as_Airy_simplify_more}
        \Prob \left( \sup_{k_2 > k_1} A_{k_1, k_2} > \frac{\epsilon_1}{27} \right) < \frac{\epsilon_2}{12}, \quad \Prob \left( \sup_{k_2 > k_1} B_{k_1, k_2} > \frac{\epsilon_1}{27} \right) < \frac{\epsilon_2}{12}, \quad \text{and} \quad \sup_{k_2 > k_1} 2\lvert C_{k_1, k_2} \rvert < \frac{\epsilon_1}{27}.
      \end{equation}
      By the estimate of $F(x)$ in Proposition \ref{prop:Airy_measure}, we can easily verify the $C_{k_1, k_2}$ part of \eqref{eq:part_3_as_Airy_simplify_more} with large enough $k_1 > K_{\epsilon_1, \epsilon_2}$. On the other hand, analogous to \eqref{eq:appl_Markov_1}, we have that for large enough $i$,
      \begin{equation}
        \Prob \left( \lvert M(-i) - F(-i) \rvert \geq i^{4/5} \right) < i^{-\frac{11}{10}}.
      \end{equation}
      Also if $k_1 > 270\epsilon^{-1}_1 + 1$, we have
      \begin{equation}
        \sum^{\infty}_{i = k_1} i^{-6/5} < \int^{\infty}_{k_1 - 1} x^{-6/5} \dd x < \frac{\epsilon_1}{54}.
      \end{equation}
      Hence if $k_1 > 270\epsilon^{-1}_1 + 1$, noting that $i(i + 1)^2 > (i - 1)i \geq i^2/2$ for all $i \geq k_1 \geq 2$, we have
      \begin{equation} \label{062601}
        \begin{split}
          \Prob \left( \sup_{k_2 > k_1} B_{k_1, k_2} > \frac{\epsilon_1}{27} \right) \leq {} &\Prob \left( \sup_{k_2 > k_1} A_{k_1, k_2} > \frac{\epsilon_1}{27} \right) \\
          \leq {}& \Prob \left( \sum^{\infty}_{i = k_1} \frac{\lvert M(-i) - F(-i) \rvert}{i^2} > \frac{\epsilon_1}{54} \right) \\
          \leq {}& \Prob \left(\bigcup^\infty_{i=k_1} \left\{ \frac{\lvert M(-i) - F(-i) \rvert}{i^2} > i^{-\frac{6}{5}} \right\} \right) \\
          \leq {}& \sum^{\infty}_{i = k_1} \Prob \left( \frac{\lvert M(-i) - F(-i) \rvert}{i^2} > i^{-\frac{6}{5}} \right) \\
          \leq {}& \sum^{\infty}_{i = k_1} i^{-\frac{11}{10}} < \int^{\infty}_{k_1 - 1} x^{-\frac{11}{10}} \dd x = 10(k_1 - 1)^{-\frac{1}{10}},
        \end{split}
      \end{equation}
      and conclude the proof of the $A_{k_1, k_2}$ and $B_{k_1, k_2}$ parts of \eqref{eq:part_3_as_Airy_simplify_more}. Now \eqref{eq:part_3_as_Airy_simplify_more} is proved, and so are \eqref{eq:part_3_as_Airy_simplify} and \eqref{eq:BC_app_2}.
    \end{enumerate}
    
    Thus we finish the proof of  \eqref{19101501}. The constant $m'_{\epsilon_1, \epsilon_2}$ can be deduced from $N_{\epsilon_1,\epsilon_2}$, $N'_{\epsilon_1, \epsilon_2}$, $K_{\epsilon_1, \epsilon_2}$ above.
  \item
    Combining \eqref{eq:statement_1_proof_Airy} and \eqref{19101501} we arrives at that for $m > \max(m_{\epsilon_1}$, $m'_{\epsilon_1, \epsilon_2})$,
    \begin{equation}
      \Prob \left( \sup_{n > m} \left\{ \left\lvert \int^{\xi^{(k)}_m}_{\xi^{(k)}_n} \frac{1}{\xi^{(k)}_j - x} \dd M(x) - \int^{-\left( 3\pi m/2 \right)^{2/3}}_{-\left( 3\pi n/2 \right)^{2/3}} \frac{1}{0 - x} \dd F(x) \right\rvert \Id(\Omega^{(k)}_{j,m,n}) \right\} > \epsilon_1 \right) < \frac{\epsilon_2}{2}.
    \end{equation}
    To complete the proof of \eqref{eq:technical_main_thm_1}, it suffices to find $m_{\epsilon_1, \epsilon_2} > \max(m_{\epsilon_1}$, $m'_{\epsilon_1, \epsilon_2})$ such that for all $m > m_{\epsilon_1, \epsilon_2}$
    \begin{equation} \label{022101}
      \Prob \left( \bigcup_{n \geq m} \big(\Omega^{(k)}_{j, m, n}\big)^c \right) < \frac{\epsilon_2}{2}.
    \end{equation}
    By \eqref{eq:ineq:xi_fluc} in Lemma \ref{eq:ineq:xi_fluc}, we have that if $m$ is large enough, then $\Prob(\lvert \xi^{(k)}_m + \left( 3\pi m/2 \right)^{2/3} \rvert > m^{3/5}) < c m^{-6/5} \log m$, and the estimate also holds if $m$ is replaced by $n$. Also by part \ref{enu:lem:non-asy_est_Airy:1} of Lemma \ref{lem:non-asy_est_Airy}, we have $\lim_{m \to \infty} \Prob \left( \lvert \xi^{(k)}_j \rvert \leq  m^{3/5} \right) = 0$. Hence it is straightforward to check that \eqref{022101} holds if $m$ is large enough, and the desired $m_{\epsilon_1, \epsilon_2}$ exists.
  \end{enumerate}
  
  Finally we complete the proof of \eqref{eq:alt_technical_main_thm_1}.
\end{proof} 

\section{Analysis of random measure $\mu$: proof of Proposition \ref{prop:Airy_measure}} \label{pf.lem_Airy}

In this section, we prove Proposition \ref{prop:Airy_measure}. We will show that the estimate of the mean and variance of $M(x)$ can be transformed to the mean and variance of a linear statistic of the extended Airy process, or more precisely, the difference between the linear statistics of two species, as mentioned in Section \ref{subsec:proof_strategy}. We observe that the mean and variance of the linear statistic have (multi)-contour integral representations. Then we prove the said estimate, first of the mean and then of the variance, via delicate saddle point analysis of the contour integrals.

For any given $x\in \mathbb{R}$, define the almost surely finite random subsets of $\natN$
\begin{equation}
  \mathcal{I}_x := \{ i \in \natN \mid \xi^{(k)}_i \in (x, +\infty) \}, \quad \mathcal{J}_x := \{ i \in \natN \mid \xi^{(k-1)}_i \in (x, +\infty) \}.
\end{equation}
and the random variable
\begin{equation} \label{eq:defn_N_x_Airy}
  N_x = \lvert \mathcal{I}_x \rvert - \lvert \mathcal{J}_x \rvert.
\end{equation}
By the interlacing property stated in Lemma \ref{lem:interlacing_Airy}, we see that $N_x$ is a Bernoulli random variable for a given $x$. Hence,
\begin{equation}
  \E N_x = \Prob(N_x = 1).
\end{equation}
We also consider the random variable
\begin{equation} \label{eq:defn_S_x_Airy}
  S_x = -\sum_{i\in \mathcal{I}_x} \xi^{(k)}_i + \sum_{i\in \mathcal{J}_x} \xi^{(k - 1)}_i.
\end{equation}
We observe that if $\xi^{(k)}_1 \leq x$, then $N_x = S_x = M(x) = 0$. Under the condition that $\xi^{(k)}_1 > x$, if $N_x = 1$, then $S_x = M(x) - \xi^{(k)}_1$, otherwise $S_x = M(x) - \xi^{(k)}_1 + x$. We conclude that (noting that $N_x \Id(\xi^{(k)}_1 > x) = N_x$)
\begin{equation} \label{eq:mu_S_N_Airy}
  M(x) = S_x + (\xi^{(k)}_1 - x(1 - N_x)) \Id(\xi^{(k)}_1 > x) = S_x + xN_x + (\xi^{(k)}_1 - x) \Id(\xi^{(k)}_1 > x).
\end{equation}
By the estimate in Lemma \ref{lem:non-asy_est_Airy}, we find that as $x \to -\infty$,  $\E (\xi^{(k)}_1 - x) \Id(\xi^{(k)}_1 > x) = -x + \bigO(1)$ and $\var (\xi^{(k)}_1 - x) \Id(\xi^{(k)}_1 > x) = \bigO(1)$. Hence to prove  Proposition \ref{prop:Airy_measure}, we only need the following estimates of the mean and variance of $S_x + xN_x$:
\begin{align}
  \E(S_x + xN_x) = {}& \frac{x}{2} - \frac{2}{\pi} a_k \sqrt{-x} + \bigO(1), \label{eq:mean_of_measure} \\
  \var(S_x + xN_x) = {}& \frac{2}{\pi} \sqrt{-x} + a_k + \bigO(\lvert x \rvert^{-1/2}), \label{eq:var_of_measure}
\end{align}
in light of the linearity of  expectation and the trivial inequality 
\begin{equation} \label{rmk:general_CS}
  |\var(X + Y)-(\var(X) + \var(Y))|\leq 2\sqrt{\var(X) \var(Y)}.
\end{equation} 

Below we prove \eqref{eq:mean_of_measure} and \eqref{eq:var_of_measure} separately. For the proofs, we define the function
  \begin{equation} \label{eq:defn_h_x(t)}
    h_x(t) =
    \begin{cases}
      0, & t \leq x, \\
      -t + x, & t > x.
    \end{cases}
  \end{equation}
We emphasize here, all integrals in the sequel, unless the integral domain is specified, are on $\realR$.

\paragraph{Proof of \eqref{eq:mean_of_measure}}

We have, by the standard formula for linear statistics of a determinantal point process, that
\begin{equation} \label{eq:mean_formula_Airy}
  \begin{split}
     \E (S_x + xN_x) 
    = {}& \E \bigg( \sum_{i\in \mathcal{J}_x} \xi^{(k - 1)}_i \bigg) - \E \bigg( \sum_{i\in \mathcal{I}_x} \xi^{(k)}_i \bigg) + x\Big(\E|\mathcal{I}_x| - \E|\mathcal{J}_x|\Big) \\
    = {}& \int^{\infty}_x t K^{k - 1, k - 1}_{\Airya}(t, t)\dd t - \int^{\infty}_x t K^{k, k}_{\Airya}(t, t)\dd t \\
    {}&\qquad+ x \left( \int^{\infty}_x K^{k, k}_{\Airya}(t, t)\dd t - \int^{\infty}_x K^{k - 1, k - 1}_{\Airya}(t, t)\dd t \right) \\
    = {}& \int h_x(t) \Big(K^{k, k}_{\Airya}(t, t) - K^{k - 1, k - 1}_{\Airya}(t, t)\Big)\dd t.
  \end{split}
\end{equation}
According to  \eqref{eq:form_ext_K} and \eqref{eq:form_ext_K_tilde}, we have 
\begin{equation} \label{20040501}
  K^{k, k}_{\Airya}(t, t) - K^{k - 1, k - 1}_{\Airya}(t, t) = \frac{1}{(2\pi \ii)^2} \int_{\sigma} \dd u \int_{\gamma} \dd v \frac{e^{\frac{u^3}{3} - tu}}{e^{\frac{v^3}{3} - tv}} \left( \prod^{k - 1}_{j = 1} \frac{u - a_j}{v - a_j} \right) \frac{1}{v - a_k}.
\end{equation}
Plugging \eqref{20040501} into \eqref{eq:mean_formula_Airy}, we have 
\begin{equation} \label{eq:double_contour_mean_Airy}
  \begin{split}
    \E (S_x + xN_x) = {}& \frac{-1}{(2\pi \ii)^2} \int_{\sigma} \dd u \int_{\gamma} \dd v \frac{e^{\frac{u^3}{3}}}{e^{\frac{v^3}{3}}} \left( \prod^{k - 1}_{j = 1} \frac{u - a_j}{v - a_j} \right) \frac{1}{v - a_k} \int^{+\infty}_x (t - x) e^{-(u - v)t}\dd t \\
    = {}& \frac{-1}{(2\pi \ii)^2} \int_{\sigma} \dd u \int_{\gamma} \dd v \frac{e^{\frac{u^3}{3} - xu}}{e^{\frac{v^3}{3} - xv}} \left( \prod^{k - 1}_{j = 1} \frac{u - a_j}{v - a_j} \right) \frac{1}{(v - a_k)(u - v)^2}.
  \end{split}
\end{equation}
By some standard residue calculation, we have that if $x < 0$, then
\begin{subequations} \label{eq:int_over_X_and_bar}
  \begin{align}
    & \E (S_x + xN_x) \notag \\
    = {}& \frac{-1}{(2\pi \ii)^2} \iint_{\mathsf{X}} \dd u \dd v \frac{e^{\frac{u^3}{3} - xu}}{e^{\frac{v^3}{3} - xv}} \left( \prod^{k - 1}_{j = 1} \frac{u - a_j}{v - a_j} \right) \frac{1}{(v - a_k)(u - v)^2} \label{eq:int_over_X_and_bar:X} \\
    & + \frac{-1}{2\pi \ii} \int^{\sqrt{-x}\ii}_{-\sqrt{-x}\ii} \Big(\frac{v^2 - x}{v - a_k}  + \sum^{k - 1}_{j = 1} \frac{1}{(v - a_j)(v - a_k)}\Big) \dd v, \label{eq:int_over_X_and_bar:bar}
  \end{align}
\end{subequations}
where the contour $\mathsf{X}$ means that the two deformed contours $\sigma$ and $\gamma$ intersect at the two saddle points $\pm\sqrt{-x} \ii$, see Figure \ref{fig:X}, and the integral is understood as the Cauchy principal value, and the contour in \eqref{eq:int_over_X_and_bar:bar} goes by the right of all $a_i$'s. By direct computation, we evaluate
\begin{equation} \label{eq:new_est_bar}
  \begin{split}
    \eqref{eq:int_over_X_and_bar:bar} = {}& \frac{-1}{2\pi \ii} \int^{\sqrt{-x}\ii}_{-\sqrt{-x}\ii} v + a_k + \frac{a_k - x}{v - a_k} + \sum^{k - 1}_{j = 1} \frac{1}{a_j - a_k} \left( \frac{1}{v - a_j} - \frac{1}{v - a_k} \right) \dd v \\
    = {}& \frac{-1}{2\pi \ii} \left[ 2a_k \sqrt{-x} \ii + (a^2_k - x) \ii \left( \pi + 2\arctan \frac{a_k}{\sqrt{-x}} \right) + \sum^{k - 1}_{j = 1} \frac{2 \ii}{a_j - a_k} \left( \arctan \frac{a_j}{\sqrt{-x}} - \arctan \frac{a_k}{\sqrt{-x}} \right) \right] \\
    = {}& \frac{1}{2} (x - a^2_k) - \frac{2}{\pi} a_k \sqrt{-x} + \bigO(\lvert x \rvert^{-\frac{1}{2}}).
  \end{split}
\end{equation}
The above derivation can be easily justified by l'H\^{o}pital's rule when certain $a_j$ equals $a_k$. 
 To evaluate \eqref{eq:int_over_X_and_bar:X}, we define two types of infinite contours:
\begin{equation} \label{eq:defn_standard_gamma_sigma}
  \gamma_{\standard}(a) = \{ e^{\frac{2\pi \ii}{3}} t + a \mid t \geq 0 \} \cup \{ e^{\frac{\pi \ii}{3}} t + a \mid t \leq 0 \}, \quad \sigma_{\standard}(b) = \{ e^{\frac{\pi \ii}{3}} t + b \mid t \geq 0 \} \cup \{ e^{\frac{2\pi \ii}{3}} t + b \mid t \leq 0 \},
\end{equation}
both oriented upward; see Figure \ref{fig:gamma_sigma_standard}. Assuming that $-x$ is large enough, we deform the contour $\mathsf{X}$ such that $u$ is on $\sigma_{\standard}(-\sqrt{-x/3})$, and $v$ is on $\gamma_{\standard}(\sqrt{-x/3})$. Direct computation shows that $\Re(u^3/3 - xu)$ attains its maximum along $\sigma$ at $\pm \sqrt{-x} \ii$, and $\Re(v^3/3 - xv)$ attains its minimum along $\gamma$ at the same two points. Hence $\pm \sqrt{-x} \ii$ are the saddle points. We divide the double contour $\mathsf{X}$ into three disjoint subsets:
\begin{enumerate}[label=(\roman*)]
\item
  $\mathsf{X}_1 = \{ u, v \in \mathsf{X} \mid \lvert u - \sqrt{-x} \ii \rvert < 1, \, \lvert v - \sqrt{-x} \ii \rvert < 1 \}$, 
\item 
  $\mathsf{X}_2 = \{ u, v \in \mathsf{X} \mid \lvert u + \sqrt{-x} \ii \rvert < 1, \, \lvert v + \sqrt{-x} \ii \rvert < 1 \}$, 
\item
  $\mathsf{X}_3 = \mathsf{X} \setminus (\mathsf{X}_1 \cup \mathsf{X}_2)$.
\end{enumerate}
We first estimate the integral over $\mathsf{X}_1$. By Taylor's expansion, if we denote $s = \lvert x \rvert^{1/4} (u - \sqrt{-x} \ii)$ and $t = \lvert x \rvert^{1/4} (v - \sqrt{-x} \ii)$, then on $\mathsf{X}_1$,
\begin{equation} \label{eq:saddle_first}
  \frac{e^{\frac{u^3}{3} - xu}}{e^{\frac{v^3}{3} - xv}} = e^{\ii (s^2 - t^2)} e^{\frac{1}{3} \lvert x \rvert^{-3/4} (s^3 - t^3)},
\end{equation}
and then, according to the definition of the Cauchy principal value of the contour integral,
\begin{align}
  \eqref{eq:int_over_X_and_bar:X} \text{ over } \mathsf{X}_1 = \frac{\ii}{\sqrt{-x}} \frac{1}{(2\pi \ii)^2} \left( \int^{e^{\pi \ii/3} \lvert x \rvert^{1/4}}_0 \dd s + \int^0_{e^{-2\pi \ii/3} \lvert x \rvert^{1/4}} \dd s \right) \left[ \int^{e^{-\pi \ii/3} \lvert x \rvert^{1/4}}_{e^{2\pi \ii/3} \lvert x \rvert^{1/4}} \dd t e^{\ii (s^2 - t^2)} \frac{f(s)}{g(t)} \frac{1}{(s - t)^2} \right],  \label{041901}
\end{align}
where all the integral $\int^b_a$ is over the line segment from $a$ to $b$, and
\begin{equation}
  f(s) = \prod^{k - 1}_{j = 1} \left( 1 + \ii \left( \frac{a_j}{\lvert x \rvert^{\frac{1}{2}}} - \frac{s}{\lvert x \rvert^{\frac{3}{4}}} \right) \right) e^{\frac{1}{3} \lvert x \rvert^{-3/4} s^3}, \quad g(t) = \prod^k_{j = 1} \left( 1 + \ii \left( \frac{a_j}{\lvert x \rvert^{\frac{1}{2}}} - \frac{t}{\lvert x \rvert^{\frac{3}{4}}} \right) \right) e^{\frac{1}{3} \lvert x \rvert^{-3/4} t^3}.
\end{equation}
We see that $f(s)$ and $g(t)^{-1}$ are an analytic and bounded function for all $\lvert s \rvert, \lvert t \rvert = \bigO(\lvert x \rvert^{1/4})$. Then
\begin{equation} \label{eq:not_so_standard}
  \begin{split}
    & \int^{e^{\pi \ii/3} \lvert x \rvert^{1/4}}_0 \dd s \int^{e^{-\pi \ii/3} \lvert x \rvert^{1/4}}_{e^{2\pi \ii/3} \lvert x \rvert^{1/4}} \dd t e^{\ii (s^2 - t^2)} \frac{f(s)}{f(t)} \frac{1}{(s - t)^2} \\
    = {}& -\int^{e^{\pi \ii/3} \lvert x \rvert^{1/4}}_0 f(s) e^{\ii s^2} \frac{\dd G(s)}{\dd s} \dd s, \\
    = {}& f(0) G(0) - f(e^{\pi \ii/3} \lvert x \rvert^{1/4}) e^{e^{7\pi \ii/6} \lvert x \rvert^{1/2}} G(e^{\pi \ii/3} \lvert x \rvert^{1/4}) + \int^{e^{\pi \ii/3} \lvert x \rvert^{1/4}}_0 G(s) (f'(s) + 2\ii sf(s)) e^{\ii s^2} \dd s,
  \end{split}
\end{equation}
where
\begin{equation}
  G(s) = \int^{e^{-\pi \ii/3} \lvert x \rvert^{1/4}}_{e^{2\pi \ii/3} \lvert x \rvert^{1/4}} \dd t e^{-\ii t^2} \frac{1}{g(t)} \frac{1}{s - t}
\end{equation}
and $G(0)$ is understood as the limit of $G(s)$ as $s$ moves to $0$ along the line segment from $e^{\pi \ii/3} \lvert x \rvert^{1/4}$ to $0$. By the analyticity and boundedness of $g(t)^{-1}$, we have that $G(s)$ is also bounded on the line segment from $0$ to $e^{\pi \ii/3} \lvert x \rvert^{1/4}$, so the integral on the left-hand side of \eqref{eq:not_so_standard} is $\bigO(1)$ as $x \to -\infty$. Similarly, we conclude that
\begin{equation} \label{eq:saddle_last}
  \int^{e^{\pi \ii/3} \lvert x \rvert^{1/4}}_0 \dd s \int^{e^{-\pi \ii/3} \lvert x \rvert^{1/4}}_{e^{2\pi \ii/3} \lvert x \rvert^{1/4}} \dd t e^{\ii (s^2 - t^2)} \frac{f(s)}{g(t)} \frac{1}{(s - t)^2} = \bigO(1).
\end{equation}
Hence, plugging the above estimates into (\ref{041901}), we see that the part of integral \eqref{eq:int_over_X_and_bar:X} over $\mathsf{X}_1$ is $\bigO(\lvert x \rvert^{-1/2})$. Similarly, the part of integral \eqref{eq:int_over_X_and_bar:X} over $\mathsf{X}_2$ is also $\bigO(\lvert x \rvert^{-1/2})$. The part of integral \eqref{eq:int_over_X_and_bar:X} over $\mathsf{X}_3$ is easier to estimate, because the factor $e^{\frac{u^3}{3} - xu}/e^{\frac{v^3}{3} - xv}$ vanishes uniformly and the factor $(u - v)^{-2}$ is bounded. By standard method, we find this part of integral \eqref{eq:int_over_X_and_bar:X} is $o(\lvert x \rvert^{-1/2})$. This together with the estimate of \eqref{eq:int_over_X_and_bar:bar} above completes the proof of \eqref{eq:mean_of_measure}.

\paragraph{Proof of \eqref{eq:var_of_measure}}

We can write, with $h_x(t)$ defined in \eqref{eq:defn_h_x(t)},
\begin{multline} \label{eq:variance_formula_Airy}
  \E[(S_x + xN_x)^2] = \int h^2_x(t) K^{k, k}_{\Airya}(t, t)\dd t + \int h^2_x(t) K^{k - 1, k - 1}_{\Airya}(t, t)\dd t \\
  \begin{aligned}[b]
    + {}& \iint h_x(s)h_x(t) 
    \begin{vmatrix}
      K^{k, k}_{\Airya}(s, s) & K^{k, k}_{\Airya}(s, t) \\
      K^{k, k}_{\Airya}(t, s) & K^{k, k}_{\Airya}(t, t)
    \end{vmatrix}
    \dd s\dd t \\
    + {}& \iint h_x(s)h_x(t)
    \begin{vmatrix}
      K^{k - 1, k - 1}_{\Airya}(s, s) & K^{k - 1, k - 1}_{\Airya}(s, t) \\
      K^{k - 1, k - 1}_{\Airya}(t, s) & K^{k - 1, k - 1}_{\Airya}(t, t)
    \end{vmatrix}
    \dd s\dd t \\
    - {}& \iint h_x(s)h_x(t)
    \begin{vmatrix}
      K^{k, k}_{\Airya}(s, s) & K^{k, k - 1}_{\Airya}(s, t) \\
      K^{k - 1, k}_{\Airya}(t, s) & K^{k - 1, k - 1}_{\Airya}(t, t)
    \end{vmatrix}
    \dd s\dd t \\
    - {}& \iint h_x(s)h_x(t)
    \begin{vmatrix}
      K^{k - 1, k - 1}_{\Airya}(s, s) & K^{k - 1, k}_{\Airya}(s, t) \\
      K^{k, k - 1}_{\Airya}(t, s) & K^{k, k}_{\Airya}(t, t)
    \end{vmatrix}
    \dd s\dd t.
  \end{aligned}
\end{multline}
Combining \eqref{eq:variance_formula_Airy} with \eqref{eq:mean_formula_Airy} and using \eqref{eq:form_ext_K} we get 
\begin{subequations} \label{eq:var_3parts_Airy}
  \begin{align}
    & \var\big[ S_x + xN_x \big] \notag \\
    = {}& \int h_x^2(t) \Big(K^{k, k}_{\Airya}(t, t) + K^{k - 1, k - 1}_{\Airya}(t, t)\Big)\dd t - 2\int h_x(t) \left( \int^{\infty}_t h_x(s) e^{a_k(s - t)} K^{k, k - 1}_{\Airya}(s, t) \dd s \right)\dd t \label{eq:var_part_2} \\
    &
      \begin{aligned}
        & - \iint h_x(s)h_x(t) \Big(K^{k, k}_{\Airya}(s, t) K^{k, k}_{\Airya}(t, s) + K^{k - 1, k - 1}_{\Airya}(s, t) K^{k - 1, k - 1}_{\Airya}(t, s) \\
        & \phantom{\smash{-\iint h_x(s)}h_x(t)} - K^{k, k - 1}_{\Airya}(s, t) \widetilde{K}^{k - 1, k}_{\Airya}(t, s) - \widetilde{K}^{k - 1, k}_{\Airya}(s, t) K^{k, k - 1}_{\Airya}(t, s)\Big) \dd s\dd t.
      \end{aligned} \label{eq:var_part_3}
  \end{align}
\end{subequations}

First, we consider the integrals in \eqref{eq:var_part_2}. A simple change of order of integration yields
\begin{equation} \label{eq:part_1_var_trans}
  \int h^2_x(t) \Big(K^{k, k}_{\Airya}(t, t) + K^{k - 1, k - 1}_{\Airya}(t, t)\Big)\dd t  
  = \frac{1}{(2\pi \ii)^2} \int_{\sigma} \dd u \int_{\gamma} \dd v \frac{e^{\frac{u^3}{3} - xu}}{e^{\frac{v^3}{3} - xv}} \frac{2}{(u - v)^4} \left( \prod^{k - 1}_{j = 1} \frac{u - a_j}{v - a_j} \right) \frac{u + v - 2a_k}{v - a_k},
\end{equation}
\begin{equation} \label{eq:part_2_var_trans}
  \int h_x(t) \left( \int^{\infty}_t h_x(s) e^{a_k(s - t)} K^{k, k - 1}_{\Airya}(s, t) \dd s \right)\dd t 
  =\frac{1}{(2\pi \ii)^2} \int_{\sigma} \dd u \int_{\gamma} \dd v \frac{e^{\frac{u^3}{3} - xu}}{e^{\frac{v^3}{3} - xv}} \frac{1}{(u - v)^4} \left( \prod^{k - 1}_{j = 1} \frac{u - a_j}{v - a_j} \right) \frac{(3u - v - 2a_k)}{u - a_k}.
\end{equation}
So \eqref{eq:var_part_2} becomes
\begin{equation}
  \frac{2}{(2\pi \ii)^2} \int_{\sigma} \dd u \int_{\gamma} \dd v \frac{e^{\frac{u^3}{3} - xu}}{e^{\frac{v^3}{3} - xv}} \left( \prod^{k - 1}_{j = 1} \frac{u - a_j}{v - a_j} \right) \frac{1}{(u - v)^2 (u - a_k)(v - a_k)}.
\end{equation}
Similarly to \eqref{eq:int_over_X_and_bar}, when $x < 0$, this integral can be written as
\begin{subequations} 
  \begin{align}
    & \frac{2}{(2\pi \ii)^2} \iint_{\mathsf{X}} \dd u \dd v \frac{e^{\frac{u^3}{3} - xu}}{e^{\frac{v^3}{3} - xv}} \left( \prod^{k - 1}_{j = 1} \frac{u - a_j}{v - a_j} \right) \frac{1}{(u - v)^2 (u - a_k)(v - a_k)} \label{eq:variance_Airy_X} \\
    + {}& \frac{2}{2\pi \ii} \int^{\sqrt{-x}\ii}_{-\sqrt{-x}\ii} \left( v^2 - x + \sum^{k - 1}_{j = 1} \frac{1}{v - a_j} - \frac{1}{v - a_k} \right) \frac{1}{(v - a_k)^2} \dd v \label{eq:variance_Airy_bar} \\
    + {}& \prod^{k - 1}_{j = 1} (a_k - a_j) e^{\frac{a^3_k}{3} - a_k x} \frac{2}{2\pi \ii} \int_{\gamma} \frac{1}{e^{\frac{v^3}{3} - xv}} \left( \prod^{k - 1}_{j = 1} \frac{1}{v - a_j} \right) \frac{1}{(v - a_k)^3} \dd v, \label{eq:single_int_gamma_L}
  \end{align}
\end{subequations}
where the contour $\mathsf{X}$ in \eqref{eq:variance_Airy_X} and the contour from $-\sqrt{-x}\ii$ to $\sqrt{-x}\ii$ in \eqref{eq:variance_Airy_bar} are the same as those in \eqref{eq:int_over_X_and_bar}. We have that analogous to \eqref{eq:int_over_X_and_bar:bar},
\begin{equation}
  \begin{split}
    \eqref{eq:variance_Airy_bar} = {}& \frac{2}{2\pi \ii} \int^{\sqrt{-x}\ii}_{-\sqrt{-x}\ii} \left[ 1 + \frac{2a_k}{v - a_k} + \frac{a^2_k - x}{(v - a_k)^2} \right] + \sum^{k - 1}_{j = 1} \left( \frac{1}{v - a_j} - \frac{1}{v - a_k} \right) \frac{1}{(v - a_k)^2} \dd v \\
    = {}& \frac{2}{2\pi \ii} \left[ 2\sqrt{-x} \ii + 2a_k \ii \left( \pi - 2 \arctan \frac{a_k}{\sqrt{-x}} \right) + (a^2_k - x) \frac{2\sqrt{-x} \ii}{a^2_k - x} \right] + \bigO(\lvert x \rvert^{-1}) \\
    = {}& \frac{4\sqrt{-x}}{\pi} + 2a_k + \bigO(\lvert x \rvert^{-\frac{1}{2}}),
  \end{split}
\end{equation}
and the integral in \eqref{eq:variance_Airy_X} is $\bigO( \lvert x \rvert^{-1})$. The saddle point analysis is omitted since it is analogous to that for \eqref{eq:int_over_X_and_bar}. The integral in \eqref{eq:single_int_gamma_L} will be cancelled out later, by a term in (\ref{eq:third_transform}).

On the other hand, the double integral in \eqref{eq:var_part_3} can be expressed as
\begin{equation} \label{eq:four_contour_Airy}
  \frac{1}{(2\pi \ii)^4} \int_{\sigma} \dd u \int_{\gamma} \dd v \int_{\sigma}\dd z \int_{\gamma} \dd w \frac{e^{\frac{u^3}{3} - xu}}{e^{\frac{v^3}{3} - xv}} \frac{e^{\frac{z^3}{3} - xz}}{e^{\frac{w^3}{3} - xw}} \prod^{k - 1}_{j = 1} \frac{u - a_j}{v - a_j} \frac{z - a_j}{w - a_j} 
  \frac{1}{(u - v)(z - w) (u - w)(z - v) (v - a_k)(w - a_k)}.
\end{equation}
In order to estimate the integral above, we will perform several steps of contour deformation.
\begin{enumerate}[label=(\Roman*)]
\item
  We first deform the contour $\gamma$ for $w$ and $v$ to $\gamma^{\outside}_w \cup \gamma^{\inside}_w$ and $\gamma^{\outside}_v \cup \gamma^{\inside}_v$, respectively, as in Figure \ref{fig:w_v_separation}, such that all $a_j$ ($j = 1, \dotsc, k$) are enclosed in $\gamma^{\inside}_w$, and then also in $\gamma^{\inside}_v$. We also slightly deform the contour $\sigma$ for $u$ and $z$ into $\sigma_u$ and $\sigma_z$, respectively, as shown in Figure \ref{fig:w_v_separation}.
  
  \begin{figure}[htb]
    \begin{minipage}[t]{5cm}
      \centering
      \includegraphics{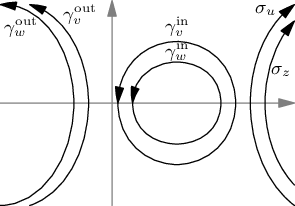}
      \caption{Deformation of contour $\gamma$ for $w$ and $v$, and contour $\sigma$ for $u$ and $z$.}
      \label{fig:w_v_separation}
    \end{minipage}
    \hspace{\stretch{1}}
    \begin{minipage}[t]{5cm}
      \centering
      \includegraphics{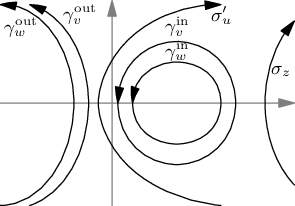}
      \caption{$\sigma_u$ is deformed into $\sigma'_u$.}
      \label{fig:gamma_u_deformed}
    \end{minipage}
    \hspace{\stretch{1}}
    \begin{minipage}[t]{5cm}
      \centering
      \includegraphics{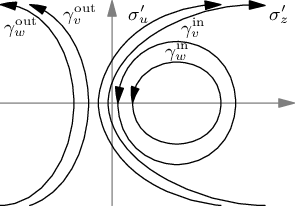}
      \caption{$\sigma_z$ is deformed into $\sigma'_z$.}
      \label{fig:gamma_z_deformed}
    \end{minipage}
  \end{figure}
  
\item
  We then further deform the contour $\sigma_u$ such that it goes between $\gamma^{\outside}_v$ and $\gamma^{\inside}_v$, and thus also goes between $\gamma^{\outside}_w$ and $\gamma^{\inside}_w$. We denote by $\sigma'_u$ the deformed $\sigma_u$; see Figure \ref{fig:gamma_u_deformed}. By residue calculation, we write \eqref{eq:four_contour_Airy} as
  \begin{subequations} \label{eq:first_transform}
    \begin{multline} \label{eq:first_transform_main}
      \frac{1}{(2\pi \ii)^4} \int_{\sigma'_u} \dd u \int_{\gamma^{\outside}_v \cup \gamma^{\inside}_v} \dd v \int_{\sigma_z} \dd z \int_{\gamma^{\outside}_w \cup \gamma^{\inside}_w} \dd w \frac{e^{\frac{u^3}{3} - xu}}{e^{\frac{v^3}{3} - xv}} \frac{e^{\frac{z^3}{3} - xz}}{e^{\frac{w^3}{3} - xw}} \prod^{k - 1}_{j = 1} \frac{u - a_j}{v - a_j} \frac{z - a_j}{w - a_j} \\
      \times \frac{1}{(u - v)(z - w) (u - w)(z - v) (v - a_k)(w - a_k)}
    \end{multline}
    \begin{flalign}
      && + {}& \frac{1}{(2\pi \ii)^2} \int_{\sigma_z} \dd z \int_{\gamma^{\inside}_w} \dd w \frac{e^{\frac{z^3}{3} - xz}}{e^{\frac{w^3}{3} -xw}} \left( \prod^{k - 1}_{j = 1} \frac{z - a_j}{w - a_j} \right) \frac{1}{(w - a_k)(z - a_k) (z - w)^2} \label{eq:first_addition_Airy} \\
      && + {}& \frac{1}{(2\pi \ii)^2} \int_{\gamma^{\inside}_v} \dd v \int_{\sigma_z} \dd z \frac{e^{\frac{z^3}{3} - xz}}{e^{\frac{v^3}{3} -xv}} \left( \prod^{k - 1}_{j = 1} \frac{z - a_j}{v - a_j} \right) \frac{-1}{(v - a_k)^2(z - a_k) (z - v)} \label{eq:first_addition_Airy_v} \\
      && + {}& \frac{1}{(2\pi \ii)^2} \int_{\gamma^{\outside}_v} \dd v \int_{\sigma_z} \dd z \frac{e^{\frac{z^3}{3} - xz}}{e^{\frac{v^3}{3} - xv}} \left( \prod^{k - 1}_{j = 1} \frac{z - a_j}{v - a_j} \right) \frac{-1}{(v - a_k)^2 (z - a_k)(z - v)} \label{eq:first_transform_a1} \\
      && + {}& \frac{1}{(2\pi \ii)^2} \int_{\gamma^{\outside}_w} \dd v \int_{\sigma_z} \dd z \frac{e^{\frac{z^3}{3} - xz}}{e^{\frac{w^3}{3} - xw}} \left( \prod^{k - 1}_{j = 1} \frac{z - a_j}{w - a_j} \right) \frac{-1}{(w - a_k)^2 (z - a_k)(z - w)}. \label{eq:first_transform_a2}
    \end{flalign}
  \end{subequations}
  \begin{proof}[Derivation of \eqref{eq:first_addition_Airy}--\eqref{eq:first_transform_a2}] 
    First we integrate $u$ over $\sigma_u - \sigma'_u$ that is equivalent to a positively oriented closed contour enclosing $\gamma^{\inside}_v$ (and also $\gamma^{\inside}_w$), and the result is a $3$-fold integral in $v, w, z$:
    \begin{multline} \label{eq:added_3-integral}
      \frac{1}{(2\pi \ii)^3} \int_{\gamma^{\outside}_v \cup \gamma^{\inside}_v} \dd v \int_{\sigma_z}\dd z \int_{\gamma^{\outside}_w \cup \gamma^{\inside}_w} \dd w \Bigg(\frac{e^{\frac{z^3}{3} - xz}}{e^{\frac{w^3}{3} - xw}} \prod^{k - 1}_{j = 1} \frac{z - a_j}{w - a_j} \frac{1}{(z - w) (v - w)(z - v) (v - a_k)(w - a_k)} \Id(v \in \gamma^{\inside}_v) \\
      + \frac{e^{\frac{z^3}{3} - xz}}{e^{\frac{v^3}{3} - xv}} \prod^{k - 1}_{j = 1} \frac{z - a_j}{v - a_j} \frac{1}{(w - v)(z - w) (z - v) (v - a_k)(w - a_k)} \Id(w \in \gamma^{\inside}_w)\Bigg).
    \end{multline}
    The integral domain of $z$ is always $\sigma_z$ so we leave it alone, and divide the integral domain of $v, w$ into $4$ subdomains:
    \begin{enumerate*}[label=(\roman*)]
    \item \label{enu:step_2_1}
      $\gamma^{\outside}_v \times \gamma^{\outside}_w$,
    \item \label{enu:step_2_2}
      $\gamma^{\outside}_v \times \gamma^{\inside}_w$,
    \item \label{enu:step_2_3}
      $\gamma^{\inside}_v \times \gamma^{\outside}_w$ and
    \item \label{enu:step_2_4}
      $\gamma^{\inside}_v \times \gamma^{\inside}_w$.  
    \end{enumerate*}
    Integrating $v, w$ on subdomain \ref{enu:step_2_1}, the result is $0$; integrating $w$ on subdomain \ref{enu:step_2_2}, the result is \eqref{eq:first_transform_a1}; integrating $v$ on subdomain \ref{enu:step_2_3}, the result is \eqref{eq:first_transform_a2}; on subdomain \ref{enu:step_2_4}, it is more complicated: the integrand can be divided into two parts, such that when we integrate one part with respect to $v$, we get \eqref{eq:first_addition_Airy}, and when we integrate the other part with respect to $w$, we get \eqref{eq:first_addition_Airy_v}.
  \end{proof}
  
\item
  Next, similarly to the previous step, we further  deform the contour $\sigma_z$ such that it goes between $\gamma^{\outside}_v$ and $\gamma^{\inside}_v$, and thus also goes between $\gamma^{\outside}_w$ and $\gamma^{\inside}_w$. Hence $\sigma_z$ becomes $\sigma'_z$; see Figure \ref{fig:gamma_z_deformed}. By residue calculation, the quantity in \eqref{eq:first_transform} becomes
  \begin{subequations} \label{eq:second_transform}
    \begin{multline} \label{eq:2nd_transform}
      \frac{1}{(2\pi \ii)^4} \int_{\sigma'_u} \dd u \int_{\gamma^{\outside}_v \cup \gamma^{\inside}_v} \dd v \int_{\sigma'_z}\dd z \int_{\gamma^{\outside}_w \cup \gamma^{\inside}_w} \dd w \frac{e^{\frac{u^3}{3} - xu}}{e^{\frac{v^3}{3} - xv}} \frac{e^{\frac{z^3}{3} - xz}}{e^{\frac{w^3}{3} - xw}} \prod^{k - 1}_{j = 1} \frac{u - a_j}{v - a_j} \frac{z - a_j}{w - a_j} \\
      \times \frac{1}{(u - v)(z - w) (u - w)(z - v) (v - a_k)(w - a_k)}
    \end{multline}
    \begin{flalign}
      && + {}& \frac{1}{(2\pi \ii)^2} \int_{\sigma'_u} \dd u \int_{\gamma^{\inside}_w} \dd w \frac{e^{\frac{u^3}{3} - xu}}{e^{\frac{w^3}{3} -xw}} \left( \prod^{k - 1}_{j = 1} \frac{u - a_j}{w - a_j} \right) \frac{1}{(w - a_k)(u - a_k) (u - w)^2} \label{eq:gamma'_u_and_w_in} \\
      && + {}& \frac{1}{(2\pi \ii)^2} \int_{\sigma'_u} \dd u \int_{\gamma^{\inside}_v} \dd v \frac{e^{\frac{u^3}{3} - xu}}{e^{\frac{v^3}{3} -xv}} \left( \prod^{k - 1}_{j = 1} \frac{u - a_j}{v - a_j} \right) \frac{-1}{(v - a_k)^2(u - a_k) (u - v)} \label{eq:first_addition_Airy_v_u} \\
      && + {}& \frac{1}{(2\pi \ii)^2} \int_{\sigma'_u} \dd u \int_{\gamma^{\outside}_v} \dd v \frac{e^{\frac{u^3}{3} - xu}}{e^{\frac{v^3}{3} - xv}} \left( \prod^{k - 1}_{j = 1} \frac{u - a_j}{v - a_j} \right) \frac{-1}{(v - a_k)^2 (u - a_k)(u - v)} \label{eq:int_out_v_and_u} \\
      && + {}& \frac{1}{(2\pi \ii)^2} \int_{\sigma'_u} \dd u \int_{\gamma^{\outside}_w} \dd w \frac{e^{\frac{u^3}{3} - xu}}{e^{\frac{w^3}{3} - xw}} \left( \prod^{k - 1}_{j = 1} \frac{u - a_j}{w - a_j} \right) \frac{-1}{(w - a_k)^2 (u - a_k)(u - w)} \label{eq:int_out_w_and_u} \\
      && + {}& \frac{1}{(2\pi \ii)^2} \int_{\sigma'_z}\dd z \int_{\gamma^{\inside}_w} \dd w \frac{e^{\frac{z^3}{3} - xz}}{e^{\frac{w^3}{3} -xw}} \left( \prod^{k - 1}_{j = 1} \frac{z - a_j}{w - a_j} \right) \frac{1}{(w - a_k)(z - a_k) (z - w)^2} \label{eq:gamma'_z_and_w_in} \\
      && + {}& \frac{1}{2\pi \ii} \int_{\gamma^{\inside}_w} \dd w \left( w^2 - x + \sum^{k - 1}_{j = 1} \frac{1}{w - a_j} - \frac{1}{w - a_k} \right) \frac{1}{(w - a_k)^2} \label{eq:gamma'_z_and_w_in_add} \\
      && + {}& \prod^{k - 1}_{j = 1} (a_k - a_j) e^{\frac{a^3_k}{3} - a_k x} \frac{1}{2\pi \ii} \int_{\gamma^{\inside}_w} \dd w \frac{1}{e^{\frac{w^3}{3} - xw}} \left( \prod^{k - 1}_{j = 1} \frac{1}{w - a_j} \right) \frac{1}{(w - a_k)^3} \label{eq:gamma'_z_and_w_in_add_2} \\
      && + {}& \frac{1}{(2\pi \ii)^2} \int_{\gamma^{\inside}_v} \dd v \int_{\sigma'_z}\dd z \frac{e^{\frac{z^3}{3} - xz}}{e^{\frac{v^3}{3} -xv}} \left( \prod^{k - 1}_{j = 1} \frac{z - a_j}{v - a_j} \right) \frac{-1}{(v - a_k)^2(z - a_k) (z - v)} \label{eq:first_addition_Airy_v_out} \\
      && + {}& \prod^{k - 1}_{j = 1} (a_k - a_j) e^{\frac{a^3_k}{3} - a_k x} \frac{1}{2\pi \ii} \int_{\gamma^{\inside}_v} \dd v \frac{1}{e^{\frac{v^3}{3} - xv}} \left( \prod^{k - 1}_{j = 1} \frac{1}{v - a_j} \right) \frac{1}{(v - a_k)^3} \label{eq:gamma'_z_and_w_in_add_2_v} \\
      && + {}& \frac{1}{(2\pi \ii)^2} \int_{\gamma^{\outside}_w} \dd w \int_{\sigma'_z}\dd z \frac{e^{\frac{z^3}{3} - xz}}{e^{\frac{w^3}{3} - xw}} \left( \prod^{k - 1}_{j = 1} \frac{z - a_j}{w - a_j} \right) \frac{-2}{(w - a_k)^2 (z - a_k)(z - w)} \label{eq:int_out_w_and_z_times_2} \\
      && + {}& \prod^{k - 1}_{j = 1} (a_k - a_j) e^{\frac{a^3_k}{3} - a_k x} \frac{1}{2\pi \ii} \int_{\gamma^{\outside}_w} \dd w \frac{1}{e^{\frac{w^3}{3} - xw}} \left( \prod^{k - 1}_{j = 1} \frac{1}{w - a_j} \right) \frac{2}{(w - a_k)^3}. \label{eq:int_out_w_and_z_times_2_new}
    \end{flalign}
  \end{subequations}
  \begin{proof}[Derivation of \eqref{eq:second_transform}]
    \begin{enumerate}[label=(\roman*)]
    \item
      \eqref{eq:first_transform_main} is transformed to the sum of \eqref{eq:2nd_transform}--\eqref{eq:int_out_w_and_u}, by an argument similar to that used to  transform  \eqref{eq:four_contour_Airy}  to the sum  of \eqref{eq:first_transform_main}--\eqref{eq:first_transform_a2}.
    \item
      \eqref{eq:first_addition_Airy} is the sum of \eqref{eq:gamma'_z_and_w_in}, \eqref{eq:gamma'_z_and_w_in_add} and \eqref{eq:gamma'_z_and_w_in_add_2}. To see it, we express the contours $\sigma_z = \sigma'_z + \gamma^{\inside}_v$, and check that if the integral domain of \eqref{eq:first_addition_Airy} is changed into $(z, w) \in \gamma^{\inside}_v \times \gamma^{\inside}_w$, by integrating $z$ first, we obtain \eqref{eq:gamma'_z_and_w_in_add} plus \eqref{eq:gamma'_z_and_w_in_add_2}.
    \item
      \eqref{eq:first_addition_Airy_v} is the sum of \eqref{eq:first_addition_Airy_v_out} and \eqref{eq:gamma'_z_and_w_in_add_2_v}. To see it, we express the contours $\sigma_z = \sigma'_z + \gamma^{\inside}_z$, where $\gamma^{\inside}_z$ encloses $\gamma^{\inside}_v$, and then find that the part of \eqref{eq:first_addition_Airy_v} where $\sigma_z$ is replaced by $\sigma'_z$ becomes \eqref{eq:first_addition_Airy_v_out} and the part where $\sigma_z$ is replaced by $\gamma^{\inside}_z$ becomes \eqref{eq:gamma'_z_and_w_in_add_2_v}.
    \item 
      \eqref{eq:first_transform_a1} and \eqref{eq:first_transform_a2} are equal,  and their sum is equal to the sum of \eqref{eq:int_out_w_and_z_times_2} and \eqref{eq:int_out_w_and_z_times_2_new}.
    \end{enumerate}
  \end{proof}

  \begin{figure}[htb]
  \end{figure}
\item \label{enu:step_deform_v}
  For further deformation of the contours, we introduce the following shorthand notations
  \begin{equation} \label{four dots}
    \Airyupperright := \sqrt{-x}\ii + \frac{1}{\sqrt{-x}}, \quad \Airyupperleft := \sqrt{-x}\ii + \frac{\sqrt{3} \ii}{\sqrt{-x}}, \quad \Airylowerright := \sqrt{-x}\ii + \frac{-\sqrt{3} \ii}{\sqrt{-x}}, \quad \Airylowerleft := \sqrt{-x}\ii + \frac{-1}{\sqrt{-x}}.
  \end{equation}
  For the $4$-fold integral \eqref{eq:2nd_transform}, we perform the following operations:
  \begin{enumerate}[label=(\roman*)]
  \item 
    deform $\sigma'_u$ such that it passes $\Airyupperleft$ and $\overline{\Airyupperleft}$;
  \item
    deform $\sigma'_z$ such that it passes $\Airyupperright$ and $\overline{\Airyupperright}$;
  \item 
    deform $\gamma^{\outside}_v$ such that it goes from $e^{-2\pi \ii/3} \cdot \infty$ to $\overline{\Airyupperleft}$, then goes along the left side of $\sigma'_u$ until it reaches $\Airyupperleft$, and then goes from $\Airyupperleft$ to $e^{2\pi \ii/3} \cdot \infty$;
  \item 
    deform $\gamma^{\inside}_v$ such that it goes from $\Airyupperright$ to $\overline{\Airyupperright}$ along the right side of $\sigma'_z$, then wraps around all $a_j$'s, and finally goes back to $\Airyupperright$; 
  \item
    and at last add an additional contour for $v$, on which the contour integral vanishes: the contour goes from $\overline{\Airyupperright}$ to $\Airyupperright$ along the left side of $\sigma'_z$, then goes from $\Airyupperright$ to $\Airyupperleft$, and further goes  from $\Airyupperleft$ to $\overline{\Airyupperleft}$ along the right side of $\sigma'_u$, and finally goes from $\overline{\Airyupperleft}$ to $\overline{\Airyupperright}$.
  \end{enumerate}
  See Figure \ref{fig:v_cross_uz} for the deformation of contours.

  Now we define the infinite contour $\gamma'_v$ as in Figure \ref{fig:v_cross_uz} that goes from $e^{-2\pi \ii/3} \cdot \infty$ to $\overline{\Airyupperleft}$, then to $\overline{\Airyupperright}$, then wraps $a_j$'s until it reaches $\Airyupperright$, and then goes to $\Airyupperleft$, and finally goes to $e^{2\pi \ii/3} \cdot \infty$. Hence the $4$-fold integral \eqref{eq:2nd_transform} can be simplified by the residue theorem with $\gamma^{\outside}_v \cup \gamma^{\inside}_v$ replaced by $\gamma'_v$. Then the formula \eqref{eq:second_transform} becomes
  \begin{subequations} \label{eq:third_transform}
    \begin{multline} \label{eq:4-fold_PV_v}
      \frac{1}{(2\pi \ii)^4} \int_{\sigma'_u} \dd u \int_{\sigma'_z}\dd z \int_{\gamma^{\outside}_w \cup \gamma^{\inside}_w} \dd w \PV \int_{\gamma'_v} \dd v \frac{e^{\frac{u^3}{3} - xu}}{e^{\frac{v^3}{3} - xv}} \frac{e^{\frac{z^3}{3} - xz}}{e^{\frac{w^3}{3} - xw}} \prod^{k - 1}_{j = 1} \frac{u - a_j}{v - a_j} \frac{z - a_j}{w - a_j} \\
      \times \frac{1}{(u - v)(z - w) (u - w)(z - v) (v - a_k)(w - a_k)}
    \end{multline}
    \begin{flalign}
      && + {}& \frac{1}{(2\pi \ii)^3} \int_{\sigma'_z}\dd z \int_{\gamma^{\outside}_w \cup \gamma^{\inside}_w} \dd w \int^{\Airyupperleft}_{\overline{\Airyupperleft}} \dd u \frac{e^{\frac{z^3}{3} - xz}}{e^{\frac{w^3}{3} - xw}} \left( \prod^{k - 1}_{j = 1} \frac{z - a_j}{w - a_j} \right) \frac{1}{(z - w) (u - w)(z - u) (u - a_k)(w - a_k)} \label{eq:u_out} \\
      && + {}& \frac{1}{(2\pi \ii)^3} \int_{\sigma'_u} \dd u \int_{\gamma^{\outside}_w \cup \gamma^{\inside}_w} \dd w  \int^{\Airyupperright}_{\overline{\Airyupperright}} \dd z \frac{e^{\frac{u^3}{3} - xu}}{e^{\frac{w^3}{3} - xw}} \left( \prod^{k - 1}_{j = 1} \frac{u - a_j}{w - a_j} \right) \frac{1}{(u - z)(z - w) (u - w) (z - a_k)(w - a_k)} \label{eq:u_in} \\
      && + {}& \frac{1}{(2\pi \ii)^2} \int_{\gamma^{\outside}_w} \dd w \int_{\sigma'_z}\dd z \frac{e^{\frac{z^3}{3} - xz}}{e^{\frac{w^3}{3} - xw}} \left( \prod^{k - 1}_{j = 1} \frac{z - a_j}{w - a_j} \right) \frac{-4}{(w - a_k)^2 (z - a_k)(z - w)} \label{eq:int_out_w_and_z_times_4} \\
    && + {}& \prod^{k - 1}_{j = 1} (a_k - a_j) e^{\frac{a^3_k}{3} - a_k x} \frac{1}{2\pi \ii} \int_{\gamma^{\inside}_w \cup \gamma^{\outside}_w} \dd w \frac{1}{e^{\frac{w^3}{3} - xw}} \left( \prod^{k - 1}_{j = 1} \frac{1}{w - a_j} \right) \frac{2}{(w - a_k)^3} \label{eq:int_in_w_out_w_twice} \\
    && + {}& \frac{1}{(2\pi \ii)^2} \int_{\sigma'_z}\dd z \int_{\gamma^{\inside}_w} \dd w \frac{e^{\frac{z^3}{3} - xz}}{e^{\frac{w^3}{3} -xw}} \left( \prod^{k - 1}_{j = 1} \frac{z - a_j}{w - a_j} \right) \frac{2(2w - z - a_k)}{(w - a_k)^2 (z - a_k) (z - w)^2} \label{eq:gamma'_z_and_w_in_twice} \\
    && + {}& 2a_k. \notag
  \end{flalign}
  \end{subequations}
  \begin{proof}[Derivation of \eqref{eq:third_transform}]
    \begin{enumerate}[label=(\roman*)]
    \item
      \eqref{eq:2nd_transform} is equal to the sum of \eqref{eq:4-fold_PV_v}--\eqref{eq:u_in}.
    \item
      The sum of \eqref{eq:int_out_v_and_u}, \eqref{eq:int_out_w_and_u}, and \eqref{eq:int_out_w_and_z_times_2} is equal to \eqref{eq:int_out_w_and_z_times_4}.
    \item
      The sum of \eqref{eq:gamma'_z_and_w_in_add_2} \eqref{eq:gamma'_z_and_w_in_add_2_v} and \eqref{eq:int_out_w_and_z_times_2_new} is equal to \eqref{eq:int_in_w_out_w_twice}.
    \item
      The sum of \eqref{eq:gamma'_u_and_w_in}, \eqref{eq:first_addition_Airy_v_u}, \eqref{eq:gamma'_z_and_w_in}, \eqref{eq:first_addition_Airy_v_out} is equal to \eqref{eq:gamma'_z_and_w_in_twice}.
    \item
      \eqref{eq:gamma'_z_and_w_in_add} is equal to $2a_k$.
    \end{enumerate}
  \end{proof}
  
\item
  Now we deform the contour $\gamma^{\outside}_w \cup \gamma^{\inside}_w$ for $w$ in the way similar to our deformation of $\gamma^{\outside}_v \cup \gamma^{\inside}_v$ for $v$ in Step \ref{enu:step_deform_v}. We perform the following operations:
  \begin{enumerate}[label=(\roman*)]
  \item
    deform $\sigma'_u$ such that it passes $\Airylowerleft$ and $\overline{\Airylowerleft}$ and meanwhile still passes $\Airyupperleft$ and $\overline{\Airyupperleft}$;
  \item
    deform $\sigma'_z$ such that it passes $\Airylowerright$ and $\overline{\Airylowerright}$ and meanwhile still passes $\Airyupperright$ and $\overline{\Airyupperright}$;
  \item 
    deform $\gamma^{\outside}_w$ such that it goes from $e^{-2\pi \ii/3} \cdot \infty$ to $\overline{\Airylowerleft}$, then goes along the left side of $\sigma'_u$ to $\Airylowerleft$, and finally goes from $\Airylowerleft$ to $e^{2\pi \ii/3} \cdot \infty$;
  \item 
    deform $\gamma^{\inside}_w$ such that it goes from $\Airylowerright$ to $\overline{\Airylowerright}$ along the right side of $\sigma'_z$, then wraps around all  $a_j$'s, and finally goes back to $\Airylowerright$;
  \item
    and at last add an additional contour for $w$, on which the contour integral vanishes: the contour goes from $\overline{\Airylowerright}$ to $\Airylowerright$ along the left-side of $\sigma'_z$, then goes from $\Airylowerright$ to $\Airylowerleft$, then from $\Airylowerleft$ to $\overline{\Airylowerleft}$ along the right side of $\sigma'_u$, and finally goes from $\overline{\Airylowerleft}$ to $\overline{\Airylowerright}$.
  \end{enumerate}
  \begin{figure}[htb]
    \begin{minipage}[t]{8cm}
      \centering
      \includegraphics{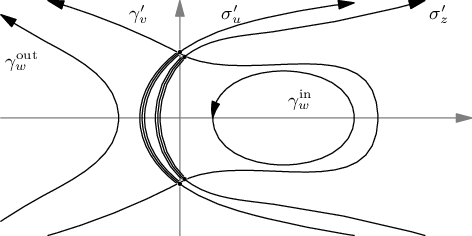}
      \caption{The deformed contours for $v$, $u$ and $z$. The four dots are $\Airyupperleft$, $\Airyupperright$, $\overline{\Airyupperleft}$ and $\overline{\Airyupperright}$.  }
      \label{fig:v_cross_uz}
    \end{minipage}
    \hspace{\stretch{1}}
    \begin{minipage}[t]{8cm}
      \centering
      \includegraphics{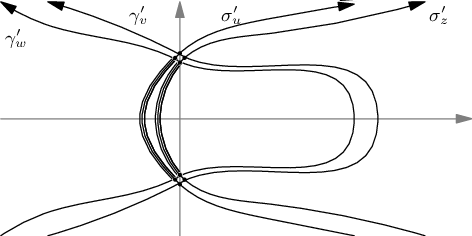}
      \caption{The deformed contours for $w$, $u$ and $z$. The four dots on top are $\Airyupperleft$, $\Airyupperright$, $\Airylowerleft$ and $\Airylowerright$, and the four dots on bottom are their complex conjugates.}
      \label{fig:w_cross_uz}
    \end{minipage}
  \end{figure}
  We have the result in Figure \ref{fig:w_cross_uz}. Similar to  $\gamma'_v$ in Figure \ref{fig:v_cross_uz}, now we define the infinite contour $\gamma'_w$  that goes from $e^{-2\pi \ii/3} \cdot \infty$ to $\overline{\Airylowerleft}$, then to $\overline{\Airylowerright}$, then wraps all $a_j$'s until it reaches $\Airylowerright$, and then goes to $\Airylowerleft$, and finally goes to $e^{2\pi \ii/3} \cdot \infty$. Contour $\gamma'_w$ is parallel to and to the left of  $\gamma'_v$.  Hence by the residue theorem, the $4$-fold integral \eqref{eq:4-fold_PV_v} can be simplified with $\gamma^{\outside}_w \cup \gamma^{\inside}_w$ replaced by $\gamma'_w$, and then formula \eqref{eq:third_transform} becomes
  \begin{subequations} \label{eq:collective_last_Airy}
    \begin{multline} \label{eq:4-fold_PV_vw}
      \frac{1}{(2\pi \ii)^4} \int_{\sigma'_u} \dd u \int_{\sigma'_z}\dd z \PV \int_{\gamma'_w} \dd w \PV \int_{\gamma'_v} \dd v \frac{e^{\frac{u^3}{3} - xu}}{e^{\frac{v^3}{3} - xv}} \frac{e^{\frac{z^3}{3} - xz}}{e^{\frac{w^3}{3} - xw}} \prod^{k - 1}_{j = 1} \frac{u - a_j}{v - a_j} \frac{z - a_j}{w - a_j} \\
      \times \frac{1}{(u - v)(z - w) (u - w)(z - v) (v - a_k)(w - a_k)}
    \end{multline}
    \begin{flalign}
      && + {}& \frac{1}{(2\pi \ii)^3} \int_{\sigma'_z}\dd z \PV \int_{\gamma'_v} \dd v \int^{\Airylowerleft}_{\overline{\Airylowerleft}} \dd u \frac{e^{\frac{z^3}{3} - xz}}{e^{\frac{v^3}{3} - xv}} \left( \prod^{k - 1}_{j = 1} \frac{z - a_j}{v - a_j} \right) \frac{1}{(u - v)(z - u) (z - v) (v - a_k)(u - a_k)} \label{eq:int_vz_u} \\
      && + {}& \frac{1}{(2\pi \ii)^3} \int_{\sigma'_u} \dd u \PV \int_{\gamma'_v} \dd v \int^{\Airylowerright}_{\overline{\Airylowerright}} \dd z \frac{e^{\frac{u^3}{3} - xu}}{e^{\frac{v^3}{3} + xv}} \left( \prod^{k - 1}_{j = 1} \frac{u - a_j}{v - a_j} \right) \frac{1}{(u - v)(u - z) (z - v) (v - a_k)(z - a_k)} \label{eq:int_vu_z} \\
      && + {}& \frac{1}{(2\pi \ii)^2} \int_{\gamma^{\inside}_w \cup \gamma^{\outside}_w} \dd w \int_{\sigma'_z}\dd z \frac{e^{\frac{z^3}{3} - xz}}{e^{\frac{w^3}{3} - xw}} \left( \prod^{k - 1}_{j = 1} \frac{z - a_j}{w - a_j} \right) \frac{-4}{(w - a_k)^2 (z - a_k)(z - w)} \label{eq:int_out_w_and_in_w_z_times_4} \\
      && + {}& \frac{1}{(2\pi \ii)^2} \int_{\sigma'_z}\dd z \int_{\gamma^{\inside}_w} \dd w \frac{e^{\frac{z^3}{3} - xz}}{e^{\frac{w^3}{3} -xw}} \left( \prod^{k - 1}_{j = 1} \frac{z - a_j}{w - a_j} \right) \frac{2}{(w - a_k)^2 (z - w)^2} \label{eq:gamma'_z_and_w_in_twice_remainder} \\
      && + {}& \eqref{eq:u_out} + \eqref{eq:u_in} + \eqref{eq:int_in_w_out_w_twice} + 2a_k.
    \end{flalign}
  \end{subequations}

  \begin{proof}[Derivation of \eqref{eq:collective_last_Airy}]
    \begin{enumerate*}[label=(\roman*)]
    \item
      \eqref{eq:4-fold_PV_v} is the sum of \eqref{eq:4-fold_PV_vw}--\eqref{eq:int_vu_z}, and
    \item 
      the sum of \eqref{eq:int_out_w_and_z_times_4} and \eqref{eq:gamma'_z_and_w_in_twice} is equal to the sum of \eqref{eq:int_out_w_and_in_w_z_times_4} and \eqref{eq:gamma'_z_and_w_in_twice_remainder}.
    \end{enumerate*}
  \end{proof}
\end{enumerate}

Now we can compute the contour integrals by saddle point method. Recall the contours $\sigma_{\standard}$ and $\gamma_{\standard}$ defined in \eqref{eq:defn_standard_gamma_sigma}). To facilitate the analysis, we fix the shape of the contours as $\sigma'_u = \sigma_{\standard}(-\sqrt{-x/3} - 1/\sqrt{-x})$, $\sigma'_z = \sigma_{\standard}(-\sqrt{-x/3} + 1/\sqrt{-x})$, $\gamma'_v = \gamma_{\standard}(\sqrt{-x/3} + 1/\sqrt{-x})$ and $\gamma'_w = \gamma_{\standard}(\sqrt{-x/3} - 1/\sqrt{-x})$, and call the $4$-fold contour consisting of them $\mathsf{\doubleX}$. Also when we consider contour integrals in the form of $\int^{C}_{\overline{C}}$ in \eqref{eq:collective_last_Airy} and \eqref{eq:third_transform} where $C \in \compC_+$, we fix the shape of the contour to be the part of $\sigma_{\standard}(a)$ between $\overline{C}$ and $C$, where $a=\Re(C) - \Im(C)/\sqrt{3}$. Below we also consider contour integrals denoted as $\int^C_{\overline{C}, \rightside}$, whose contour is the part of $\gamma_{\standard}(b)$ between $\overline{C}$ and $C$, where $b = \Re(C) + \Im(C)/\sqrt{3}$. (For symmetry, $\int^C_{\overline{C}}$ should be expressed $\int^C_{\overline{C}, \leftside}$. But since it occurs more often than $\int^C_{\overline{C}, \rightside}$, we omit the $\leftside$ subscript for simplicity.)
\begin{enumerate}[label=(\arabic*)]
\item
  The $4$-fold integral \eqref{eq:4-fold_PV_vw} can be estimated in the same way as the $2$-fold integral \eqref{eq:int_over_X_and_bar:X}. To perform the saddle point analysis, we denote the subsets of $\mathsf{\doubleX}$
  \begin{enumerate}[label=(\roman*)]
  \item
    $\mathsf{\doubleX}_1 = \{ u, v, z, w \in \mathsf{\doubleX} \mid \lvert u - \sqrt{-x} \ii \rvert < 1, \, \lvert v - \sqrt{-x} \ii \rvert < 1, \, \lvert z - \sqrt{-x} \ii \rvert < 1, \, \lvert w - \sqrt{-x} \ii \rvert < 1 \}$, 
  \item 
    $\mathsf{\doubleX}_2 = \{ u, v, z, w \in \mathsf{\doubleX} \mid \lvert u + \sqrt{-x} \ii \rvert < 1, \, \lvert v + \sqrt{-x} \ii \rvert < 1, \, \lvert z + \sqrt{-x} \ii \rvert < 1, \, \lvert w + \sqrt{-x} \ii \rvert < 1 \}$, 
  \item
    $\mathsf{\doubleX}_3 = \mathsf{\doubleX} \setminus (\mathsf{\doubleX}_1 \cup \mathsf{\doubleX}_2)$.
  \end{enumerate}
  By standard saddle point method, we have that the parts of integral \eqref{eq:4-fold_PV_vw} over $\mathsf{\doubleX}_1$ and $\mathsf{\doubleX}_2$ are both $\bigO(\lvert x \rvert^{-1})$, and the part of integral \eqref{eq:4-fold_PV_vw} over $\mathsf{\doubleX}_3$ is $o(\lvert x \rvert^{-1})$. Hence we have that \eqref{eq:4-fold_PV_vw} is $\bigO(\lvert x \rvert^{-1})$.
\item
  Both of the $3$-fold integrals \eqref{eq:int_vz_u} and \eqref{eq:int_vu_z} can be evaluated similarly, and they are $\bigO(\lvert x \rvert^{-1})$.
\item
  The $3$-fold integral \eqref{eq:u_out} can be written as the sum of
  \begin{subequations}
    \begin{align}
      & \frac{1}{(2\pi \ii)^3} \int_{\sigma'_z}\dd z \PV \int_{\gamma''_w} \dd w \int^{\Airyupperleft}_{\overline{\Airyupperleft}} \dd u \frac{e^{\frac{z^3}{3} - xz}}{e^{\frac{w^3}{3} - xw}} \left( \prod^{k - 1}_{j = 1} \frac{z - a_j}{w - a_j} \right)   \frac{1}{(z - w) (u - w)(z - u) (u - a_k)(w - a_k)} \label{eq:first_trivial} \\
      + {}& \frac{1}{(2\pi \ii)^2} \int^{\Airyupperleft}_{\overline{\Airyupperleft}} \dd u \int^{\sqrt{-x} \ii + e^{\pi \ii/6} \sqrt{-3/x}}_{-\sqrt{-x} \ii + e^{-\pi \ii/6} \sqrt{-3/x}} \dd z \frac{-1}{(u - z)^2 (u - a_k)(z - a_k)} \label{eq:last_added_1} \\
      + {}& \frac{1}{(2\pi \ii)^2} \int_{\sigma''_z} \dd z \int^{\Airyupperleft}_{\overline{\Airyupperleft}} \dd u \frac{e^{\frac{z^3}{3} - xz}}{e^{\frac{u^3}{3} - xu}} \left( \prod^{k - 1}_{j = 1} \frac{z - a_j}{u - a_j} \right) \frac{1}{(z - u)^2(u - a_k)^2} \label{eq:gamma''_z_and_u} \\
      + {}& \frac{1}{2\pi \ii} \int^{\Airyupperleft}_{\overline{\Airyupperleft}} \dd u \left( (u^2 - x) + \sum^{k - 1}_{j = 1} \frac{1}{u - a_j} \right) \frac{1}{(u - a_k)^2}, \label{eq:last_done}
    \end{align}
  \end{subequations}
  where the contours $\gamma''_w = \gamma_{\standard}(\sqrt{-x/3} + 2/\sqrt{-x})$ and $\sigma''_z = \sigma_{\standard}(-\sqrt{-x/3} - 2/\sqrt{-x})$. We note that $\sigma''_z$ is similar to $\sigma'_z$ but keeps the contour for $u$ in \eqref{eq:gamma''_z_and_u} on its right, and the contour $\gamma''_w$ is similar to $\gamma'_w$ but intersects $\sigma'_z$ at $\sqrt{-x} \ii + e^{\pi \ii/6} \sqrt{-3/x}$ and $-\sqrt{-x} \ii + e^{-\pi \ii/6} \sqrt{-3/x}$. On the other hand, the $3$-fold \eqref{eq:u_in} can be written as the sum of 
  \begin{subequations}
    \begin{align}
      & \frac{1}{(2\pi \ii)^3} \int_{\sigma'_u} \dd u \PV \int_{\gamma''_w} \dd w  \int^{\Airyupperright}_{\overline{\Airyupperright}} \dd z \frac{e^{\frac{u^3}{3} - xu}}{e^{\frac{w^3}{3} - xw}} \left( \prod^{k - 1}_{j = 1} \frac{u - a_j}{w - a_j} \right) \frac{1}{(u - z)(z - w) (u - w) (z - a_k)(w - a_k)} \\
      + {}& \frac{1}{(2\pi \ii)^2} \int^{\sqrt{-x} \ii + (1 + 3\sqrt{3} \ii))/\sqrt{-4x}}_{-\sqrt{-x} \ii + (1 - 3\sqrt{3} \ii))/\sqrt{-4x}} \dd u \int^{\Airyupperright}_{\overline{\Airyupperright}} \dd z \frac{-1}{(u - z)^2 (u - a_k)(z - a_k)} \label{eq:last_added_2_Airy} \\
      + {}& \frac{1}{(2\pi \ii)^2} \int_{\sigma'_u} \dd u \int^{\Airyupperright}_{\overline{\Airyupperright}} \dd z \frac{e^{\frac{u^3}{3} - xu}}{e^{\frac{z^3}{3} - xz}} \left( \prod^{k - 1}_{j = 1} \frac{u - a_j}{z - a_j} \right) \frac{1}{(u - z)^2(z - a_k)^2}, \label{eq:two_cancellation_early}
    \end{align}
  \end{subequations}
  where the contour $\sigma''_w = \sigma_{\standard}(-\sqrt{-x/3} - 2/\sqrt{-x})$ is the same as $\sigma''_z$ in \eqref{eq:gamma''_z_and_u}, and it intersects $\sigma'_u$ at $\sqrt{-x} \ii + (1 + 3\sqrt{3} \ii))/\sqrt{-4x}$ and $-\sqrt{-x} \ii + (1 - 3\sqrt{3} \ii))/\sqrt{-4x}$.
\item
  The $2$-fold integral \eqref{eq:int_out_w_and_in_w_z_times_4} can be written as the sum
  \begin{subequations}
    \begin{align}
      & \frac{1}{(2\pi \ii)^2} \iint_{\mathsf{X}} \dd v\dd z \frac{e^{\frac{z^3}{3} - xz}}{e^{\frac{w^3}{3} - xw}} \left( \prod^{k - 1}_{j = 1} \frac{z - a_j}{w - a_j} \right) \frac{-4}{(w - a_k)^2 (z - a_k)(z - w)} \\
      + {}& \frac{1}{2\pi \ii} \int^{\sqrt{-x}\ii}_{-\sqrt{-x}\ii}\dd z \frac{4}{(z - a_k)^3}, \label{eq:another_single_contour}
    \end{align}
  \end{subequations}
  by the same transform as \eqref{eq:double_contour_mean_Airy} is transformed into \eqref{eq:int_over_X_and_bar}, and the double contour $\mathsf{X}$ and the single contour in \eqref{eq:another_single_contour} are the same as in \eqref{eq:int_over_X_and_bar}. Note that the contour \eqref{eq:int_over_X_and_bar}  is to the right of all $a_k$'s while  the contour  in \eqref{eq:another_single_contour} is to the left of all $a_k$'s.
\item
  The $2$-fold integral \eqref{eq:gamma'_z_and_w_in_twice_remainder} can be written as the sum of  
  \begin{subequations}
    \begin{align}
      & \frac{1}{(2\pi \ii)^2} \int_{\sigma''_z} \dd z \int^{\Airyupperleft}_{\overline{\Airyupperleft}} \dd w \frac{e^{\frac{z^3}{3} - xz}}{e^{\frac{w^3}{3} -xw}} \left( \prod^{k - 1}_{j = 1} \frac{z - a_j}{w - a_j} \right) \frac{-1}{(w - a_k)^2 (z - w)^2} \label{eq:one_cancellation} \\
      + {}& \frac{1}{(2\pi \ii)^2} \int_{\sigma''_z} \dd z \int^{\Airyupperleft}_{\overline{\Airyupperleft}, \rightside} \dd w \frac{e^{\frac{z^3}{3} - xz}}{e^{\frac{w^3}{3} -xw}} \left( \prod^{k - 1}_{j = 1} \frac{z - a_j}{w - a_j} \right) \frac{1}{(w - a_k)^2 (z - w)^2} \\
      + {}& \frac{1}{(2\pi \ii)^2} \int_{\sigma''_z} \dd z \int^{\Airyupperright}_{\overline{\Airyupperright}} \dd w \frac{e^{\frac{z^3}{3} - xz}}{e^{\frac{w^3}{3} -xw}} \left( \prod^{k - 1}_{j = 1} \frac{z - a_j}{w - a_j} \right) \frac{-1}{(w - a_k)^2 (z - w)^2} \label{eq:two_cancellation} \\
      + {}& \frac{1}{(2\pi \ii)^2} \int_{\sigma''_z} \dd z \int^{\Airyupperright}_{\overline{\Airyupperright}, \rightside} \dd w \frac{e^{\frac{z^3}{3} - xz}}{e^{\frac{w^3}{3} -xw}} \left( \prod^{k - 1}_{j = 1} \frac{z - a_j}{w - a_j} \right) \frac{1}{(w - a_k)^2 (z - w)^2}, \label{eq:last_trivial}
    \end{align}
  \end{subequations}
  where $\sigma''_z$ is defined the same as $\sigma''_z$ in \eqref{eq:gamma''_z_and_u}. We note that \eqref{eq:one_cancellation} cancels with \eqref{eq:gamma''_z_and_u}, and \eqref{eq:two_cancellation} cancels with \eqref{eq:two_cancellation_early}.
\item
  The $1$-fold integral \eqref{eq:int_in_w_out_w_twice}, after taking into account of the negative sign in \eqref{eq:var_part_3}, cancels with \eqref{eq:single_int_gamma_L}.
\item
  \eqref{eq:last_done} can be evaluated similarly to \eqref{eq:variance_Airy_bar}, and it is $2\sqrt{-x}/\pi - a_k + \bigO(\lvert x \rvert^{-1/2})$.
\item
  \eqref{eq:last_added_1} and \eqref{eq:last_added_2_Airy} are $\bigO(\lvert x \rvert^{-1/2} \log \lvert x \rvert)$, and all the other  integrals from \eqref{eq:first_trivial} to \eqref{eq:last_trivial} not mentioned above, are $\bigO(\lvert x \rvert^{-1/2})$, as $x \to -\infty$. Since the evaluations are all by standard saddle point analysis, we omit the details.
\end{enumerate}
Hence we obtain the final proof of \eqref{eq:var_of_measure}.

\section{Eigenvector distribution for  GUE with external source} \label{s. GUE limit}
 In this section, we prove Theorem \ref{thm:main_thm_2}, and also Corollary \ref{cor:to_thm_2}, based on   Proposition \ref{prop:large_small_deviation_mu^k_N}. 
Recall the notations $\sigma_i$ in \eqref{20040810} and $\vec{x}_i$ in \eqref{20040811} for the eigenvalues and eigenvectors of $G_{\totalalpha} = G^{(N)}_{\totalalpha}$.  Let $\lambda_1 > \lambda_2 > \dotsb > \lambda_{N - 1}$ be the ordered eigenvalues of $G^{(N-1)}_{\totalalpha}$, which is obtained by removing the first column and first row of $G_{\totalalpha}$.

In order to prove Theorem \ref{thm:main_thm_2}, we define, analogous to \eqref{20041001}, \eqref{eq:defn_M(x)} and \eqref{eq:defn_F(x)}, the random measure $\mu_N = \mu^{(k)}_N$ with the random density function $\phi_N(x)$, (random) complementary distribution function $M_N(x)$ for $\mu_N$, and the mean of $M_N(x)$ such that
\begin{equation} \label{20041101}
  \phi_N(x) =
  \begin{cases}
    1 & \sigma_i < x \leq \lambda_{i - 1} \text{ for all }  i \in \llbracket 1,  N\rrbracket \\
    0 & \text{otherwise},
  \end{cases}
  \quad \text{and} \quad
  \begin{aligned}
    M_N(x) = {}& \mu_N((x, +\infty)) = \int^{\infty}_x \phi_N(t) \dd t, \\
    F_N(x) = {}& \E M_N(x).
  \end{aligned}
\end{equation}
Then for any $L \in \llbracket j, N \rrbracket$,
\begin{equation} \label{eq:defn_P_GUE_tilde}
  \lvert x_{j1} \rvert^2 = \prod^{j - 1}_{i = 1} \frac{\sigma_j - \lambda_i}{\sigma_j - \sigma_i} \prod^L_{i = j + 1} \frac{\sigma_j - \lambda_{i - 1}}{\sigma_j - \sigma_i} \mathcal{G}^{(k)}_{j, L}(\totalalpha; N), \quad \text{where} \quad \mathcal{G}^{(k)}_{j, L}(\totalalpha;N) := \prod^N_{i = L + 1} \frac{\sigma_j - \lambda_{i - 1}}{\sigma_j - \sigma_i}.
\end{equation}
Here the superscript $k$ in the notation $\mathcal{G}_{j,L}^{(k)}(\totalalpha;N)$ reminds us that the external source is of rank $k$; see $\alpha_1, \dotsc, \alpha_k$ in Assumption \ref{main assum}. Analogous to \eqref{eq:log_A_expr} and \eqref{19100302}, we write
\begin{equation} \label{eq:GUE_P_n_integral}
  \log \mathcal{G}_{j, L}^{(k)}(\totalalpha;N) =\int^{\lambda_L}_{\sigma_N} \frac{1}{\sigma_j - x} \dd M_N(x)= \int^{\sigma_L}_{\sigma_N} \frac{1}{\sigma_j - x} \dd M_N(x),
\end{equation}
where in the last step we used the fact that $\dd M_N(x) = 0$ on $(\lambda_L, \sigma_L]$ by definition. Analogous to Proposition \ref{prop:Airy_measure}, we have the following key technical result on the estimate of $F_N(x)$ and $\var M_N(x)$, and we state it for $x$ either in the bulk or on the edge of the distribution of $\sigma_1, \dotsc, \sigma_N$. The proof will be given in Section \ref{pf.lem_GUE}.
\begin{pro} \label{prop:large_small_deviation_mu^k_N}
  \begin{enumerate}
  \item (Bulk)
    Let $\epsilon > 0$ be any small (but fixed) constant and $N$ be large enough. For $x \in ((-2 + \epsilon)\sqrt{N}, 2\sqrt{N} - N^{-1/10})$ , we have
    \begin{gather}
      F_N(x) = E_N(x) + \bigO(N^{-\frac{1}{12}} (2\sqrt{N} - x)^{\frac{1}{2}}), \quad E_N(x) = \frac{1}{2\pi} \left( -\sqrt{4N - x^2} + x \arccos \frac{x}{2\sqrt{N}} \right) + \sqrt{N} - \frac{x}{2}, \label{eq:GUE_measure_mean} \\
      \var M_N(x) =  V_N(x) + \bigO(N^{-\frac{7}{24}} (2\sqrt{N} - x)^{\frac{1}{4}}), \quad V_N(x) = \frac{1}{\pi} \left( \sqrt{1 - \frac{x^2}{4N}} + \arccos  \frac{x}{2\sqrt{N}} \right). \label{eq:GUE_measure_variance}
    \end{gather}
  \item (Edge) \label{enu:prop:large_small_deviation_mu^k_N:2}
    Let $C$ be a large enough positive constant and $N$ be large enough. For $x \in [2\sqrt{N} - N^{-1/10}, 2\sqrt{N} - CN^{-1/6}]$, we have, with $\xi = N^{1/6} (x - 2\sqrt{N})$,
    \begin{align}
      F_N(x) = {}& N^{-1/6} \left( \frac{-\xi}{2} + \bigO(1) \right), \label{eq:GUE_measure_mean_local} \\
      \var M_N(x) = {}& N^{-1/3} \left( \frac{2}{\pi} \sqrt{-\xi} + \bigO( \lvert \xi \rvert^{1/4} ) \right). \label{eq:GUE_measure_variance_local}
    \end{align}
  \end{enumerate}
\end{pro}

To prove Theorem \ref{thm:main_thm_2}, we first define $\tilde{\sigma}_{i} = N^{1/6}(\sigma_i - 2\sqrt{N})$ and $\tilde{\lambda}_{i} = N^{1/6}(\lambda_i - 2\sqrt{N})$ for $i = 1, \dotsc, L$, given a fixed $L$. Then we use Lemma \ref{lem:Airy_limit}, and find that as a random vector, $(\tilde{\sigma}_{1}, \dotsc, \tilde{\sigma}_{L}, \tilde{\lambda}_{1}, \dotsc, \tilde{\lambda}_{L})$ converge weakly to $(\xi^{(k)}_1, \dotsc, \xi^{(k)}_L, \xi^{(k - 1)}_1, \dotsc, \xi^{(k - 1)}_L)$ as $N \to \infty$. Hence for any fixed $L > j$,
\begin{equation}
  \begin{split}
    \lim_{N \to \infty} \prod^{j - 1}_{i = 1} \frac{\sigma_j - \lambda_i}{\sigma_j - \sigma_i} \prod^L_{i = j + 1} \frac{\sigma_j - \lambda_{i - 1}}{\sigma_j - \sigma_i} = {}& \lim_{N \to \infty} \prod^{j - 1}_{i = 1} \frac{\tilde{\sigma}_{j} - \tilde{\lambda}_{i}}{\tilde{\sigma}_{j} - \tilde{\sigma}_{i}} \prod^L_{i = j + 1} \frac{\tilde{\sigma}_{j} - \tilde{\lambda}_{i - 1}}{\tilde{\sigma}_{j} - \tilde{\sigma}_{i}} \\
    = {}& \prod^{j - 1}_{i = 1} \frac{\xi^{(k)}_j - \xi^{(k - 1)}_i}{\xi^{(k)}_j - \xi^{(k)}_i} \prod^L_{i = j + 1} \frac{\xi^{(k)}_j - \xi^{(k - 1)}_{i - 1}}{\xi^{(k)}_j - \xi^{(k)}_i} = L^{-\frac{1}{3}} \Xi_{j}^{(k)}(\totala; L).
  \end{split}
\end{equation}
Using the convergence result in Theorem \ref{thm:main_thm_1}, we only need to show that for any $\epsilon_1, \epsilon_2 > 0$, there is a sufficiently large  $ L_{\epsilon_1, \epsilon_2}>0$ such that for any $L\geq L_{\epsilon_1, \epsilon_2}$,
\begin{equation} \label{eq:transformed_to_G_jL}
  \limsup_{N \to \infty} \Prob \left( \left\lvert \log \mathcal{G}_{j, L}^{(k)}(\totalalpha;N) - \left[ \log \frac{L^{1/3}}{N^{1/3}} + \log  \left( \frac{3\pi}{2} \right)^{1/3}  \right] \right\rvert > \epsilon_1 \right) < \epsilon_2.
\end{equation}
Here the deterministic term in \eqref{eq:transformed_to_G_jL} corresponds to the model-dependent constant in Theorem \ref{thm:main_thm_2}. Below we are going to prove the equivalent estimate
\begin{equation} \label{eq:transformed_to_G_jL_alt}
  \limsup_{N \to \infty} \Prob \left( \left\lvert \int^{\sigma_L}_{\sigma_N} \frac{1}{\sigma_j - x} \dd M_N(x) - \int^{2\sqrt{N} - N^{-1/6} \left( 3\pi L/2 \right)^{2/3}}_{-2\sqrt{N}} \frac{1}{2\sqrt{N} - x} \dd E_N(x) \right\rvert > \epsilon_1 \right) < \epsilon_2.
\end{equation}
The equivalence between \eqref{eq:transformed_to_G_jL_alt} and \eqref{eq:transformed_to_G_jL} is shown by
\begin{equation}
  \begin{split}
    & \int^{2\sqrt{N} - N^{-1/6} \left( 3\pi L/2 \right)^{2/3}}_{-2\sqrt{N}} \frac{\dd E_N(x)}{2\sqrt{N} - x} \\
    = {}& \frac{1}{2\pi} \int^{1 - \frac{1}{2} \left( 3\pi LN^{-1}/2 \right)^{2/3}}_{-1} \frac{\arccos y - \pi}{1 - y} \dd y \\
    = {}& \frac{1}{2\pi} \int^1_{-1} \frac{\arccos y}{1 - y} \dd y - \frac{1}{2} \int^{1 - \frac{1}{2} \left( 3\pi LN^{-1}/2 \right)^{2/3}}_{-1} \frac{\dd y}{1 - y} - \frac{1}{2\pi} \int^1_{1 - \frac{1}{2} \left( 3\pi LN^{-1}/2 \right)^{2/3}} \frac{\arccos y}{1 - y} \dd y \\
    = {}& \log 2 - \frac{1}{2} \left( \log 2 - \log \left( \frac{1}{2} \left( \frac{3\pi L}{2N} \right)^{2/3} \right) \right) + \bigO(N^{-\frac{1}{3}}) = \log \frac{L^{1/3}}{N^{1/3}} + \log \left( \frac{3\pi}{2} \right)^{1/3}  + \bigO(N^{-\frac{1}{3}}),
  \end{split}
\end{equation}
where we used the fact that 
\begin{align}
\frac{1}{2\pi} \int^1_{-1} \frac{\arccos y}{1 - y} \dd y = -\frac{2}{\pi} \int^{\pi/2}_0 \log(\sin \theta) \dd \theta=\log 2. \label{041902}
\end{align}
Here  the first step in (\ref{041902}) is a simple consequence of integration by parts.  The second step of (\ref{041902}) is a basic integral, and the result can be  found in \cite[4.224.3]{Gradshteyn-Ryzhik07}, for instance. Nevertheless, for the reader's convenience, we  illustrate the derivation as follows: First, one notices that $\int^{\pi/2}_0 \log(\sin \theta) \dd \theta=\int^{\pi/2}_0 \log(\cos \theta) \dd \theta$ and thus $\int^{\pi/2}_0 \log(\sin \theta) \dd \theta=\frac12\int^{\pi/2}_0 \log(\sin\theta\cos \theta) \dd \theta=\frac12\int_0^{\pi/2} \log(\sin(2\theta)) \dd \theta-\frac{\pi}{4} \log 2$. Then, by a simple change of variable $2\theta\mapsto \theta'$  and the fact $\int_{0}^{\pi} \log (\sin \theta')\dd \theta'= 2\int_{0}^{\pi/2} \log (\sin\theta')\dd \theta'$, one can conclude the result.

In order to prove \eqref{eq:transformed_to_G_jL_alt}, we will mainly rely on Proposition \ref{prop:large_small_deviation_mu^k_N}. There is an apparent obstacle: By part \ref{enu:lem:non-asy_est_GUE:1} of Lemma \ref{lem:non-asy_est_GUE}, $\sigma_j$ and $\sigma_N$ are close to $2\sqrt{N}$ and $-2\sqrt{N}$, respectively, but our estimates in Proposition \ref{prop:large_small_deviation_mu^k_N} do not cover the domain $[-2\sqrt{N}, (-2 + \epsilon)\sqrt{N})$. In order to handle the integral over $[-2\sqrt{N}, (-2 + \epsilon)\sqrt{N})$, we need the observation that the measure $\dd \mu_N(x) = -\dd M(x)$ is dominated by the Lebesgue measure, according to the definition in \eqref{20041101}. To be precise, we have that for any $\epsilon_3 > 0$, if $\epsilon < \epsilon_3$, $\sigma_j > (2 - \epsilon)\sqrt{N}$ and $\sigma_N \in ((-2 - \epsilon)\sqrt{N}, (-2 + \epsilon)\sqrt{N})$, then
\begin{equation}
  \left\lvert \int^{(-2 + \epsilon)\sqrt{N}}_{\sigma_N} \frac{1}{\sigma_j - x} \dd M_N(x) \right\rvert < \left\lvert \int^{(-2 + \epsilon)\sqrt{N}}_{\sigma_N} \frac{1}{\sigma_j - x} \dd x \right\rvert < \frac{2\epsilon \sqrt{N}}{(4 - 2\epsilon)\sqrt{N}} < \epsilon_3.
\end{equation}
We also note that by part \ref{enu:lem:non-asy_est_GUE:1} of Lemma \ref{lem:non-asy_est_GUE}, the assumptions $\sigma_j > (2 - \epsilon)\sqrt{N}$ and $\sigma_N \in ((-2 - \epsilon)\sqrt{N}, (-2 + \epsilon)\sqrt{N})$ hold with high probability as $N \to \infty$. We conclude that 
\begin{equation}
  \left\lvert \int^{(-2 + \epsilon)\sqrt{N}}_{\sigma_N} \frac{1}{\sigma_j - x} \dd M_N(x) \right\rvert < \epsilon_3 \quad \text{with high probability}.
\end{equation}
On the other hand, for the same $\epsilon_3$, and a possibly smaller $\epsilon > 0$, we have that for all large enough $N$
\begin{equation} \left\lvert \int^{(-2 + \epsilon)\sqrt{N}}_{-2\sqrt{N}} \frac{1}{2\sqrt{N} - x} \dd E_N(x) \right\rvert \leq \frac{1}{(4 - \epsilon) \sqrt{N}} \left\lvert \int^{(-2 + \epsilon)\sqrt{N}}_{-2\sqrt{N}} \dd E_N(x) \right\rvert = \frac{E_N((-2 + \epsilon)\sqrt{N}) - E_N(-2\sqrt{N})}{(4 - \epsilon) \sqrt{N}} < \epsilon_3,
\end{equation}
Hence, in order to see \eqref{eq:transformed_to_G_jL_alt}, it remains to show that given any $\epsilon > 0$, $\epsilon_1, \epsilon_2 > 0$, for all large enough (but fixed) $L$, such that for all large enough $N$, one has
\begin{equation} \label{eq:transformed_to_G_jL_alt3}
  \Prob \left( \left\lvert \int^{\sigma_L}_{(-2 + \epsilon)\sqrt{N}} \frac{1}{\sigma_j - x} \dd M_N(x) - \int^{2\sqrt{N} - N^{-1/6} \left( 3\pi L/2 \right)^{2/3}}_{(-2 + \epsilon)\sqrt{N}} \frac{1}{2\sqrt{N} - x} \dd E_N(x) \right\rvert > \epsilon_1 \right) < \epsilon_2.
\end{equation}
Next, we observe that by \eqref{eq:GUE_measure_mean} and \eqref{eq:GUE_measure_mean_local}, for any $\epsilon_3 > 0$, there is a small enough $\epsilon > 0$, such that for all large enough $N$
\begin{equation}
  \left\lvert \int^{2\sqrt{N} - N^{-1/6} \left( 3\pi L/2 \right)^{2/3}}_{(-2 + \epsilon)\sqrt{N}} \frac{1}{2\sqrt{N} - x} \dd (E_N(x) - F_N(x)) \right\rvert < \epsilon_3.
\end{equation}
Hence,  instead of \eqref{eq:transformed_to_G_jL_alt3}, we only need to show, under the same assumption,
\begin{equation} \label{eq:transformed_to_G_jL_alt2}
  \Prob \left( \left\lvert \int^{\sigma_L}_{(-2 + \epsilon)\sqrt{N}} \frac{1}{\sigma_j - x} \dd M_N(x) - \int^{2\sqrt{N} - N^{-1/6} \left( 3\pi L/2 \right)^{2/3}}_{(-2 + \epsilon)\sqrt{N}} \frac{1}{2\sqrt{N} - x} \dd F_N(x) \right\rvert > \epsilon_1 \right) < \epsilon_2.
\end{equation}
Theorem \ref{thm:main_thm_2} will then follow if \eqref{eq:transformed_to_G_jL_alt2} is proved.

\begin{proof}[Proof of \eqref{eq:transformed_to_G_jL_alt2}]
  First, we set the  integer
  \begin{equation}
    N_0 = \left\lfloor \frac{2}{3\pi} N^{\frac{1}{10}} \right\rfloor.
  \end{equation}
  We note that by part \ref{enu:lem:non-asy_est_GUE:2} of Lemma \ref{lem:non-asy_est_GUE}, $\sigma_{N_0}=2\sqrt{N} - N^{-1/10} + \bigO(N^{-8/75})$ with high probability. To prove \eqref{eq:transformed_to_G_jL_alt2}, it suffices to show the following two inequalities with the same conditions:
  \begin{align}
    \Prob \left( \left\lvert \int^{\sigma_L}_{\sigma_{N_0}} \frac{1}{\sigma_j - x} \dd M_N(x) - \int^{2\sqrt{N} - N^{-1/6} \left( 3\pi L/2 \right)^{2/3}}_{2\sqrt{N} - N^{-1/10}} \frac{1}{2\sqrt{N} - x} \dd F_N(x) \right\rvert > \epsilon_1 \right) < {}& \epsilon_2, \label{eq:almost_Airy_part} \\
    \Prob \left( \left\lvert \int^{\sigma_{N_0}}_{(-2 + \epsilon)\sqrt{N}} \frac{1}{\sigma_j - x} \dd M_N(x) - \int^{2\sqrt{N} - N^{-1/10}}_{(-2 + \epsilon)\sqrt{N}} \frac{1}{2\sqrt{N} - x} \dd F_N(x) \right\rvert > \epsilon_1 \right) < {}& \epsilon_2. \label{eq:not_Airy_part}
  \end{align}

  The proof of \eqref{eq:almost_Airy_part} is similar to that of \eqref{eq:alt_technical_main_thm_1} after a change of variable $\xi = N^{1/6} (x - 2\sqrt{N})$. Actually, we can show the following stronger estimate, which is an analogue of \eqref{eq:alt_technical_main_thm_1}:
  \begin{equation} \label{almost_Airy_stronger}
    \Prob \left( \sup_{L < n \leq N^{1/10}} \left\lvert \int^{\sigma_L}_{\sigma_n} \frac{1}{\sigma_j - x} \dd M_N(x) - \int^{2\sqrt{N} - N^{-1/6} \left( 3\pi L/2 \right)^{2/3}}_{2\sqrt{N} - N^{-1/6} \left( 3\pi n/2 \right)^{2/3}} \frac{1}{2\sqrt{N} - x} \dd F_N(x) \right\rvert > \epsilon_1 \right) < \epsilon_2,
  \end{equation}
  where the supremum for $n$ is bounded by $L$ and $N^{1/10}$, rather than in \eqref{eq:alt_technical_main_thm_1} for all $n$ greater than $m$. Note that the proof of \eqref{eq:alt_technical_main_thm_1} only relies on three ingredients as explained in Remark \ref{rmk:Airy_ingredients}, and all of them have their counterparts in the proof of (\ref{almost_Airy_stronger}): the counterpart of Proposition \ref{prop:Airy_measure} is part \ref{enu:prop:large_small_deviation_mu^k_N:2} of Proposition \ref{prop:large_small_deviation_mu^k_N}; the counterpart of Lemma \ref{lem:non-asy_est_Airy} (the rigidity of Airy process) is parts \ref{enu:lem:non-asy_est_GUE:1} and \ref{enu:lem:non-asy_est_GUE:2} of Lemma \ref{lem:non-asy_est_GUE} (the rigidity of deformed GUE eigenvalues); the dominance by Lebesgue measure is the same for $\mu$ and $\mu_N$. The only slight difference is that part \ref{enu:lem:non-asy_est_GUE:2} of Lemma \ref{lem:non-asy_est_GUE} is only valid for $n$ less than $N^{1/10}$ since the Airy process approximation of the GUE minor process with external source can only be extended to such an intermediate regime.
 
  The proof of \eqref{eq:not_Airy_part} again follows the idea of the proof of \eqref{eq:alt_technical_main_thm_1}, with some necessary modifications. We describe the proof parallel to that of \eqref{eq:alt_technical_main_thm_1}, and give detailed explanations on the modifications. Analogous to \eqref{022102}, we denote the events
    \begin{equation} \label{0221022}
      \widehat{ \Omega}^{(k)}_{j, m} := \left\{ \omega \; \middle\vert \; \lvert N^{\frac{1}{6}}(\sigma_j - 2\sqrt{N}) \rvert \leq  m^{\frac{3}{5}}, \; \left\lvert \sigma_m - \Upsilon_m \right\rvert \leq N^{-\frac{1}{6}} m^{\frac{3}{5}} \right \},
    \end{equation}
    where $\Upsilon_m$ is defined in \eqref{eq:defn_Upsilon}.  
  \begin{enumerate}
  \item 
    Analogous to \eqref{eq:statement_1_proof_Airy}, we have the estimate that if $N$ is large enough, then
    \begin{equation}
      \left\lvert \int^{\sigma_{N_0}}_{(-2 + \epsilon)\sqrt{N}} \frac{1}{\sigma_j - x} \dd F_N(x)-\int^{2\sqrt{N} - N^{-1/10}}_{(-2 + \epsilon)\sqrt{N}} \frac{1}{2\sqrt{N} - x} \dd F_N(x) \right\rvert  \Id(\widehat{\Omega}^{(k)}_{j, N_0})< \frac{2\epsilon_1}{3}.
    \end{equation}
  \item 
    Next, analogous to \eqref{19101501}, we show
    \begin{equation} \label{eq:anal_19101501}
      \Prob \left( \left\lvert \int^{\sigma_{N_0}}_{(-2 + \epsilon)\sqrt{N}} \frac{1}{\sigma_j - x} \dd (M_N(x) - F_N(x)) \Id(\widehat{\Omega}^{(k)}_{j, N_0}) \right\rvert > \frac{\epsilon_1}{3} \right) < \frac{\epsilon_2}{2}.
    \end{equation}
    Using an integration by parts decomposition like \eqref{eq:IbP}, we find that it suffices to show
    \begin{align}
      \Prob \left( \left\lvert \frac{M_N(\sigma_{N_0}) - F_N(\sigma_{N_0})}{\sigma_j - \sigma_{N_0}} \Id(\widehat{\Omega}^{(k)}_{j, N_0}) \right\rvert > \frac{\epsilon_1}{9} \right) < {}& \frac{\epsilon_2}{6}, \label{eq:BC_app_1_like} \\
      \Prob \left( \left\lvert \frac{M_N((-2 + \epsilon)\sqrt{N}) - F_N((-2 + \epsilon)\sqrt{N})}{\sigma_j - (-2 + \epsilon)\sqrt{N}} \Id(\widehat{\Omega}^{(k)}_{j, N_0}) \right\rvert > \frac{\epsilon_1}{9} \right) < {}& \frac{\epsilon_2}{6}, \label{eq:BC_app_1_like_right} \\
      \Prob \left( \left\lvert \int^{\sigma_{N_0}}_{(-2 + \epsilon) \sqrt{N}} \frac{M_N(x) - F_N(x)}{(\sigma_j - x)^2} \dd x \Id(\widehat{\Omega}^{(k)}_{j, N_0}) \right\rvert > \frac{\epsilon_1}{9} \right) < {}& \frac{\epsilon_2}{6}. \label{eq:BC_app_2_like}
    \end{align}
    We first see that \eqref{eq:BC_app_1_like_right} only needs that $M_N((-2 + \epsilon)\sqrt{N})$ is close to $F_N((-2 + \epsilon)\sqrt{N})$, and it is a direct consequence of \eqref{eq:GUE_measure_variance} at $x = (-2 + \epsilon)\sqrt{N}$, by using Markov inequality. Estimate \eqref{eq:BC_app_1_like} can be verified in the same way as for \eqref{eq:BC_app_1}. On the other hand, to verify \eqref{eq:BC_app_2_like}, it suffices to show that
    \begin{equation} \label{eq:BC_app_2_like_alt}
      \Prob \left( \int^{2\sqrt{N} - N^{-1/10}}_{(-2 + \epsilon) \sqrt{N}} \frac{\lvert M_N(x) - F_N(x) \rvert}{(2\sqrt{N} - x)^2} \dd x > \frac{\epsilon_1}{9} \right) < \frac{\epsilon_2}{6},
    \end{equation}
    by the arguement like that around \eqref{eq:eq:BC_app_2_middle} and \eqref{eq:part_3_as_Airy_simplify}.  More specifically, to prove \eqref{eq:BC_app_2_like_alt}, we introduce
    \begin{equation}
      \ell_i = 2\sqrt{N} - iN^{-\frac{1}{6}}
  \end{equation}
  and analogous to \eqref{eq:defn_ABC_k_1k_1} we set
  \begin{equation} 
      \begin{gathered}
        \widehat{A}_{k_1, k_2} = \sum^{k_2}_{i = k_1} \frac{\lvert M_N(\ell_i) - F_N(\ell_i) \rvert}{(i - 1)i}, \quad \widehat{B}_{k_1, k_2} = \sum^{k_2}_{i = k_1} \frac{\lvert M_N(\ell_i) - F_N(\ell_i) \rvert}{i(i + 1)}, \\
        \widehat{C}_{k_1, k_2} = 2 \sum^{k_2 - 1}_{i = k_1 + 1} \frac{F_N(\ell_i)}{(i - 1)i(i + 1)} - \frac{F_N(\ell_{k_1})}{k_1(k_1 + 1)} + \frac{F_N(\ell_{k_2})}{(k_2 - 1)k_2}.
      \end{gathered}
  \end{equation}
  Then analogous to \eqref{eq:ineq:ABC} and \eqref{eq:part_3_as_Airy_simplify_more}, we have (without loss of generality, assuming $N^{1/15}$ and $(4 - \epsilon) N^{2/3}$ are integers)
  \begin{equation}
    \int^{2\sqrt{N} - N^{-1/10}}_{(-2 + \epsilon) \sqrt{N}} \frac{\lvert M_N(x) - F_N(x) \rvert}{(2\sqrt{N} - x)^2} \dd x\leq N^{\frac{1}{6}} (A_{N^{1/15}, (4 - \epsilon)N^{2/3}} + B_{N^{1/15}, (4 - \epsilon)N^{2/3}} + 2C_{N^{1/15}, (4 - \epsilon)N^{2/3}}),
  \end{equation}
  and need only to prove
  \begin{equation}
    \Prob \left( N^{\frac{1}{6}} A_{N^{1/15}, (4 - \epsilon)N^{2/3}} > \frac{\epsilon_1}{27} \right) < \frac{\epsilon_2}{12}, \quad \Prob \left( N^{\frac{1}{6}} B_{N^{1/15}, (4 - \epsilon)N^{2/3}} > \frac{\epsilon_1}{27} \right) < \frac{\epsilon_2}{12}, \quad 2 N^{\frac{1}{6}}C_{N^{1/15}, (4 - \epsilon)N^{2/3}} < \frac{\epsilon_1}{27},
  \end{equation}
  As $N$ is large enough, it is straightforward to  show the last inequality for $C_{N^{1/15}, (4 - \epsilon)N^{2/3}}$. Also, like \eqref{062601}, we have that for large enough $N$,
  \begin{equation}
    \begin{aligned}
       \Prob \left( A_{N^{1/15}, (4 - \epsilon)N^{2/3}} > \frac{\epsilon_1}{27} N^{-\frac{1}{6}} \right) \leq {}& \Prob \left( A_{N^{1/15}, (4 - \epsilon)N^{2/3}} > \frac{\epsilon_1}{27} N^{-\frac{1}{6}} \right) \\
       \leq {}& \Prob \left( \sum^{(4 - \epsilon)N^{2/3}}_{i = N^{1/15}} \frac{\lvert M_N(\ell_i) - F_N(\ell_i) \rvert}{i^2} > \frac{\epsilon_1}{54} N^{-\frac{1}{6}} \right) \\
       \leq {}& \Prob \left( \bigcup^{(4 - \epsilon)N^{2/3}}_{i = N^{1/15}} \left\{ \frac{\lvert M_N(\ell_i) - F_N(\ell_i) \rvert}{i^2} > i^{-\frac{6}{5}} N^{-\frac{1}{6}} \right\} \right) \\
       \leq {}& \sum^{(4 - \epsilon)N^{2/3}}_{i = N^{1/15}} \Prob \left( \frac{\lvert M_N(\ell_i) - F_N(\ell_i) \rvert}{i^2} > i^{-\frac{6}{5}} N^{-\frac{1}{6}} \right) \\
       \leq {}& \sum^{(4 - \epsilon)N^{2/3}}_{i = N^{1/15}} i^{-\frac{11}{10}} < \frac{\epsilon_2}{12}.
    \end{aligned}
  \end{equation}
  We thus finish the proof of \eqref{eq:BC_app_2_like_alt}, and then \eqref{eq:BC_app_2_like}, and finally finish the proof of \eqref{eq:anal_19101501}.
\item
  At last, we need to show that the event $\widehat{\Omega}^{(k)}_{j, N_0}$ satisfies that $\Prob(\widehat{\Omega}^{(k)}_{j, N_0}) > 1 - \epsilon_2/2$ if $N$ is large enough. (This is analogous to \eqref{022101} but is much weaker.) This follows directly from the estimates of $\sigma_j$ and $\sigma_{N_0}$ in parts \ref{enu:lem:non-asy_est_GUE:1} and \ref{enu:lem:non-asy_est_GUE:2} of Lemma \ref{lem:non-asy_est_GUE}.
  \end{enumerate}
  Hence, we conclude the proof of Theorem \ref{thm:main_thm_2}. 
\end{proof}

Next, we prove Corollary \ref{cor:to_thm_2}.
\begin{proof}
  Let $U=(u_{ij})\in \mathbb{C}^{(N-k)\times (N-k)}$ be  any unitary matrix independent of $G_{\totalalpha}$, then by the unitary invariance of GUE, the random matrix $(I_k\oplus U) G_{\totalalpha}(I_k\oplus U^*)$ has the same distribution as $G_{\totalalpha}$. Hence the eigenvector $\vec{x}_j = (x_{j1}, \dotsc, x_{jN})^{\top}$ of $G_{\totalalpha}$ associated with $\sigma_j$, which is a random unit vector, satisfies that fixing $x_{j1}, \dotsc, x_{jk}$, the conditional distribution of the truncated random vector $(x_{j, k + 1}, \dotsc, x_{jN})^{\top}$ is unitarily invariant.

Furthermore, we have $\sum_{i=k+1}^N|x_{ji}|^2=1-O(kN^{-1/3})$ with high probability by Theorem \ref{thm:main_thm_2}. So the distribution of $x_{ji}$ is approximated by that of a component of a Haar distributed complex unit random vector in $\compC^{N - k}$, as $N \to \infty$. Combining the above facts we can conclude the proof of Corollary \ref{cor:to_thm_2} easily. 
\end{proof}

\section{Analysis of random measure $\mu_N$: proof of Proposition \ref{prop:large_small_deviation_mu^k_N}} \label{pf.lem_GUE}

In this section, we prove Proposition \ref{prop:large_small_deviation_mu^k_N}. To show \eqref{eq:GUE_measure_mean} and \eqref{eq:GUE_measure_mean_local}, we use  the method similar to the proof of \eqref{eq:mean_variance_measure_Airy}. We consider, analogous to \eqref{eq:defn_N_x_Airy}, the random variable (with notation abused)
\begin{equation}
  N_x = \lvert \mathcal{I}_x \rvert - \lvert \mathcal{J}_x \rvert \quad \text{where} \quad \mathcal{I}_x := \{i \in \natN \mid \sigma_i \in (x, +\infty)\}, \quad \mathcal{J}_x := \{ i \in \natN \mid \lambda_i \in (x, +\infty)\}. 
\end{equation}
By the interlacing property, we note that $N_x$ is a Bernoulli random variable. Then
\begin{equation}
  \E N_x = \Prob(N_x = 1).
\end{equation}
We also consider, analogous to \eqref{eq:defn_S_x_Airy} the random variable (with notation abused)
\begin{equation}
  S_x = -\sum_{i\in \mathcal{I}_x} \sigma_i + \sum_{i\in \mathcal{J}_x} \lambda_i.
\end{equation}
We observe that if $\sigma_1 \leq x$, then $N_x = S_x = M_N(x) = 0$. Under the condition that $\sigma_1 > x$, we have that if $N_x = 1$, then $S_x = M_N(x) - \sigma_1$, otherwise $S_x = M_N(x) - \sigma_1 + x$. Similar to \eqref{eq:mu_S_N_Airy}, we conclude that
\begin{equation} \label{eq:expr_mu^k_N}
  M_N(x) = S_x + xN_x + (\sigma_1 - x) \Id(\sigma_1 > x).
\end{equation}

For the second term in \eqref{eq:expr_mu^k_N} that only involves $\sigma_1$, by part \ref{enu:lem:non-asy_est_GUE:1} of Lemma \ref{lem:non-asy_est_GUE}, we have the following estimates that hold for all $x \in \realR$:
\begin{equation} \label{eq:mean_and_var_GUE_large}
  \E[(\sigma_1 - x) \Id(\sigma_1 > x)] = 2\sqrt{N} - x + \bigO(N^{-\frac{1}{6}}), \quad \var[(\sigma_1 - x) \Id(\sigma_1 > x)] = \bigO(N^{-\frac{1}{3}}).
\end{equation}
Then by the linearity of expectation and \eqref{rmk:general_CS}, the mean and variance estimates in Proposition \ref{prop:large_small_deviation_mu^k_N} follow from \eqref{eq:mean_and_var_GUE_large} and the four estimates \eqref{eq:GUE_measure_mean_alt}--\eqref{eq:GUE_measure_variance_local_alt} below whose proofs will be given later. First, we note that for $x \in ((-2 + \epsilon)\sqrt{N}, 2\sqrt{N} - N^{-1/10})$ 
\begin{align}
  \E(S_x + xN_x) = {}& E_N(x) - (2\sqrt{N} - x) + \bigO(N^{-\frac{1}{12}} (2\sqrt{N} - x)^{\frac{1}{2}}), \label{eq:GUE_measure_mean_alt} \\
  \var (S_x + xN_x) = {}& V_N(x) + \bigO(N^{-\frac{1}{3}}), \label{eq:GUE_measure_variance_alt}
\end{align}
and for $x \in [2\sqrt{N} - N^{-1/10}, 2\sqrt{N} - CN^{-1/6}]$, with $x = 2\sqrt{N} + N^{-1/6} \xi$, then we have the uniform convergence
\begin{align}
  N^{\frac{1}{6}} \E(S_x + xN_x) = {}& \frac{\xi}{2} -\frac{2}{\pi} a_k \sqrt{-\xi} + \bigO(1), \label{eq:GUE_measure_mean_local_alt} \\
  N^{\frac{1}{3}} \var (S_x + xN_x) = {}& \frac{2}{\pi} \sqrt{-\xi} + \bigO( \lvert \xi \rvert^{\frac{1}{4}} ). \label{eq:GUE_measure_variance_local_alt}
\end{align}

Below we prove \eqref{eq:GUE_measure_mean_alt}, \eqref{eq:GUE_measure_variance_alt} and \eqref{eq:GUE_measure_mean_local_alt}. The proof of \eqref{eq:GUE_measure_variance_local_alt} is similar to that of \eqref{eq:GUE_measure_variance_alt} and we only give a sketch.

\paragraph{Proof of \eqref{eq:GUE_measure_mean_alt}}

To faciliate the saddle point analysis below, we define several types of contours. First, we recall that $\alpha_1, \dotsc, \alpha_k$ are defined by $a_1, \dotsc, a_k$ in Assumption \ref{main assum}. We take a constant $r \in [0, +\infty)$ such that $r$ is bigger than all $a_1, \dotsc, a_k$. Then we let $\Gamma^{\round}_{\standard, N}(a)$ be the positively oriented boundary of the open set $\{ z \in \compC \mid \lvert z \rvert < \sqrt{N} + a \} \cup \{ z \in \compC \mid \lvert z - \sqrt{N} \rvert < N^{1/6} r \}$ where $a \geq 0$, such that it is almost the circle centred at $0$ with radius $\sqrt{N} + a$ and it encloses all $\alpha_1, \dotsc, \alpha_k$. (Actually, if $r = 0$, then $\Gamma$ is the circle centred at $0$ with radius $\sqrt{N} + a$.) We also define for any $b \in \realR$ the upward vertical contour
\begin{equation}
  \Sigma^{|}_{\standard, N}(b) = \{ b + \ii t \mid -\infty < t < +\infty \}.
\end{equation}
Next, for any $\theta \in (0, \pi/6)$, we define the positively oriented contour
\begin{multline}
  \Gamma^{\corner}_{\standard, N}(\theta, a) =  \Big\{ (\sqrt{N} + a) e^{(\frac{\pi}{3} - \theta) \ii} + e^{\frac{2\pi}{3} \ii} t \mid -\frac{2}{\sqrt{3}} (\sqrt{N} + a) \sin \left( \frac{\pi}{3} - \theta \right) \leq t \leq 0 \Big\} \\
  \cup \Big\{ (\sqrt{N} + a) e^{(\frac{5\pi}{3} + \theta) \ii} + e^{\frac{\pi}{3} \ii} t \mid 0 \leq t \leq \frac{2}{\sqrt{3}} (\sqrt{N} + a) \sin \left( \frac{\pi}{3} - \theta \right) \Big\} \cup \Big\{ (\sqrt{N} + a)e^{\ii t} \mid \frac{\pi}{3} - \theta \leq t \leq \frac{5\pi}{3} + \theta \Big \}.
\end{multline}
We also define for any $b < \sqrt{N}$ the upward infinite polygonal contour
\begin{multline}
  \Sigma^{\zig}_{\standard, N}(b) = \Big\{ b + e^{\frac{\pi}{3} \ii} t \mid 0 \leq t \leq 2(\sqrt{N} - b) + \sqrt{N} \Big\} \cup \Big\{ b + e^{\frac{2\pi}{3} \ii} t \mid -\sqrt{N} -  2(\sqrt{N} - b) \leq t \leq 0 \Big\} \\
  \cup \Big\{ \frac{3}{2}\sqrt{N} + \ii t \mid \lvert t \rvert \geq \sqrt{3} (\frac{3}{2}\sqrt{N} - b) \Big\}. 
\end{multline}
See Figures \ref{fig:Gamma_Sigma_standard_1} and \ref{fig:Gamma_Sigma_standard_2} for their shapes.
\begin{figure}[htb]
  \begin{minipage}[t]{0.3\linewidth}
    \centering
    \includegraphics{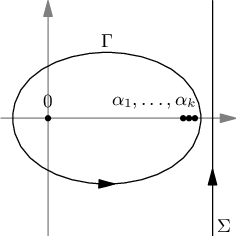}
    \caption{Contours $\Gamma$ and $\Sigma$.}
    \label{fig:Gamma_Sigma}
  \end{minipage}
  \hspace{\stretch{1}}
  \begin{minipage}[t]{0.3\linewidth}
    \centering
    \includegraphics{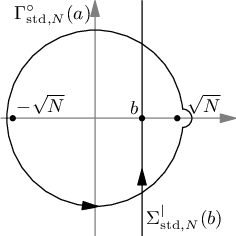}
    \caption{Shapes of $\Gamma^{\round}_{\standard, N}(a)$ and $\Sigma^{|}_{\standard, N}(b)$.}
    \label{fig:Gamma_Sigma_standard_1}
  \end{minipage}
  \hspace{\stretch{1}}
  \begin{minipage}[t]{0.3\linewidth}
    \centering
    \includegraphics{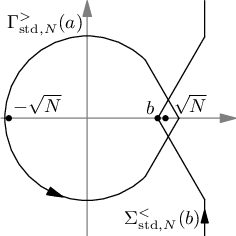}
    \caption{Shapes of $\Gamma^{\corner}_{\standard, N}(a)$ and $\Sigma^{\zig}_{\standard, N}(b)$.}
    \label{fig:Gamma_Sigma_standard_2}
  \end{minipage}
\end{figure}

We have that
\begin{equation} \label{eq:mean_GUE_by_kernel}
  \begin{split}
    \E (S_x + xN_x) = {}&  \E \left( \sum_{i \in \mathcal{J}_x} \lambda_i \right) - \E \left( \sum_{i \in \mathcal{I}_x} \sigma_i \right)  + x\E|\mathcal{I}_x| - x\E|\mathcal{J}_x| \\
    = {}& \int^{\infty}_x t K^{1, 1}_{\GUEalpha}(t, t)\dd t - \int^{\infty}_x t K^{0, 0}_{\GUEalpha}(t, t)\dd t 
     + x \left( \int^{\infty}_x K^{0, 0}_{\GUEalpha}(t, t)\dd t - \int^{\infty}_x K^{1, 1}_{\GUEalpha}(t, t)\dd t \right) \\
    = {}& -\int^{+\infty}_x (t - x) \left( K^{0, 0}_{\GUEalpha}(t, t) - K^{1, 1}_{\GUEalpha}(t, t) \right)\dd t.
  \end{split}
\end{equation}
In light of \eqref{20041005}, one has 
\begin{equation}
  K^{0, 0}_{\GUEalpha}(t, t) - K^{1, 1}_{\GUEalpha}(t, t) = \frac{1}{(2\pi \ii)^2} \int_{\Sigma}\dd z \int_{\Gamma} \dd w \frac{e^{\frac{z^2}{2} - tz} z^{N - k}}{e^{\frac{w^2}{2} - tw} w^{N - k}} \left( \prod^k_{j = 2} \frac{z - \alpha_j}{w - \alpha_j} \right) \frac{1}{w - \alpha_1},
\end{equation}
which implies
\begin{equation} \label{eq:double_contour_GUE_mean}
  \begin{split}
    \E (S_x + xN_x) = {}& \frac{-1}{(2\pi \ii)^2} \int_{\Sigma}\dd z \int_{\Gamma} \dd w \frac{e^{\frac{z^2}{2}} z^{N - k}}{e^{\frac{w^2}{2}} w^{N - k}} \left( \prod^k_{j = 2} \frac{z - \alpha_j}{w - \alpha_j} \right) \frac{1}{w - \alpha_1} \int^{\infty}_x e^{-(z - w)t} (t - x)\dd t \\
    = {}& \frac{-1}{(2\pi \ii)^2} \int_{\Sigma}\dd z \int_{\Gamma} \dd w \frac{e^{\frac{z^2}{2} - xz} z^{N - k}}{e^{\frac{w^2}{2} - xw} w^{N - k}} \left( \prod^k_{j = 2} \frac{z - \alpha_j}{w - \alpha_j} \right) \frac{1}{(w - \alpha_1)(z - w)^2}.
  \end{split}
\end{equation}
By some standard residue calculation, we have that
\begin{subequations} \label{eq:int_over_Phi_and_bar}
  \begin{align}
    & \E (S_x + xN_x) \notag \\
    = {}& \frac{-1}{(2\pi \ii)^2} \iint_{\Phi} \dd z \dd w \frac{e^{\frac{z^2}{2} - xz} z^{N - k}}{e^{\frac{w^2}{2} - xw} w^{N - k}} \left( \prod^k_{j = 2} \frac{z - \alpha_j}{w - \alpha_j} \right) \frac{1}{(w - \alpha_1)(z - w)^2} \label{eq:int_over_Phi_and_bar:Phi} \\
    & + \frac{-1}{2\pi \ii} \int^{\upperintersect}_{\lowerintersect} \frac{w - x + (N - k)/w}{w - \alpha_1} + \sum^k_{j = 2} \frac{1}{(w - \alpha_j)(w - \alpha_1)} \dd w, \label{eq:int_over_Phi_and_bar:bar}
  \end{align}
\end{subequations}
where
\begin{enumerate*}[label=(\roman*)]
\item
  the contour $\Phi$ means that the $\Sigma$ contour cuts into the $\Gamma$ contour, and divide its interior into two parts, such that they intersect at the two saddle points $\frac{1}{2}(x + \sqrt{4N - x^2}\ii)$ and $\frac{1}{2}(x - \sqrt{4N - x^2}\ii)$, which are the $\upperintersect$ and $\lowerintersect$ respectly in \eqref{eq:int_over_Phi_and_bar:bar}, and all $\alpha_1, \dotsc, \alpha_k$ are in the right part,
\item
  the integral in \eqref{eq:int_over_Phi_and_bar:Phi} is understood as the Cauchy principal value, and 
\item 
  the contour of integral in \eqref{eq:int_over_Phi_and_bar:bar} goes by the right of $0$ and all $\alpha_1, \dotsc, \alpha_k$. 
\end{enumerate*}
We then have, if all $\alpha_1, \dotsc, \alpha_k$ are distinct, then
\begin{equation} \label{eq:first_moment_GUE}
  \begin{split}
    \eqref{eq:int_over_Phi_and_bar:bar} = {}& \frac{-1}{2\pi \ii} \int^{\frac{1}{2}(x + \sqrt{4N - x^2}\ii)}_{\frac{1}{2}(x - \sqrt{4N - x^2}\ii)} 1 + \left( \alpha_1 - x + \frac{N - k}{\alpha_1} \right) \frac{1}{w - \alpha_1} - \frac{N - k}{\alpha_1} \frac{1}{w} \\
    & \phantom{\smash{\frac{-1}{2\pi \ii} \int^{\frac{1}{2}(x + \sqrt{4N - x^2}\ii)}_{\frac{1}{2}(x - \sqrt{4N - x^2}\ii)}}} + \sum^k_{j = 1} \frac{1}{\alpha_1 - \alpha_j} \left( \frac{1}{w - \alpha_1} - \frac{1}{w - \alpha_j} \right) \dd w \\
    = {}& \frac{1}{\pi} \left[-\frac{\sqrt{4N - x^2}}{2} - \left( \alpha_1 - x + \frac{N - k}{\alpha_1} \right) \arccot\frac{x - 2\alpha_1}{\sqrt{4N - x^2}} + \frac{N - k}{\alpha_1} \arccos \frac{x}{2\sqrt{N}}
      \vphantom{\sum^k_{j = 2} \frac{1}{\alpha_1 - \alpha_j} \left( \arccot \frac{2\alpha_1 - x}{\sqrt{4N - x^2}} - \arccot \frac{2\alpha_j - x}{\sqrt{4N - x^2}} \right)} \right. \\
    & \phantom{\frac{1}{\pi}} \left. - \sum^k_{j = 2} \frac{1}{\alpha_1 - \alpha_j} \left( \arccot \frac{x - 2\alpha_1}{\sqrt{4N - x^2}} - \arccot \frac{x - 2\alpha_j}{\sqrt{4N - x^2}} \right) \right] \\
    = {}& \frac{1}{\pi} \left[ -\frac{\sqrt{4N - x^2}}{2} - \frac{1}{2} \left( \alpha_1 - x - \frac{N - k}{\alpha_1} \right) \arccos \frac{x}{2\sqrt{N}} - \frac{\pi}{2} \left( \alpha_1 - x + \frac{N - k}{\alpha_1} \right)
    \vphantom{- \sum^k_{j = 2} \frac{N^{-1/6}}{a_k - a_{k + 1 - j}} \left( \arccot \frac{x - 2\sqrt{N} - 2N^{1/6} a_k}{\sqrt{4N - x^2}} - \arccot \frac{x - 2\sqrt{N} - 2N^{1/6} a_{k + 1 - j}}{\sqrt{4N - x^2}} \right)} \right. \\
    & \phantom{\frac{1}{\pi}} - \left( \alpha_1 - x + \frac{N - k}{\alpha_1} \right) \left( \arccot \frac{x - 2\sqrt{N} - 2N^{1/6} a_k}{\sqrt{4N - x^2}} - \arccot \frac{x - 2\sqrt{N}}{\sqrt{4N - x^2}} \right) \\
    & \phantom{\frac{1}{\pi}} \left. - \sum^k_{j = 2} \frac{N^{-1/6}}{a_k - a_{k + 1 - j}} \left( \arccot \frac{x - 2\sqrt{N} - 2N^{1/6} a_k}{\sqrt{4N - x^2}} - \arccot \frac{x - 2\sqrt{N} - 2N^{1/6} a_{k + 1 - j}}{\sqrt{4N - x^2}} \right) \right] \\
    = {}& \frac{1}{\pi} \left[ -\frac{\sqrt{4N - x^2}}{2} + \frac{x}{2} \arccos \frac{x}{2\sqrt{N}} - \frac{\pi}{2} (2\sqrt{N} - x) \right. \\
    & \phantom{\frac{1}{\pi}} \left. \vphantom{-\frac{\sqrt{4N - x^2}}{2} + \frac{x}{2} \arccos \frac{x}{2\sqrt{N}} - \frac{\pi}{2} (2\sqrt{N} - x)}
      - (2N^{1/6} a_k + \bigO(N^{-1/6})) \arccos \frac{x}{2\sqrt{N}} \right] + \bigO(N^{-1/5}) \\
    = {}& E_N(x) - (2\sqrt{N} - x) + \bigO(N^{-\frac{1}{12}} (2\sqrt{N} - x)^{\frac{1}{2}}),
  \end{split}
\end{equation}
where we note that $\alpha_i = \sqrt{N} + N^{1/6} a_{k-i+1}$ and $ N^{-1/10} < 2\sqrt{N} - x < (4 - \epsilon) \sqrt{N}$, and all the $\bigO$ terms are uniform in $x$. If some $a_i$ are identical, we use l'H\^{o}pital's rule if necessary, and derive the same result.

On the other hand, the integral in \eqref{eq:int_over_Phi_and_bar:Phi} has the estimate as $\bigO((4N - x^2)^{-1/2}) = \bigO(N^{-1/5})$. To see it, we deform the contours $\Gamma$ and $\Sigma$ into standard shapes depending on the value of $x$:
\begin{enumerate}[label=(\roman*)]
\item
  If $(-2 + \epsilon) \sqrt{N} < x < 1.9\sqrt{N}$, then let $\Gamma$ be $\Gamma^{\round}_{\standard, N}(0)$, and $\Sigma$ be $\Sigma^{|}_{\standard, N}(x/2)$.
\item
  If $1.9\sqrt{N} \leq x <2\sqrt{N} - N^{-1/10}$, then let $\Gamma$ be $\Gamma^{\corner}_{\standard, N}(\theta, 0)$ with $\theta = \arccos(x/\sqrt{4N})$, and $\Sigma$ be $\Sigma^{\zig}_{\standard, N}(x/2 - \sqrt{N - x^2/4}/\sqrt{3})$.
\end{enumerate}
Then we can estimate it by the standard saddle point analysis, similar to the analysis for the integral over $X$ in \eqref{eq:int_over_X_and_bar:X}. We omit the detail.

Hence we obtain \eqref{eq:GUE_measure_mean_alt}. 

\paragraph{Proof of \eqref{eq:GUE_measure_mean_local_alt}}

Like in the proof of \eqref{eq:GUE_measure_mean_alt}, we also make use of \eqref{eq:int_over_Phi_and_bar} and compute \eqref{eq:int_over_Phi_and_bar:Phi} and \eqref{eq:int_over_Phi_and_bar:bar} separately. However, now we require that $\upperintersect$ and $\lowerintersect$, the intersections of $\Sigma$ with $\Gamma$, are $\sqrt{N} \pm (2\sqrt{N} - x)^{1/2} N^{1/4} \ii$. (For $x \in [2\sqrt{N} - N^{-1/10}, 2\sqrt{N} - CN^{-1/6})$, they are very close to $\frac{1}{2}(x \pm\sqrt{4N - x^2}\ii)$.) Let $\xi = N^{1/6} (x - 2\sqrt{N})$, such that $-N^{1/15} < \xi < -C$. First we consider \eqref{eq:int_over_Phi_and_bar:bar}. We have, with $v = N^{-1/6}(w - \sqrt{N})$, if $-N^{1/15} < \xi < -C$,
\begin{equation} \label{eq:int_over_Phi_and_bar:bar_N^1/3}
  \begin{split}
    N^{\frac{1}{6}} \times \eqref{eq:int_over_Phi_and_bar:bar} = {}& \frac{-1}{2\pi \ii} \int^{\sqrt{-\xi}\ii}_{-\sqrt{-\xi}\ii} \frac{v^2 - \xi}{v - a_k} + \sum^k_{j = 2} \frac{1}{(v - a_{k + 1 - j})(v - a_k)} - N^{-\frac{1}{3}} \frac{v^3 + k}{1 - N^{-\frac{1}{3}} v}\dd v \\
    = {}& \frac{-1}{2\pi \ii} \int^{\sqrt{-\xi}\ii}_{-\sqrt{-\xi}\ii} \frac{v^2 - \xi}{v - a_k} + \sum^k_{j = 2} \frac{1}{(v - a_{k + 1 - j})(v - a_k)}\dd v + \bigO(N^{-\frac{1}{3}} \xi^3) \\
    = {}& \frac{-1}{2\pi \ii} \int^{\sqrt{-\xi}\ii}_{-\sqrt{-\xi}\ii} v + a_k + \frac{a_k - \xi}{v - a_k} + \sum^{k - 1}_{j = 1} \frac{1}{a_j - a_k} \left( \frac{1}{v - a_j} - \frac{1}{v - a_k} \right) \dd v + \bigO(N^{-\frac{1}{3}} \xi^3) \\
    = {}& \frac{1}{2}(\xi - a^2_k) - \frac{2}{\pi} a_k \sqrt{-\xi} + \bigO(\lvert \xi \rvert^{-\frac{1}{2}}).
  \end{split}
\end{equation}
Here we note that $N^{-1/3} \xi^3 = \bigO(N^{-2/15})$. This estimate is similar to that of \eqref{eq:int_over_X_and_bar:bar}, and we make use of the calculation in \eqref{eq:new_est_bar}.

Next, we consider \eqref{eq:int_over_Phi_and_bar:Phi}. We deform $\Gamma$ into $\Gamma^{\corner}_{\standard, N}(0, (2\sqrt{N} - x)^{1/2} N^{1/4}/\sqrt{3})$, and deform $\Sigma$ into $\Sigma^{\zig}_{\standard, N}((\sqrt{N} - 2\sqrt{N} - x)^{1/2} N^{1/4}/\sqrt{3})$. With the change of variables $\xi = N^{1/6} (x - 2\sqrt{N})$, $u = N^{-1/6}(z - \sqrt{N})$ and $v = N^{-1/6}(w - \sqrt{N})$, we have
\begin{equation} \label{eq:double_contour_mean_GUE_shifted}
  N^{\frac{1}{6}} \times \eqref{eq:int_over_Phi_and_bar:Phi} = \frac{-1}{(2\pi \ii)^2} \iint_{\mathsf{X}_N} \dd u \dd v \frac{e^{\frac{u^3}{3} - \xi u + f_N(u)}}{e^{\frac{v^3}{3} - \xi v + f_N(v)}} \left( \prod^k_{j = 2} \frac{u - a_{k + 1 - j}}{v - a_{k + 1 - j}} \right) \frac{1}{(v - a_k)(u - v)^2},
\end{equation}
where
\begin{equation}
  f_N(u) = (N - k)\log(1 + N^{-1/3}u) - N^{2/3}u + \frac{1}{2} N^{1/3} u^2 - \frac{1}{3} u^3,
\end{equation}
and the contour $\mathsf{X}_N$ is transformed from the deformed $\Phi = \Gamma \times \Sigma$ in (\ref{eq:int_over_Phi_and_bar:Phi}) by the change of variables. We note that in the region $u, v = o(N^{1/3})$, the contour $\mathsf{X}_N$ in \eqref{eq:double_contour_mean_GUE_shifted} overlaps with the contour $\mathsf{X}$ in \eqref{eq:int_over_X_and_bar:X}, if $\xi$ is identified with $x$ there. Hence our discussion here is analogous to \eqref{eq:int_over_X_and_bar:X}. We divide the double contour $\mathsf{X}_N$ into three disjoint subsets:
\begin{enumerate}[label=(\roman*)]
\item
  $\mathsf{X}_{1, N} = \{ u, v \in \mathsf{X}_N \mid \lvert u - \sqrt{-x} \ii \rvert < 1, \, \lvert v - \sqrt{-x} \ii \rvert < 1 \}$, 
\item 
  $\mathsf{X}_{2, N} = \{ u, v \in \mathsf{X}_N \mid \lvert u + \sqrt{-x} \ii \rvert < 1, \, \lvert v + \sqrt{-x} \ii \rvert < 1 \}$, 
\item
  $\mathsf{X}_{3, N} = \mathsf{X}_N \setminus (\mathsf{X}_{1, N} \cup \mathsf{X}_{2, N})$.
\end{enumerate}
By saddle point analysis similar to that given in \eqref{eq:saddle_first}--\eqref{eq:saddle_last}, we derive that the part of integral \eqref{eq:double_contour_mean_GUE_shifted} over $\mathsf{X}_{1, N}$ is $\bigO(\lvert \xi \rvert^{-1/2})$ for $-N^{1/15} < \xi < -C$. The same estimate holds for the part of integral \eqref{eq:double_contour_mean_GUE_shifted} over $\mathsf{X}_{2, N}$. At last, by standard method, we find the part of integral \eqref{eq:double_contour_mean_GUE_shifted} over $\mathsf{X}_{3, N}$ is $o(\lvert \xi \rvert^{-1/2})$. 

Combining \eqref{eq:int_over_Phi_and_bar} with the estimates of \eqref{eq:int_over_Phi_and_bar:bar_N^1/3} and \eqref{eq:double_contour_mean_GUE_shifted}, we prove the  identity \eqref{eq:GUE_measure_mean_local_alt}.

\paragraph{Proof of \eqref{eq:GUE_measure_variance_alt}}

We let $h_x(t)$ be defined in \eqref{eq:defn_h_x(t)}, and analogous to \eqref{eq:mean_formula_Airy}, we have
\begin{multline} \label{eq:variance_formula_GUE}
  \E[(S_x + xN_x)^2] = \int h^2_x(t) K^{0, 0}_{\GUEalpha}(t, t)\dd t + \int h^2_x(t) K^{1, 1}_{\GUEalpha}(t, t)\dd t \\
  \begin{aligned}[b]
    + {}& \iint h_x(s)h_x(t)
  \begin{vmatrix}
    K^{0, 0}_{\GUEalpha}(s, s) & K^{0, 0}_{\GUEalpha}(s, t) \\
    K^{0, 0}_{\GUEalpha}(t, s) & K^{0, 0}_{\GUEalpha}(t, t)
  \end{vmatrix}
  \dd s\dd t \\
  + {}& \iint h_x(s)h_x(t)
  \begin{vmatrix}
    K^{1, 1}_{\GUEalpha}(s, s) & K^{1, 1}_{\GUEalpha}(s, t) \\
    K^{1, 1}_{\GUEalpha}(t, s) & K^{1, 1}_{\GUEalpha}(t, t)
  \end{vmatrix}
  \dd s\dd t \\
  - {}& \iint h_x(s)h_x(t)
  \begin{vmatrix}
    K^{0, 0}_{\GUEalpha}(s, s) & K^{0, 1}_{\GUEalpha}(s, t) \\
    K^{1, 0}_{\GUEalpha}(t, s) & K^{1, 1}_{\GUEalpha}(t, t)
  \end{vmatrix}
  \dd s\dd t \\
  - {}& \iint h_x(s)h_x(t)
  \begin{vmatrix}
    K^{1, 1}_{\GUEalpha}(s, s) & K^{1, 0}_{\GUEalpha}(s, t) \\
    K^{0, 1}_{\GUEalpha}(t, s) & K^{0, 0}_{\GUEalpha}(t, t)
  \end{vmatrix}
  \dd s\dd t.
  \end{aligned}
\end{multline}
Hence, by \eqref{eq:variance_formula_GUE} and \eqref{eq:mean_GUE_by_kernel}, we have analogous to \eqref{eq:var_3parts_Airy}
\begin{subequations}
  \begin{align}
    & \var\big[ S_x + xN_x \big] \notag \\
    = {}& \int h_x^2(t) \Big(K^{0, 0}_{\GUEalpha}(t, t) + K^{1, 1}_{\GUEalpha}(t, t)\Big)\dd t - 2\int h_x(t) \left( \int^{\infty}_t h_x(s) e^{\alpha_1 (s - t)} K^{0, 1}_{\GUEalpha}(s, t) \dd s \right)\dd t \label{eq:var_part_1_GUE} \\
    &
      \begin{aligned}
        & - \iint h_x(s)h_x(t) \Big(K^{0, 0}_{\GUEalpha}(s, t) K^{0, 0}_{\GUEalpha}(t, s) + K^{1, 1}_{\GUEalpha}(s, t) K^{1, 1}_{\GUEalpha}(t, s) \\
        & \phantom{\smash{-\iint h_x(s)}} - K^{0, 1}_{\GUEalpha}(s, t) \widetilde{K}^{1, 0}_{\GUEalpha}(t, s) - \widetilde{K}^{1, 0}_{\GUEalpha}(s, t) K^{0, 1}_{\GUEalpha}(t, s)\Big) \dd s\dd t.
      \end{aligned} \label{eq:var_part_3_GUE}
  \end{align}
\end{subequations}

First we consider the integral in \eqref{eq:var_part_1_GUE}. Like \eqref{eq:part_1_var_trans} and \eqref{eq:part_2_var_trans}, we write
\begin{equation}
  \int h^2_x(t) \Big(K^{0, 0}_{\GUEalpha}(t, t) + K^{1, 1}_{\GUEalpha}(t, t)\Big)\dd t =\frac{1}{(2\pi \ii)^2} \int_{\Sigma} \dd z \int_{\Gamma} \dd w \frac{e^{\frac{z^2}{2} - xz} z^{N - k}}{e^{\frac{w^2}{2} - xw} w^{N - k}} \left( \prod^k_{j = 2} \frac{z - \alpha_j}{w - \alpha_j} \right) \frac{z + w - 2\alpha_1}{w - \alpha_1} \frac{2}{(z - w)^4},
\end{equation}
\begin{multline}
  \int h_x(t) \left( \int^{\infty}_t h_x(s) e^{\alpha_1 (s - t)} K^{0, 1}_{\GUEalpha}(s, t) \dd s \right)\dd t \\
  = \frac{1}{(2\pi \ii)^2} \int_{\Sigma} \dd z \int_{\Gamma} \dd w \frac{e^{\frac{z^2}{2} - xz} z^{N - k}}{e^{\frac{w^2}{2} - xw} w^{N - k}} \left( \prod^k_{j = 2} \frac{z - \alpha_j}{w - \alpha_j} \right) \frac{1}{(z - w)^4}  \frac{3z - w - 2\alpha_1}{z - \alpha_1}.
\end{multline}
So \eqref{eq:var_part_1_GUE} becomes
\begin{equation} \label{eq:easier_part_of_variance_GUE}
  \frac{2}{(2\pi \ii)^2} \int_{\Sigma} \dd z \int_{\Gamma} \dd w \frac{e^{\frac{z^2}{2} - xz} z^{N - k}}{e^{\frac{w^2}{2} - xw} w^{N - k}} \left( \prod^k_{j = 2} \frac{z - \alpha_j}{w - \alpha_j} \right) \frac{1}{(z - w)^2 (z - \alpha_1)(w - \alpha_1)}.
\end{equation}
Similarly to  \eqref{eq:int_over_Phi_and_bar}, when $x < 0$, this integral can be written as
\begin{subequations} \label{eq:variance_GUE_total}
  \begin{align}
    & \frac{2}{(2\pi \ii)^2} \iint_{\Phi} \dd z \dd w \frac{e^{\frac{z^2}{2} - xz} z^{N - k}}{e^{\frac{w^2}{2} - xw} w^{N - k}} \left( \prod^k_{j = 2} \frac{z - \alpha_j}{w - \alpha_j} \right) \frac{1}{(z - w)^2 (z - \alpha_1)(w - \alpha_1)} \label{eq:variance_GUE_Phi} \\
    + {}& \frac{2}{2\pi \ii} \int^{\upperintersect}_{\lowerintersect} \dd w \left( w + \frac{N - k}{w} - x + \sum^k_{j = 2} \frac{1}{w - \alpha_j} - \frac{1}{w - \alpha_1} \right) \frac{1}{(w - \alpha_1)^2} \label{eq:variance_GUE_bar} \\
    + {}& e^{\frac{\alpha^2_1}{2} - a_k x} \alpha^{N - k}_1 \prod^k_{j = 2} (\alpha_1 - \alpha_j) \frac{2}{2\pi \ii} \int_{\Gamma} \dd w \frac{1}{e^{\frac{w^2}{2} - xw} w^{N - k}} \left( \prod^k_{j = 2} \frac{1}{w - \alpha_j} \right) \frac{1}{(w - \alpha_1)^3}, \label{eq:single_int_gamma_L_GUE}
  \end{align}
\end{subequations}
where the contour $\Phi$ is the same as in \eqref{eq:int_over_Phi_and_bar}, and the contours $\Sigma$ and $\Gamma$ intersect at the two saddle points $\frac{1}{2}(x + \sqrt{4N - x^2}\ii)$ and $\frac{1}{2}(x - \sqrt{4N - x^2}\ii)$, which are the $\upperintersect$ and $\lowerintersect$ respectively. We also require that in \eqref{eq:variance_GUE_Phi} the contour $\Sigma$ lies to the left of all $\alpha_1, \dotsc, \alpha_k$, and the contour in \eqref{eq:variance_GUE_bar} lies to the right of $0$ and all $\alpha_1, \dotsc, \alpha_k$. We evaluate \eqref{eq:variance_GUE_bar} analogous to \eqref{eq:first_moment_GUE}, and have that if all $\alpha_1, \dotsc, \alpha_k$ are distinct, then
\begin{equation}
  \begin{split}
    \eqref{eq:variance_GUE_bar} = {}& \frac{2}{2\pi \ii} \int^{\frac{1}{2}(x + \sqrt{4N - x^2}\ii)}_{\frac{1}{2}(x - \sqrt{4N - x^2}\ii)} \dd w \left[ \left( 1 - \frac{N - k}{\alpha^2_1} - \sum^k_{j = 2} \frac{1}{(\alpha_1 - \alpha_j)^2} \right) \frac{1}{w - \alpha_1} + \frac{N - k}{\alpha^2_1} \frac{1}{w} + \sum^k_{j = 2} \frac{1}{(\alpha_1 - \alpha_j)^2} \frac{1}{w - \alpha_j} \right. \\
  & \phantom{\smash{\frac{2}{2\pi \ii} \int^{\frac{1}{2}(x + \sqrt{4N - x^2}\ii)}_{\frac{1}{2}(x - \sqrt{4N - x^2}\ii)} \dd w}} \left. + \left( \alpha_1 - x + \frac{N - k}{\alpha_1} + \sum^k_{j = 2} \frac{1}{\alpha_1 - \alpha_j} \right) \frac{1}{(w - \alpha_1)^2} - \frac{1}{(w - \alpha_1)^3} \right] \\
  = {}& \frac{2}{\pi} \left( \sqrt{1 - \frac{x^2}{4N}} + \arccos \frac{x}{2\sqrt{N}} \right) + \bigO(N^{-\frac{1}{3}}).
  \end{split}
\end{equation}
The case when some $\alpha_i$ are identical is similar. We also evaluate the integral in \eqref{eq:variance_GUE_Phi} by standard saddle point analysis similar to \eqref{eq:int_over_X_and_bar:X}, and find that it is $\bigO((4N - x^2)^{-1}) = \bigO(N^{-2/5})$. The integral in \eqref{eq:single_int_gamma_L_GUE} will be cancelled out later.

On the other hand, the double integral in \eqref{eq:var_part_3_GUE} can be expressed as
\begin{multline} \label{eq:four_contour_GUE}
  \frac{1}{(2\pi \ii)^4} \int_{\Sigma} \dd u \int_{\Gamma} \dd v \int_{\Sigma}\dd z \int_{\Gamma} \dd w \frac{e^{\frac{u^2}{2} - xu} u^{N - k}}{e^{\frac{v^2}{2} - xv} v^{N - k}} \frac{e^{\frac{z^2}{2} - xz} z^{N - k}}{e^{\frac{w^2}{2} - xw} w^{N - k}} \prod^k_{j = 2} \frac{u - \alpha_j}{v - \alpha_j} \frac{z - \alpha_j}{w - \alpha_j} \\
  \times \frac{1}{(u - v)(z - w) (u - w)(z - v) (v - \alpha_1)(w - \alpha_1)}.
\end{multline}
In order to estimate the integral above, we will perform several steps of contour deformation. Since the arguments are parallel to those for the contour deformations of \eqref{eq:four_contour_Airy} in Section \ref{pf.lem_Airy}, we omit most justifications. Also similar to the computation of \eqref{eq:four_contour_Airy}, below we assume that all contour integrals in the form of $\int^C_{\overline{C}}$ have the contour going between $0$ and $\min \{ \alpha_1, \dotsc, \alpha_k \}$, unless they are specially marked as $\int^C_{\overline{C}, \rightside}$, in which case the contour goes to the right of $\max \{ \alpha_1, \dotsc, \alpha_k \}$.
\begin{enumerate}[label=(\Roman*)]
\item \label{enu:step_variance_GUE:1}
  We first deform the contours for $w$ and $v$ to $\Gamma^{\outside}_w \cup \Gamma^{\inside}_w$ and $\Gamma^{\outside}_v \cup \Gamma^{\inside}_v$, respectively,  as in Figure \ref{fig:w_v_separation_GUE}, such that all $\alpha_j$ ($j = 1, \dotsc, k$) are enclosed in $\Gamma^{\inside}_w$, and then also in $\Gamma^{\inside}_v$, and $0$ is enclosed in $\Gamma^{\outside}_w$, and then also in $\Gamma^{\outside}_v$. We also slightly deform the contour $\Sigma$ and denote it by $\Sigma_u$ and $\Sigma_z$ for the contour of $u$ and $z$, respectively.  
  
\item
  We then further deform the contour $\Sigma_u$ such that it goes between $\Gamma^{\outside}_v$ and $\Gamma^{\inside}_v$, and thus also goes between $\Gamma^{\outside}_w$ and $\Gamma^{\inside}_w$. We denote by  $\Sigma'_u$ the deformed $\Sigma_u$; see Figure \ref{fig:gamma_u_deformed_GUE}. By residue calculation involving a $3$-fold contour integral like \eqref{eq:added_3-integral}, we write \eqref{eq:four_contour_GUE} as
  \begin{subequations} \label{eq:first_transform_GUE}
    \begin{multline}
      \frac{1}{(2\pi \ii)^4} \int_{\Sigma'_u} \dd u \int_{\Gamma^{\outside}_v \cup \Gamma^{\inside}_v} \dd v \int_{\Sigma_z}\dd z \int_{\Gamma^{\outside}_w \cup \Gamma^{\inside}_w} \dd w \frac{e^{\frac{u^2}{2} - xu} u^{N - k}}{e^{\frac{v^2}{2} - xv} v^{N - k}} \frac{e^{\frac{z^2}{2} - xz} z^{N - k}}{e^{\frac{w^2}{2} - xw} w^{N - k}} \prod^k_{j = 2} \frac{u - \alpha_j}{v - \alpha_j} \frac{z - \alpha_j}{w - \alpha_j} \\
      \times \frac{1}{(u - v)(z - w) (u - w)(z - v) (v - \alpha_1)(w - \alpha_1)}
    \end{multline}
    \begin{flalign}
      && + {}& \frac{1}{(2\pi \ii)^2} \int_{\Sigma_z} \dd z \int_{\Gamma^{\inside}_w} \dd w \frac{e^{\frac{z^2}{2} - xz} z^{N - k}}{e^{\frac{w^2}{2} - xw} w^{N - k}} \left( \prod^k_{j = 2} \frac{z - \alpha_j}{w - \alpha_j} \right) \frac{1}{(w - \alpha_1)(z - \alpha_1) (z - w)^2} \\ 
      && + {}& \frac{1}{(2\pi \ii)^2} \int_{\Gamma^{\inside}_v} \dd v \int_{\Sigma_z} \dd z \frac{e^{\frac{z^2}{2} - xz} z^{N - k}}{e^{\frac{v^2}{2} - xv} v^{N - k}} \left( \prod^k_{j = 2} \frac{z - \alpha_j}{v - \alpha_j} \right) \frac{-1}{(v - \alpha_1)^2(z - \alpha_1) (z - v)} \\ 
      && + {}& \frac{1}{(2\pi \ii)^2} \int_{\Gamma^{\outside}_v} \dd v \int_{\Sigma_z} \dd z \frac{e^{\frac{z^2}{2} - xz} z^{N - k}}{e^{\frac{v^2}{2} - xv} v^{N - k}} \left( \prod^k_{j = 2} \frac{z - \alpha_j}{v - \alpha_j} \right) \frac{-1}{(v - \alpha_1)^2 (z - \alpha_1)(z - v)} \\ 
      && + {}& \frac{1}{(2\pi \ii)^2} \int_{\Gamma^{\outside}_w} \dd v \int_{\Sigma_z} \dd z \frac{e^{\frac{z^2}{2} - xz} z^{N - k}}{e^{\frac{w^2}{2} - xw} w^{N - k}} \left( \prod^k_{j = 2} \frac{z - \alpha_j}{w - \alpha_j} \right) \frac{-1}{(w - \alpha_1)^2 (z - \alpha_1)(z - w)}. 
    \end{flalign}
  \end{subequations}
  
  \begin{figure}[htb]
  \begin{minipage}[t]{0.3\linewidth}
    \centering
    \includegraphics{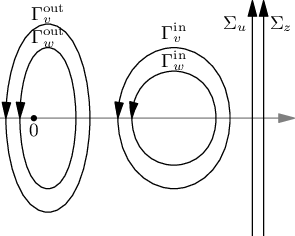}
    \caption{Separation of contours for $w$ and $v$.}
    \label{fig:w_v_separation_GUE}
  \end{minipage}
  \hspace{\stretch{1}}
    \begin{minipage}[t]{0.3\linewidth}
      \centering
      \includegraphics{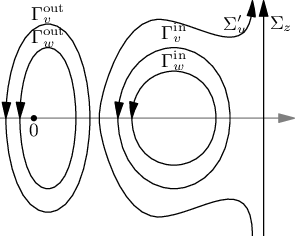}
      \caption{$\Sigma_u$ is deformed into $\Sigma'_u$.}
      \label{fig:gamma_u_deformed_GUE}
    \end{minipage}
    \hspace{\stretch{1}}
    \begin{minipage}[t]{0.3\linewidth}
      \centering
      \includegraphics{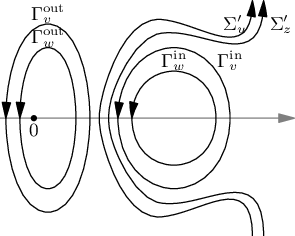}
      \caption{$\Sigma_z$ is deformed into $\Sigma'_z$.}
      \label{fig:gamma_z_deformed_GUE}
    \end{minipage}
  \end{figure}
\item
  Next, similarly to the previous step, we further  deform the contour $\Sigma_z$ such that it goes between $\Gamma^{\outside}_v$ and $\Gamma^{\inside}_v$, and thus also goes between $\Gamma^{\outside}_w$ and $\Gamma^{\inside}_w$. Hence $\Sigma_z$ becomes $\Sigma'_z$; see Figure \ref{fig:gamma_z_deformed_GUE}. By residue calculation, the quantity in \eqref{eq:first_transform_GUE} becomes
  \begin{subequations} \label{eq:second_transform_GUE}
    \begin{multline} \label{eq:2nd_transform_GUE}
      \frac{1}{(2\pi \ii)^4} \int_{\Sigma'_u} \dd u \int_{\Gamma^{\outside}_v \cup \Gamma^{\inside}_v} \dd v \int_{\Sigma'_z}\dd z \int_{\Gamma^{\outside}_w \cup \Gamma^{\inside}_w} \dd w \frac{e^{\frac{u^2}{2} - xu} u^{N - k}}{e^{\frac{v^2}{2} - xv} v^{N - k}} \frac{e^{\frac{z^2}{2} - xz} z^{N - k}}{e^{\frac{w^2}{2} - xw} w^{N - k}} \prod^k_{j = 2} \frac{u - \alpha_j}{v - \alpha_j} \frac{z - \alpha_j}{w - \alpha_j} \\
      \times \frac{1}{(u - v)(z - w) (u - w)(z - v) (v - \alpha_1)(w - \alpha_1)}
    \end{multline}
    \begin{flalign}
      && + {}& \frac{1}{(2\pi \ii)^2} \int_{\Sigma'_u} \dd u \int_{\Gamma^{\inside}_w} \dd w \frac{e^{\frac{u^2}{2} - xu} u^{N - k}}{e^{\frac{w^2}{2} -xw} w^{N - k}} \left( \prod^k_{j = 2} \frac{u - \alpha_j}{w - \alpha_j} \right) \frac{1}{(w - \alpha_1)(u - \alpha_1) (u - w)^2} \\ 
      && + {}& \frac{1}{(2\pi \ii)^2} \int_{\Sigma'_u} \dd u \int_{\Gamma^{\inside}_v} \dd v \frac{e^{\frac{u^2}{2} - xu} u^{N - k}}{e^{\frac{v^2}{2} - xv} v^{N - k}} \left( \prod^k_{j = 2} \frac{u - \alpha_j}{v - \alpha_j} \right) \frac{-1}{(v - \alpha_1)^2(u - \alpha_1) (u - v)} \\ 
      && + {}& \frac{1}{(2\pi \ii)^2} \int_{\Sigma'_u} \dd u \int_{\Gamma^{\outside}_v} \dd v \frac{e^{\frac{u^2}{2} - xu} u^{N - k}}{e^{\frac{v^2}{2} - xv} v^{N - k}} \left( \prod^k_{j = 2} \frac{u - \alpha_j}{v - \alpha_j} \right) \frac{-1}{(v - \alpha_1)^2 (u - \alpha_1)(u - v)} \\ 
      && + {}& \frac{1}{(2\pi \ii)^2} \int_{\Sigma'_u} \dd u \int_{\Gamma^{\outside}_w} \dd w \frac{e^{\frac{u^2}{2} - xu} u^{N - k}}{e^{\frac{w^2}{2} - xw} w^{N - k}} \left( \prod^k_{j = 2} \frac{u - \alpha_j}{w - \alpha_j} \right) \frac{-1}{(w - \alpha_1)^2 (u - \alpha_1)(u - w)} \\ 
      && + {}& \frac{1}{(2\pi \ii)^2} \int_{\Sigma'_z}\dd z \int_{\Gamma^{\inside}_w} \dd w \frac{e^{\frac{z^2}{2} - xz} z^{N - k}}{e^{\frac{w^2}{2} - xw} w^{N - k}} \left( \prod^k_{j = 2} \frac{z - \alpha_j}{w - \alpha_j} \right) \frac{1}{(w - \alpha_1)(z - \alpha_1) (z - w)^2} \\ 
      && + {}& \frac{1}{2\pi \ii} \int_{\Gamma^{\inside}_w} \dd w \left( w - x + \frac{N - k}{w} + \sum^k_{j = 2} \frac{1}{w - \alpha_j} - \frac{1}{w - \alpha_1} \right)\frac{1}{(w - \alpha_1)^2} \label{eq:gamma'_z_and_w_in_add_GUE} \\
      && + {}& \prod^k_{j = 2} (\alpha_1 - \alpha_j) e^{\frac{\alpha^2_1}{2} - \alpha_1 x} \alpha^{N - k}_1 \frac{1}{2\pi \ii} \int_{\Gamma^{\inside}_w} \dd w \frac{1}{e^{\frac{w^2}{2} - xw} w^{N - k}} \left( \prod^k_{j = 2} \frac{1}{w - \alpha_j} \right) \frac{1}{(w - \alpha_1)^3} \\ 
      && + {}& \frac{1}{(2\pi \ii)^2} \int_{\Gamma^{\inside}_v} \dd v \int_{\Sigma'_z}\dd z \frac{e^{\frac{z^2}{2} - xz} z^{N - k}}{e^{\frac{v^2}{2} - xv} v^{N - k}} \left( \prod^k_{j = 2} \frac{z - \alpha_j}{v - \alpha_j} \right) \frac{-1}{(v - \alpha_1)^2(z - \alpha_1) (z - v)} \\ 
      && + {}& \prod^k_{j = 2} (\alpha_1 - \alpha_j) e^{\frac{\alpha^2_1}{2} - \alpha_1 x} \alpha^{N - k}_1 \frac{1}{2\pi \ii} \int_{\Gamma^{\inside}_v} \dd v \frac{1}{e^{\frac{v^2}{2} - xv} v^{N - k}} \left( \prod^k_{j = 2} \frac{1}{v - \alpha_j} \right) \frac{1}{(v - \alpha_1)^3} \\ 
      && + {}& \frac{1}{(2\pi \ii)^2} \int_{\Gamma^{\outside}_w} \dd v \int_{\Sigma'_z}\dd z \frac{e^{\frac{z^2}{2} - xz} z^{N - k}}{e^{\frac{w^2}{2} - xw} w^{N - k}} \left( \prod^k_{j = 2} \frac{z - \alpha_j}{w - \alpha_j} \right) \frac{-2}{(w - \alpha_1)^2 (z - \alpha_1)(z - w)} \\ 
      && + {}& \prod^k_{j = 2} (\alpha_1 - \alpha_j) e^{\frac{\alpha^2_1}{2} - \alpha_1 x} \alpha^{N - k}_1 \frac{1}{2\pi \ii} \int_{\Gamma^{\outside}_w} \dd w \frac{1}{e^{\frac{w^2}{2} - xw} w^{N - k}} \left( \prod^k_{j = 2} \frac{1}{w - \alpha_j} \right) \frac{2}{(w - \alpha_1)^3}. 
    \end{flalign}
  \end{subequations}
  
\item \label{enu:step_deform_v_GUE}
  For further deformation of the contours, we introduce the shorthand notations $\GUEupperright$, $\GUEupperleft$, $\GUElowerright$ and $\GUElowerleft$, analogous to $\Airyupperright$, $\Airyupperleft$, $\Airylowerright$ and $\Airylowerleft$ defined in \eqref{four dots}. We delay the concrete assignment of their values to Remark \ref{rmk:defn_four_dots_GUE}, and only indicate that they are close to $\frac{1}{2}(x+\sqrt{4N-x^2}\ii)$, and their relative positions are shown in the subsequent figures schematically; especially see Figure \ref{fig:w_cross_uz_GUE}.

  For the $4$-fold integral \eqref{eq:2nd_transform_GUE}, we perform the following operations:
  \begin{enumerate}[label=(\roman*)]
  \item 
    deform $\Sigma'_u$ such that it passes $\GUEupperleft$ and $\overline{\GUEupperleft}$;
  \item
    deform $\Sigma'_z$ such that it passes $\GUEupperright$ and $\overline{\GUEupperright}$;
  \item 
    deform $\Gamma^{\outside}_v$ such that it goes from $\overline{\GUEupperleft}$, along the left side of $\Sigma'_u$ until it reaches $\GUEupperleft$, then wraps around $0$ (and $\Gamma^{\outside}_w$) and finally goes back to $\overline{\GUEupperleft}$;
  \item 
    deform $\Gamma^{\inside}_v$ such that it goes from $\GUEupperright$ to $\overline{\GUEupperright}$ along the right side of $\Sigma'_z$, then wraps around all $\alpha_j$'s, and finally goes back to $\GUEupperright$;
  \item
    and at last add an additional contour for $v$, on which the contour integral vanishes: the contour goes from $\overline{\GUEupperright}$ to $\GUEupperright$ along the left-side of $\Sigma'_z$, then goes from $\GUEupperright$ to $\GUEupperleft$, and further goes from $\GUEupperleft$ to $\overline{\GUEupperleft}$ along the right-side of $\Sigma'_u$, and finally goes from $\overline{\GUEupperleft}$ to $\overline{\GUEupperright}$.
  \end{enumerate}
  See Figure \ref{fig:v_cross_uz_GUE} for the deformation of contours.
  \begin{figure}[htb]
    \begin{minipage}[t]{0.45\linewidth}
      \centering
      \includegraphics{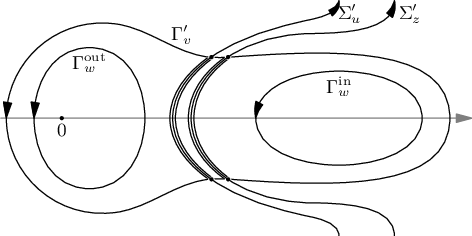}
      \caption{The deformed contours for $v$, $u$ and $z$. The four intersections are $\GUEupperleft$, $\GUEupperright$, $\overline{\GUEupperleft}$ and $\overline{\GUEupperright}$.}
      \label{fig:v_cross_uz_GUE}
    \end{minipage}
    \hspace{\stretch{1}}
    \begin{minipage}[t]{0.45\linewidth}
      \centering
      \includegraphics{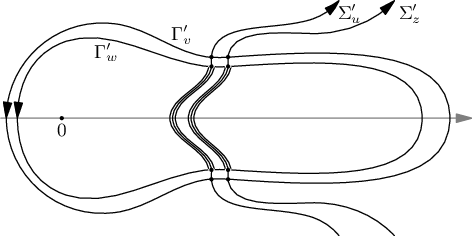}
      \caption{The deformed contours for $w$, $u$ and $z$. The four intersections on $\compC_+$ are $\GUEupperleft$, $\GUEupperright$, $\GUElowerleft$ and $\GUElowerright$, and the four intersections on $\compC_-$ are their complex conjugates.}
      \label{fig:w_cross_uz_GUE}
    \end{minipage}
  \end{figure}
  
  Now we define the contour $\Gamma'_v$ as in Figure \ref{fig:v_cross_uz_GUE} that goes from $\overline{\GUEupperleft}$ to $\overline{\GUEupperright}$, then wraps $\alpha_j$'s until it reaches $\GUEupperright$, and then goes to $\GUEupperleft$, and finally wraps $0$ and goes to $\overline{\GUEupperleft}$. Hence the $4$-fold integral \eqref{eq:2nd_transform_GUE} can be simplified by the residue theorem with $\Gamma^{\outside}_v \cup \Gamma^{\inside}_v$ replaced by $\Gamma'_v$. Then the formula \eqref{eq:second_transform_GUE} becomes
  \begin{subequations} \label{eq:third_transform_GUE}
    \begin{multline} \label{eq:4-fold_PV_v_GUE}
      \frac{1}{(2\pi \ii)^4} \int_{\Sigma'_u} \dd u \int_{\Sigma'_z}\dd z \int_{\Gamma^{\outside}_w \cup \Gamma^{\inside}_w} \dd w \PV \int_{\Gamma'_v} \dd v \frac{e^{\frac{u^2}{2} - xu} u^{N - k}}{e^{\frac{v^2}{2} - xv} v^{N - k}} \frac{e^{\frac{z^2}{2} - xz} z^{N - k}}{e^{\frac{w^2}{2} - xw} w^{N - k}} \prod^k_{j = 2} \frac{u - \alpha_j}{v - \alpha_j} \frac{z - \alpha_j}{w - \alpha_j} \\
      \times \frac{1}{(u - v)(z - w) (u - w)(z - v) (v - \alpha_1)(w - \alpha_1)}
    \end{multline}
    \begin{flalign}
      && + {}& \frac{1}{(2\pi \ii)^3} \int_{\Sigma'_z}\dd z \int_{\Gamma^{\outside}_w \cup \Gamma^{\inside}_w} \dd w \int^{\GUEupperleft}_{\overline{\GUEupperleft}} \dd u \frac{e^{\frac{z^2}{2} - xz} z^{N - k}}{e^{\frac{w^2}{2} - xw} w^{N - k}} \left( \prod^k_{j = 2} \frac{z - \alpha_j}{w - \alpha_j} \right)  \frac{1}{(z - w) (u - w)(z - u) (u - \alpha_1)(w - \alpha_1)} \label{eq:u_out_GUE} \\
      && + {}& \frac{1}{(2\pi \ii)^3} \int_{\Sigma'_u} \dd u \int_{\Gamma^{\outside}_w \cup \Gamma^{\inside}_w} \dd w  \int^{\GUEupperright}_{\overline{\GUEupperright}} \dd z \frac{e^{\frac{u^2}{2} - xu} u^{N - k}}{e^{\frac{w^2}{2} - xw} w^{N - k}} \left( \prod^k_{j = 2} \frac{u - \alpha_j}{w - \alpha_j} \right)  \frac{1}{(u - z)(z - w) (u - w) (z - \alpha_1)(w - \alpha_1)} \label{eq:u_in_GUE} \\
      && + {}&\frac{1}{(2\pi \ii)^2} \int_{\Gamma^{\outside}_w} \dd v \int_{\Sigma'_z}\dd z \frac{e^{\frac{z^2}{2} - xz} z^{N - k}}{e^{\frac{w^2}{2} - xw} w^{N - k}} \left( \prod^k_{j = 2} \frac{z - \alpha_j}{w - \alpha_j} \right) \frac{-4}{(w - \alpha_1)^2 (z - \alpha_1)(z - w)} \\ 
      && + {}& \prod^k_{j = 2} (\alpha_1 - \alpha_j) e^{\frac{\alpha^2_1}{2} - \alpha_1 x} \alpha^{N - k}_1 \frac{1}{2\pi \ii} \int_{\Gamma^{\inside}_w \cup \Gamma^{\outside}_w} \dd w \frac{1}{e^{\frac{w^2}{2} - xw} w^{N - k}} \left( \prod^k_{j = 2} \frac{1}{w - \alpha_j} \right) \frac{2}{(w - \alpha_1)^3} \label{eq:int_in_w_out_w_twice_GUE} \\ 
      && + {}&\frac{1}{(2\pi \ii)^2} \int_{\Sigma'_z}\dd z \int_{\Gamma^{\inside}_w} \dd w \frac{e^{\frac{z^2}{2} - xz} z^{N - k}}{e^{\frac{w^2}{2} - xw} w^{N - k}} \left( \prod^k_{j = 2} \frac{z - \alpha_j}{w - \alpha_j} \right) \frac{2(2w - z - \alpha_1)}{(w - \alpha_1)^2 (z - \alpha_1) (z - w)^2} \\ 
      && + {}& 1 - \frac{N - k}{\alpha^2_1}. \notag
    \end{flalign}
  \end{subequations}
  Here we remark that the $1$-fold integral \eqref{eq:gamma'_z_and_w_in_add_GUE} is equal to $1 - \frac{N - k}{\alpha^2_1}$.
\item \label{enu:step_variance_GUE:5}
  Now we deform the contour $\Gamma^{\outside}_w \cup \Gamma^{\inside}_w$ for $w$ in the way similar to our deformation of $\Gamma^{\outside}_v \cup \Gamma^{\inside}_v$ for $v$ in Step \ref{enu:step_deform_v_GUE}. We perform the following operations:
  \begin{enumerate}
  \item
    deform $\Sigma'_u$ such that it passes $\GUElowerleft$ and $\overline{\GUElowerleft}$, and meanwhile still passes $\GUEupperleft$ and $\overline{\GUEupperleft}$;
  \item
    deform $\Sigma'_z$ such that it passes $\GUElowerright$ and $\overline{\GUElowerright}$, and meanwhile still passes $\GUEupperright$ and $\overline{\GUEupperright}$;
  \item 
    deform $\Gamma^{\outside}_w$ such that it goes from $\overline{\GUElowerleft}$, along the left side of $\Sigma'_u$ until it reaches $\GUElowerleft$, then wraps around $0$ and finally goes back to $\overline{\GUElowerleft}$;
  \item 
    deform $\Gamma^{\inside}_w$ such that it goes from $\GUElowerright$ to $\overline{\GUElowerright}$ along the right-side of $\Sigma'_z$, then wraps around all $\alpha_j$, and finally goes back to $\GUElowerright$\;;
  \item
    and at last add an additional contour for $w$, on which the contour integral vanishes: the contour goes from $\overline{\GUElowerright}$ to $\GUElowerright$ along the left-side of $\Sigma'_z$, then goes from $\GUElowerright$ to $\GUElowerleft$, and further goes from $\GUElowerleft$ to $\overline{\GUElowerleft}$ along the right-side of $\Sigma'_u$, and finally goes from $\overline{\GUElowerleft}$ to $\overline{\GUElowerright}$\;.
  \end{enumerate}
  We have the result in Figure \ref{fig:w_cross_uz_GUE}. Similar to $\Gamma'_v$ in Figure \ref{fig:v_cross_uz_GUE}.  Now we define the contour $\Gamma'_ w$ that goes from $\overline{\GUElowerleft}$ to $\overline{\GUElowerright}$\;, then wraps $\alpha_j$'s until it reaches $\GUElowerright$\;, and then goes to $\GUElowerleft$, and finally wraps $0$ and goes to $\overline{\GUElowerleft}$. Hence by the residue theorem, the $4$-fold integral \eqref{eq:4-fold_PV_v_GUE} can be simplified with $\Gamma^{\outside}_w \cup \Gamma^{\inside}_w$ replaced by $\Gamma'_w$, and then formula \eqref{eq:third_transform_GUE} becomes
  \begin{subequations}
    \begin{multline} \label{eq:4-fold_PV_vw_GUE}
      \frac{1}{(2\pi \ii)^4} \int_{\Sigma'_u} \dd u \int_{\Sigma'_z}\dd z \PV \int_{\Gamma'_w} \dd w \PV \int_{\Gamma'_v} \dd v \frac{e^{\frac{u^2}{2} - xu} u^{N - k}}{e^{\frac{v^2}{2} - xv} v^{N - k}} \frac{e^{\frac{z^2}{2} - xz} z^{N - k}}{e^{\frac{w^2}{2} - xw} w^{N - k}} \prod^k_{j = 2} \frac{u - \alpha_j}{v - \alpha_j} \frac{z - \alpha_j}{w - \alpha_j} \\
      \times \frac{1}{(u - v)(z - w) (u - w)(z - v) (v - \alpha_1)(w - \alpha_1)}
    \end{multline}
    \begin{flalign}
      && + {}& \frac{1}{(2\pi \ii)^3} \int_{\Sigma'_z}\dd z \PV \int_{\Gamma'_v} \dd v \int^{\GUElowerleft}_{\overline{\GUElowerleft}} \dd u \frac{e^{\frac{z^2}{2} - xz} z^{N - k}}{e^{\frac{v^2}{2} - xv} v^{N - k}} \left( \prod^k_{j = 2} \frac{z - \alpha_j}{v - \alpha_j} \right) \frac{1}{(u - v)(z - u) (z - v) (v - \alpha_1)(u - \alpha_1)} \label{eq:int_vz_u_GUE} \\
      && + {}& \frac{1}{(2\pi \ii)^3} \int_{\Sigma'_u} \dd u \PV \int_{\Gamma'_v} \dd v \int^{\GUElowerright}_{\overline{\GUElowerright}} \dd z \frac{e^{\frac{u^2}{2} - xu} u^{N - k}}{e^{\frac{v^3}{3} + xv}} \left( \prod^k_{j = 2} \frac{z - \alpha_j}{v - \alpha_j} \right) \frac{1}{(u - v)(z - w) (z - v) (v - \alpha_1)(z - \alpha_1)}  \label{eq:int_vu_z_GUE} \\
      && + {}& \frac{1}{(2\pi \ii)^2} \int_{\Gamma^{\inside}_w \cup \Gamma^{\outside}_w} \dd w \int_{\Sigma'_z}\dd z \frac{e^{\frac{z^2}{2} - xz} z^{N - k}}{e^{\frac{w^2}{2} - xw} w^{N - k}} \left( \prod^k_{j = 2} \frac{z - \alpha_j}{w - \alpha_j} \right) \frac{-4}{(w - \alpha_1)^2 (z - \alpha_1)(z - w)} \label{eq:int_out_w_and_in_w_z_times_4_GUE} \\
      && + {}& \frac{1}{(2\pi \ii)^2} \int_{\Sigma'_z}\dd z \int_{\Gamma^{\inside}_w} \dd w \frac{e^{\frac{z^2}{2} - xz} z^{N - k}}{e^{\frac{w^2}{2} - xw} w^{N - k}} \left( \prod^k_{j = 2} \frac{z - \alpha_j}{w - \alpha_j} \right) \frac{2}{(w - \alpha_1)^2 (z - w)^2} \label{eq:gamma'_z_and_w_in_twice_remainder_GUE} \\
      && + {}& \eqref{eq:u_out_GUE} + \eqref{eq:u_in_GUE} + \eqref{eq:int_in_w_out_w_twice_GUE} + 1 - \frac{N - k}{\alpha^2_1}.
    \end{flalign}
  \end{subequations}
\end{enumerate}

Now we can compute the contour integrals by saddle point method. First we need to fix the shapes of contours $\Gamma'_v$, $\Gamma'_w$, $\Sigma'_u$ and $\Sigma'_z$.
\begin{enumerate}[label=(\roman*)]
\item 
  If $(-2 + \epsilon) \sqrt{N} < x < 1.9\sqrt{N}$, then let $\Gamma'_w$ be $\Gamma^{\round}_{\standard, N}(0)$, $\Gamma'_v$ be $\Gamma^{\round}_{\standard, N}(1)$, $\Sigma'_u$ be $\Sigma^{|}_{\standard, N}(x/2)$ and $\Sigma'_z$ be $\Sigma^{|}_{\standard, N}(x/2 + 1)$.
\item
  If $1.9\sqrt{N} \leq x<2\sqrt{N} - N^{-1/10}$, then with $\theta = \arccos(x/\sqrt{4N})$, let $\Gamma'_w$ be $\Gamma^{\corner}_{\standard, N}(\theta, 0)$, $\Gamma'_v$ be $\Gamma^{\corner}_{\standard, N}(\theta, (2 - x/\sqrt{N})^{-1/4})$, $\Sigma'_u$ be $\Sigma^{\zig}_{\standard, N}(x/2 - \sqrt{N - x^2/4}/\sqrt{3})$ and $\Sigma'_z$ be $\Sigma^{\zig}_{\standard, N}(x/2 - \sqrt{N - x^2/4}/\sqrt{3} + (2 - x/\sqrt{N})^{-1/4})$.
\end{enumerate}
\begin{rmk} \label{rmk:defn_four_dots_GUE}
  We note that in either case, the four contours have four intersections around $\frac{1}{2}(x+\sqrt{4N-x}\ii)$, and they are the desired $\GUEupperright$, $\GUEupperleft$, $\GUElowerright$ and $\GUElowerleft$.
\end{rmk}
\begin{enumerate}[label=(\arabic*)]
\item
  The $4$-fold integral \eqref{eq:4-fold_PV_vw_GUE} can be estimated in the same way as for \eqref{eq:4-fold_PV_vw}, and it is $\bigO((2\sqrt{N} - x)^{-1} N^{-1/4})$.
\item
  Both of the $3$-fold integrals \eqref{eq:int_vz_u_GUE} and \eqref{eq:int_vu_z_GUE} can be estimated in the same way as for \eqref{eq:int_vz_u} and \eqref{eq:int_vu_z}. We note that in the evaluation of \eqref{eq:int_vz_u} and \eqref{eq:int_vu_z}, we specified the shapes of the contour from $\overline{\Airylowerleft}$ to $\Airylowerleft$ for $u$ and the contour from $\overline{\Airylowerright}$ to $\Airylowerright$ for $z$. Here we can deform the contours from $\overline{\GUElowerleft}$ to $\GUElowerleft$ for $u$ and the contour from $\overline{\GUElowerright}$ to $\GUElowerright$ for $z$ in a similar way: the contour for $u$ is the part of $\Sigma'_u$ inside of $\Gamma'_v$, and the contour for $z$ is the part of $\Sigma'_z$ inside of $\Gamma'_w$. We conclude that both \eqref{eq:int_vz_u_GUE} and \eqref{eq:int_vu_z_GUE} are $\bigO((2\sqrt{N} - x)^{-1} N^{-1/4})$.
\item
  The $3$-fold integral \eqref{eq:u_out_GUE} can be written as the sum of
  \begin{subequations}
    \begin{align}
      & \frac{1}{(2\pi \ii)^3} \int_{\Sigma'_z}\dd z \PV \int_{\Gamma''_w} \dd w \int^{\GUEupperleft}_{\overline{\GUEupperleft}} \dd u \frac{e^{\frac{z^2}{2} - xz} z^{N - k}}{e^{\frac{w^2}{2} - xw} w^{N - k}} \left( \prod^k_{j = 2} \frac{z - \alpha_j}{w - \alpha_j} \right)    \frac{1}{(z - w) (u - w)(z - u) (u - \alpha_1)(w - \alpha_1)} \label{eq:first_trivial_GUE} \\
      + {}& \frac{1}{(2\pi \ii)^2} \int^{\GUEupperleft}_{\overline{\GUEupperleft}} \dd u \int^{\upperintersect'}_{\lowerintersect'} \dd z \frac{-1}{(u - z)^2 (u - \alpha_1)(z - \alpha_1)} \label{eq:last_added_1_GUE} \\
      + {}& \frac{1}{(2\pi \ii)^2} \int_{\Sigma''_z}\dd z \int^{\GUEupperleft}_{\overline{\GUEupperleft}} \dd u \frac{e^{\frac{z^2}{2} - xz} z^{N - k}}{e^{\frac{u^2}{2} - xu} u^{N - k}} \left( \prod^k_{j = 2} \frac{z - \alpha_j}{u - \alpha_j} \right) \frac{1}{(z - u)^2(u - \alpha_1)^2} \label{eq:gamma''_z_and_u_GUE} \\
      + {}& \frac{1}{2\pi \ii} \int^{\GUEupperleft}_{\overline{\GUEupperleft}} \dd u \left( (u - x + \frac{N - k}{u}) + \sum^k_{j = 2} \frac{1}{u - \alpha_j} \right) \frac{1}{(u - \alpha_1)^2}, \label{eq:last_done_GUE}
    \end{align}
  \end{subequations}
  where the contours
  \begin{align}
    \Gamma''_w = {}&
    \begin{cases}
      \Gamma^{\round}_{\standard, N}(2), & (-2 + \epsilon) \sqrt{N} < x < 1.9\sqrt{N}, \\
      \Gamma^{\corner}_{\standard, N}(\arccos(x/\sqrt{4N}), 2(2 - x/\sqrt{N})^{-1/4}), & 1.9\sqrt{N} \leq x <2\sqrt{N} - N^{-1/10},
    \end{cases} \\
    \Sigma''_z = {}&
    \begin{cases}
     \Sigma^{|}_{\standard, N}(x/2 - 1), & (-2 + \epsilon) \sqrt{N} < x < 1.9\sqrt{N}, \\
     \Sigma^{\zig}_{\standard, N}(x/2 - \sqrt{N - x^2/4}/\sqrt{3} - (2 - x/\sqrt{N})^{-1/4}), & 1.9\sqrt{N} \leq x < 2\sqrt{N} - N^{-1/10},
    \end{cases}
  \end{align}
  $\upperintersect'$ and $\lowerintersect'$ are the intersections between $\Gamma''_w$ and $\Sigma'_z$. Also we take the contours from $\overline{\GUEupperleft}$ to $\GUEupperleft$ for $u$ to be the part of $\Sigma'_u$ inside $\Gamma'_v$, the contour from $\lowerintersect'$ to $\upperintersect'$ for $z$ to be the part of $\Sigma'_z$ inside $\Gamma''_w$. On the other hand, the $3$-fold \eqref{eq:u_in_GUE} can bewritten as the sum of
  \begin{subequations}
    \begin{align}
      & \frac{1}{(2\pi \ii)^3} \int_{\Sigma'_u} \dd u \PV \int_{\Gamma''_w} \dd w  \int^{\GUEupperright}_{\overline{\GUEupperright}}\dd z \frac{e^{\frac{u^2}{2} - xu} u^{N - k}}{e^{\frac{w^2}{2} - xw} w^{N - k}} \left( \prod^k_{j = 2} \frac{u - \alpha_j}{w - \alpha_j} \right) \frac{1}{(u - z)(z - w) (u - w) (z - \alpha_1)(w - \alpha_1)} \\
      + {}& \frac{1}{(2\pi \ii)^2} \int^{\upperintersect''}_{\lowerintersect''} \dd u \int^{\GUEupperright}_{\overline{\GUEupperright}} \dd z \frac{-1}{(u - z)^2 (u - \alpha_1)(z - \alpha_1)} \label{eq:last_added_2_GUE} \\
      + {}& \frac{1}{(2\pi \ii)^2} \int_{\Sigma'_u} \dd u \int^{\GUEupperright}_{\overline{\GUEupperright}} \dd z \frac{e^{\frac{u^2}{2} - xu} u^{N - k}}{e^{\frac{z^2}{2} - xz} z^{N - k}} \left( \prod^k_{j = 2} \frac{u - \alpha_j}{z - \alpha_j} \right) \frac{1}{(u - z)^2(z - \alpha_1)^2}. \label{eq:two_cancellation_early_GUE}
    \end{align}
  \end{subequations}
  where the contour $\Gamma''_w$ is the same as $\Gamma''_w$ in \eqref{eq:first_trivial_GUE}, and it intersects $\Sigma'_u$ at $\upperintersect''$ and $\lowerintersect''$, and the contour from $\overline{\GUEupperright}$ to $\GUEupperright$ for $z$ is the part of $\Sigma'_z$ inside $\Gamma'_v$.
\item
  The $2$-fold integral \eqref{eq:int_out_w_and_in_w_z_times_4_GUE} can be written as the sum of
  \begin{subequations}
    \begin{align}
      & \frac{1}{(2\pi \ii)^2} \iint_{\Phi} \dd v\dd z \frac{e^{\frac{z^2}{2} - xz} z^{N - k}}{e^{\frac{w^2}{2} - xw} w^{N - k}} \left( \prod^k_{j = 2} \frac{z - \alpha_j}{w - \alpha_j} \right) \frac{-4}{(w - \alpha_1)^2 (z - \alpha_1)(z - w)} \\
      + {}& \frac{1}{2\pi \ii} \int^{\upperintersect}_{\lowerintersect} \dd z \frac{4}{(z - \alpha_1)^3}, \label{eq:single_contour_at_last}
    \end{align}
    by the same transform as \eqref{eq:double_contour_GUE_mean} is transformed into \eqref{eq:int_over_Phi_and_bar}, and the contour $\Phi$ and the integral limits $\upperintersect$, $\lowerintersect$ are as defined in \eqref{eq:int_over_Phi_and_bar}. Note that in \eqref{eq:single_contour_at_last} the contour is between $0$ and $\min(\alpha_1, \dotsc, \alpha_k)$ while in \eqref{eq:int_over_Phi_and_bar:bar} the contour is to the right of $0$ and all $\alpha_1, \dotsc, \alpha_k$. 
  \end{subequations}
\item
  The $2$-fold integral \eqref{eq:gamma'_z_and_w_in_twice_remainder_GUE} can be written as the sum of
  \begin{subequations}
    \begin{align}
      & \frac{1}{(2\pi \ii)^2} \int_{\Sigma''_z} \dd z \int^{\GUEupperleft}_{\overline{\GUEupperleft}} \dd w \frac{e^{\frac{z^3}{3} - xz}}{e^{\frac{w^3}{3} -xw}} \left( \prod^k_{j = 2} \frac{z - \alpha_j}{w - \alpha_j} \right) \frac{-1}{(w - \alpha_1)^2 (z - w)^2} \label{eq:one_cancellation_GUE} \\
      + {}& \frac{1}{(2\pi \ii)^2} \int_{\Sigma''_z} \dd z \int^{\GUEupperleft}_{\overline{\GUEupperleft}, \rightside} \dd w \frac{e^{\frac{z^3}{3} - xz}}{e^{\frac{w^3}{3} -xw}} \left( \prod^k_{j = 2} \frac{z - \alpha_j}{w - \alpha_j} \right) \frac{1}{(w - \alpha_1)^2 (z - w)^2} \label{eq:one_cancellation_GUE_2} \\
      + {}& \frac{1}{(2\pi \ii)^2} \int_{\Sigma''_z} \dd z \int^{\GUEupperright}_{\overline{\GUEupperright}} \dd w \frac{e^{\frac{z^3}{3} - xz}}{e^{\frac{w^3}{3} -xw}} \left( \prod^k_{j = 2} \frac{z - \alpha_j}{w - \alpha_j} \right) \frac{-1}{(w - \alpha_1)^2 (z - w)^2} \label{eq:two_cancellation_GUE} \\
      + {}& \frac{1}{(2\pi \ii)^2} \int_{\Sigma''_z} \dd z \int^{\GUEupperright}_{\overline{\GUEupperright}, \rightside} \dd w \frac{e^{\frac{z^3}{3} - xz}}{e^{\frac{w^3}{3} -xw}} \left( \prod^k_{j = 2} \frac{z - \alpha_j}{w - \alpha_j} \right) \frac{1}{(w - \alpha_1)^2 (z - w)^2}, \label{eq:last_trivial_GUE}
    \end{align}
  \end{subequations}
  where $\Sigma''_z$ is the same as $\Sigma''_z$ in \eqref{eq:gamma''_z_and_u_GUE}, the contour for $w$ in \eqref{eq:one_cancellation_GUE} is the part of $\Sigma'_u$ inside of $\Gamma'_v$, the contour for $w$ in \eqref{eq:one_cancellation_GUE_2} is the part of $\Gamma'_v$ to the right of $\Sigma'_u$, the contour for $w$ in \eqref{eq:two_cancellation_GUE} is the part of $\Sigma'_z$ inside of $\Gamma'_v$, and the contour for $w$ in \eqref{eq:last_trivial_GUE} is the part of $\Gamma'_v$ to the right of $\Sigma'_z$. We note that \eqref{eq:one_cancellation_GUE} cancels with \eqref{eq:gamma''_z_and_u_GUE}, and \eqref{eq:two_cancellation_GUE} cancels with \eqref{eq:two_cancellation_early_GUE}.
\item
  The $1$-fold integral \eqref{eq:int_in_w_out_w_twice_GUE} cancels with \eqref{eq:single_int_gamma_L_GUE}.
\item
  \eqref{eq:last_done_GUE} can be evaluated similar to \eqref{eq:variance_GUE_bar}, and it is $\frac{1}{\pi}(\sqrt{1 - x^2/(4N)} + \arccos(x/\sqrt{4N}) + \bigO(N^{-1/3})$.
\item 
  \eqref{eq:last_added_1_GUE} and \eqref{eq:last_added_2_GUE} are $\bigO(N^{-2/5} \log N)$, and all other integrals from \eqref{eq:first_trivial_GUE} to \eqref{eq:last_trivial_GUE} not mentioned above, are  $\bigO(N^{-2/5})$.
\end{enumerate}
Hence we obtain the final proof of \eqref{eq:GUE_measure_variance_alt}.

\paragraph{Sketch of proof of \eqref{eq:GUE_measure_variance_local_alt}}

Like the proof of \eqref{eq:GUE_measure_variance_alt}, we try to estimate \eqref{eq:easier_part_of_variance_GUE} and \eqref{eq:four_contour_GUE}.

For \eqref{eq:easier_part_of_variance_GUE}, we again use the decomposition \eqref{eq:variance_GUE_total}, but the precise shapes of the contours are not to be the same as in the proof of \eqref{eq:GUE_measure_mean_alt}, and the integral limits $\upperintersect, \lowerintersect$ are not to be $\frac{1}{2}(x + \sqrt{4N - x^2}\ii)$ and $\frac{1}{2}(x - \sqrt{4N - x^2}\ii)$. Instead, we deform the contour $\Phi = \Gamma \times \Sigma$ as in the proof of \eqref{eq:GUE_measure_mean_local_alt}, that is, $\Gamma$ into $\Gamma^{\corner}_{\standard, N}(0, (2\sqrt{N} - x)^{1/2} N^{1/4}/\sqrt{3})$, and deform $\Sigma$ into $\Sigma^{\zig}_{\standard, N}((\sqrt{N} - 2\sqrt{N} - x)^{1/2} N^{1/4}/\sqrt{3})$, and then let the integral limits $\upperintersect, \lowerintersect$ be $\sqrt{N} \pm (2\sqrt{N} - x)^{1/2} N^{1/4} \ii$. Then we have that in the regime that $x = N^{-1/6} \xi + 2\sqrt{N}$ and $-N^{1/15} < \xi < -C$, \eqref{eq:variance_GUE_Phi} is $\bigO(N^{-1/3} (-\xi)^{-1})$, and \eqref{eq:variance_GUE_bar} is $N^{-1/3}(\frac{4}{\pi} \sqrt{-\xi} + 2a_k) + \bigO(N^{-1/3} (-\xi)^{-1})$.

For \eqref{eq:four_contour_GUE}, we take the transforms as in Steps \ref{enu:step_variance_GUE:1} -- \ref{enu:step_variance_GUE:5}, with only one methodological difference: The intersection points $\GUEupperright$, $\GUEupperleft$, $\GUElowerright$ and $\GUElowerleft$ now should be around $\sqrt{N} + (2\sqrt{N} - x)^{1/2} N^{1/4} \ii$, in consistence with our choice of $\Phi$ and $\upperintersect, \lowerintersect$. However, practically we can still use the deformation of the contours in the regime $1.9\sqrt{N} \leq x < 2\sqrt{N} - N^{-1/10}$, because in the regime where we are working, $\frac{1}{2}(x + \sqrt{4N - x^2}\ii)$ and $\sqrt{N} + (2\sqrt{N} - x)^{1/2} N^{1/4} \ii$ are very close to each other. Hence we can still use the saddle point method that is used in the proof of \eqref{eq:GUE_measure_variance_alt}, especially that for the regime $1.9\sqrt{N} \leq x <2\sqrt{N} - N^{-1/10}$.

At last, we derive \eqref{eq:GUE_measure_variance_local_alt} by the method delineated above, with much detail omitted.

\section{Nondegeneracy of the limiting distribution} \label{s.nondegeneracy}

In this section, we prove that the random variable $\Xi_{j}^{(k)}(\totala;\infty)$ defined in \eqref{eq:defn_Xi^infty} is nondegenerate, i.e., Theorem \ref{thm.nondegeneracy}, which tells that the distribution is not supported on a single point. It is equivalent to show that $\log \Xi_{j}^{(k)}(\totala;\infty)$ is nondegenerate. It is not a trivial task, since $\log \Xi_{j}^{(k)}(\totala;\infty)$ is a non-linear functional on the externded Airy process. Our proof relies on that $\log \Xi_{j}^{(k)}(\totala;\infty)$ is the limit of $\log(N^{1/3} \lvert x_{j1} \rvert^2)$ by Theorem \ref{thm:main_thm_2} (up to a constant), which is a non-linear functional on the GUE minor process with external source. The advantage of the latter is that the eigenvalue distribution of a GUE-type matrix has a log-gas representation (c.f. (\ref{eq:distr_sigma_1})) that does not pass to the limiting processes, e.g.~an Airy-type process. We make use of the log-gas representation, and prove Lemma \ref{lem:tough_lem}, which leads to the proof of Theorem \ref{thm.nondegeneracy} in a straightforward way.

Since in the statement and proof of Lemma \ref{lem:tough_lem} we will condition on $\lambda_1, \dotsc, \lambda_{N - 1}$, $\sigma_1, \dotsc, \sigma_{j - 1}, \sigma_{j + 1}, \dotsc, \sigma_N$ and play with the randomness of $\sigma_j$ only, we will denote by $\omega = (\lambda_1, \dotsc, \lambda_{N - 1}, \sigma_1, \dotsc, \sigma_{j - 1}, \sigma_{j + 1}, \dotsc, \sigma_N)$ a generic realization of the  collection of the given eigenvalues, for notational simplicity.

\begin{lem} \label{lem:tough_lem}
  There exist $M, \epsilon, \epsilon', \epsilon'' > 0$, such that if $N$ is large enough, there exists an event $A $ of $\omega$,  with $\Prob(A) > \epsilon$, and the following estimates hold  for any fixed $\omega \in A$,
 \begin{align}
   \Prob \left( M^{-1} < N^{\frac{1}{3}} \lvert x_{j1} \rvert^2 < M \mid \omega \right) > \epsilon', \quad \text{and} \quad \var \left( \log (N^{\frac{1}{3}}  \lvert x_{j1} \rvert^2) \mid \omega \text{ and } M^{-1} < N^{\frac{1}{3}} \lvert x_{j1} \rvert^2 < M \right) > \epsilon'', \label{032401}
 \end{align}
\end{lem}

\begin{proof}[Proof of Theorem \ref{thm.nondegeneracy} by Lemma \ref{lem:tough_lem}]
  Lemma \ref{lem:tough_lem} shows that for all large enough $N$, the followings hold:
  \begin{enumerate}[label=(\roman*)]
  \item 
    $\Prob(M^{-1} < N^{1/3} \lvert x_{j1} \rvert^2 < M) > \Prob(A \cap \{ M^{-1} < N^{1/3} \lvert x_{j1} \rvert^2 < M \}) > \epsilon \epsilon'$, and
  \item
    the conditional variance
    \begin{equation}
      \begin{split}
        & \var \left( \log(N^{\frac{1}{3}} \lvert x_{j1} \rvert^2) \mid M^{-1} < N^{\frac{1}{3}} \lvert x_{j1} \rvert^2 < M \right) \\
        \geq {}& \frac{\Prob(A \cap \{ M^{-1} < N^{1/3} \lvert x_{j1} \rvert^2 < M \})}{\Prob(M^{-1} < N^{1/3} \lvert x_{j1} \rvert^2 < M)} \var \left( \log(N^{\frac{1}{3}} \lvert x_{j1} \rvert^2) \mid A \cap \{ M^{-1} < N^{\frac{1}{3}} \lvert x_{j1} \rvert^2 < M \} \right) \\
        > {}& \epsilon \epsilon' \epsilon''.
      \end{split}
    \end{equation}
  \end{enumerate}
  Using Theorem \ref{thm:main_thm_2}, we have $\Prob(M^{-1} < (3\pi/2)^{1/3} \Xi_{j}^{(k)}(\totala;\infty) < M) > \epsilon \epsilon'$ and $\var(\log((3\pi/2)^{1/3} \Xi_{j}^{(k)}(\totala;\infty)) \mid M^{-1} < (3\pi/2)^{1/3} \Xi_{j}^{(k)}(\totala;\infty) < M) > \epsilon \epsilon' \epsilon''$. Then we can conclude the proof of Theorem \ref{thm.nondegeneracy} easily.
\end{proof}

The remaining part of this section is devoted to the proof of Lemma \ref{lem:tough_lem}. To make notations simple, in this section we denote $\lambdatilde_i = N^{1/6}(\lambda_i - 2\sqrt{N})$ and $\sigmatilde_i = N^{1/6}(\sigma_i - 2\sqrt{N})$. We recall that $\sigmatilde_i$'s and $\lambdatilde_i$'s converge jointly in distribution to $\xi^{(k)}_i$'s and $\xi^{(k - 1)}_i$'s, by Lemma \ref{lem:Airy_limit}. We further denote $F_j = \log(N^{1/3} \lvert x_{j1} \rvert^2)$. 

\begin{proof}[Proof of Lemma \ref{lem:tough_lem}] 
  We discuss first the $j = 1$ case in detail, and then extend the discussion to the $j > 1$ case.
  
  \paragraph{The $j = 1$ case}
  Before we commence the detailed discussion, we first provide a heuristic outline of the construction of the event $A$ and also the proof of (\ref{032401}). We start with the first estimate in (\ref{032401}). It suffices to find an event $A$, so that for each $\omega\in A$, the conditional probability of $F_1\equiv F_1(\tilde{\sigma}_1; \omega):=\log (N^{1/3}|x_{11}|^2)$ being bounded both below and above is non-negligible. From (\ref{20041121}), it is clear that  $F_1(\tilde{\sigma}_1; \omega)$ is an increasing function of $\tilde{\sigma}_1$ in the domain $(\tilde{\lambda}_1, \infty)$, given $\omega$. We will construct an $A$, so that $F_1(\tilde{\lambda}_1+\varepsilon; \omega)$ and $F_1(M; \omega)$ are bounded below and above, respectively, and further the conditional probability of $\tilde{\sigma}_1\in [\tilde{\lambda}_1+\varepsilon, M]$ given $\omega\in A$ is greater than $1/3$. Here $\varepsilon$ and $M$ are some sufficiently small and large constants respectively, which will be specified as $\epsilon_3$ and $M_1$ in the detailed discussion. The desired event $A$ is eventually given by $A_5$ in (\ref{event A5}), via a successive construction of a sequence of events $A_5\subseteq A_4\subseteq A_3\subseteq A_2\subseteq A_1$. Among all of them, especially, on event $A_2$,  $\tilde{\sigma}_1$ is bounded above with non-negligible conditional probability, and then on  $A_3\subseteq A_2$,  $\tilde{\sigma}_1$ is further bounded below by  $\tilde{\lambda}_1+\varepsilon$ with non-negligible conditional  probability. In the final $A_5$, we also take into account the lower and upper bounds of  $F_1(\tilde{\lambda}_1+\varepsilon; \omega)$ and $F_1(M; \omega)$ respectively. In order to show the second estimate in (\ref{032401}), we shall show the following two points for $\omega\in A_5$: (i) $\tilde{\sigma}_1$ spreads fairly even in the interval $[\tilde{\lambda}_1+\varepsilon, M]$; (ii) $F_1(\tilde{\sigma};M)$ varies monotonically and significantly in the interval $[\tilde{\lambda}_1+\varepsilon, M]$. The proof of (i) will rely on an upper bound on the derivative of the conditional density function of $\tilde{\sigma}_1$ given $\omega\in A_5$ (c.f. (\ref{eq:est_f_dev})), and the proof of (ii) will rely on a lower bound of the derivative of $F_1(\tilde{\sigma};\omega)$ (c.f. (\ref{eq:difference_F_lower})). 
  
  With the above heuristics, we conduct the detailed proof in the sequel. We note that $\xi^{(k - 1)}_1 > \xi^{(k)}_2$ almost surely and they are both continuous random variables. Hence there exist some $c_1 \in \realR$ and $\epsilon_2 > 0$ such $\Prob(\xi^{(k)}_1 \in (c_1, c_1 + \epsilon_2) \text{ and } \xi^{(k - 1)}_2 < c_1 - \epsilon_2) > \epsilon_1$ for some $\epsilon_1 > 0$. Hence the event $A_1$ defined by
  \begin{equation} \label{def of A1}
    A_1 = \{ \lambdatilde_1 \in (c_1, c_1 + \epsilon_2) \text{ and } \sigmatilde_2 < c_1 - \epsilon_2 \}
  \end{equation}
  satisfies that $\Prob(A_1) > \epsilon_1$ for large enough $N$.
  
  Let $p(\omega)$ be the marginal density of $\omega$, whose formula is not relevant. If $\omega$ is fixed, the conditional density of $\sigma_1$ given $\omega$ is, by \cite[Theorem 1]{Adler-van_Moerbeke-Wang11}
  \begin{equation} \label{eq:distr_sigma_1}
    p_{\omega}(\sigma) = \frac{1}{C_{\omega}} \exp(-f_{\omega, \alpha_1}(\sigma)) \Id(\sigma > \lambda_1), \quad \text{where} \quad f_{\omega, \alpha_1}(\sigma) = \frac{\sigma^2}{2} - \alpha_1 \sigma - \sum^N_{i = 2} \log(\sigma - \sigma_i),
  \end{equation}
  for some constant $C_{\omega}$, or equivalently, the conditional density for $\sigmatilde_1$ is 
  \begin{equation} \label{def of f_tilde}
    \tilde{p}_{\omega}(\sigmatilde) = \frac{p_{\omega}(\sigma)}{N^{1/6}} = \frac{1}{\tilde{C}_{\omega, a_k}} \exp(-\tilde{f}_{\omega, a_k}(\sigmatilde)) \Id(\sigmatilde > \lambdatilde_1), \quad \text{where} \quad \tilde{f}_{\omega, a_k}(\sigmatilde) = N^{\frac{1}{3}} \sigmatilde - a_k \sigmatilde + N^{-\frac{1}{3}} \frac{\sigmatilde^2}{2} - \sum^N_{i = 2} \log(\sigmatilde - \sigmatilde_i),
  \end{equation}
  and $\tilde{C}_{\omega, a_k} = N^{-N/6} \exp(2N^{2/3} a_k) C_{\omega}$. Hence
  \begin{equation} \label{eq:derivatives_ftilde}
    \tilde{f}'_{\omega, a_k}(\sigmatilde) = N^{\frac{1}{3}} - a_k + N^{-\frac{1}{3}} \sigmatilde - \sum^N_{i = 2} \frac{1}{\sigmatilde - \sigmatilde_i}, \quad \tilde{f}''_{\omega, a_k}(\sigmatilde) = N^{-\frac{1}{3}} + \sum^N_{i = 2} \frac{1}{(\sigmatilde - \sigmatilde_i)^2}, \quad \text{and} \quad \int^{\infty}_{\lambdatilde_1} \tilde{p}_{\omega}(\sigmatilde) \dd \sigmatilde = 1.
  \end{equation}
  We note that as a random variable independent of $N$,  $\xi^{(k)}_1 < +\infty$ almost surely, and thus there exists a constant $M_1 > c_1 + \epsilon_2$ such that $\Prob(\xi^{(k)}_1 > M_1)<\frac{\epsilon_1}{6}$. Consequently,   for large enough $N$,
  \begin{equation} \label{eq:tail_est_sigmatilde_1}
    \Prob(\sigmatilde_1 > M_1) = \int p(\omega) \dd \omega \int^{\infty}_{M_1} \tilde{p}_{\omega}(\sigmatilde) \dd \sigmatilde < \frac{\epsilon_1}{6}.
  \end{equation}
  Then the event $A_2$ defined by
  \begin{equation}
    A_2 = \left\{ \omega \in A_1 \mid \int^{M_1}_{\lambdatilde_1} \tilde{p}_{\omega}(\sigmatilde) \dd\sigmatilde > \frac{2}{3} \right\}
  \end{equation}
  satisfies $\Prob(A_2) > \epsilon_1/2$ for large enough $N$. Otherwise we will have
  \begin{equation}
    \int p(\omega) \dd \omega \int^{\infty}_{M_1} \tilde{p}_{\omega}(\sigmatilde) \dd \sigmatilde \geq \int_{A_1 \setminus A_2} p(\omega) \dd \omega \int^{\infty}_{M_1} \tilde{p}_{\omega}(\sigmatilde) \dd \sigmatilde >\frac{1}{3} \Prob(A_1 \setminus A_2)  \geq \frac{\epsilon_1}{6},
  \end{equation}
  contradictory to \eqref{eq:tail_est_sigmatilde_1}.
  
  Next, since $\xi^{(k)}_1 > \xi^{(k - 1)}_1$ almost surely and $\xi^{(k)}_1, \xi^{(k - 1)}_1$ are both continuous random variables, there exist $\epsilon_3 > 0$ such that $\Prob(\xi^{(k)}_1 > \xi^{(k - 1)}_1 + \epsilon_3) > 1 - \epsilon_1/12$, and then for large enough $N$,
  \begin{equation} \label{eq:head_est_sigmatilde_1}
    \Prob(\sigmatilde_1 \leq  \lambdatilde_1 + \epsilon_3) = \int p(\omega) \dd \omega \int^{\lambdatilde_1 + \epsilon_3}_{\lambdatilde_1} \tilde{p}_{\omega}(\sigmatilde) \dd \sigmatilde < \frac{\epsilon_1}{12}.
  \end{equation}
  Then the event $A_3$ defined by
  \begin{equation}
    A_3 = \left\{ \omega \in A_2 \mid \int^{\lambdatilde_1 + \epsilon_3}_{\lambdatilde_1} \tilde{p}_{\omega}(\sigmatilde) \dd\sigmatilde < \frac{1}{3} \right\}
  \end{equation}
  satisfies $\Prob(A_3) > \epsilon_1/4$ for large enough $N$. Otherwise, we will have
  \begin{equation}
    \int p(\omega) \dd \omega \int^{\lambdatilde_1 + \epsilon_3}_{\lambdatilde_1} \tilde{p}_{\omega}(\sigmatilde) \dd \sigmatilde \geq \int_{A_2 \setminus A_3} p(\omega) \dd \omega \int^{\lambdatilde_1 + \epsilon_3}_{\lambdatilde_1} \tilde{p}_{\omega}(\sigmatilde) \dd \sigmatilde > \frac{1}{3}\Prob(A_2 \setminus A_3)  \geq  \frac{\epsilon_1}{12},
  \end{equation}
  contradictory to \eqref{eq:head_est_sigmatilde_1}. We note that if $\omega \in A_3$, then
  \begin{equation} \label{eq:comparison_int_ptilde}
    \begin{gathered}
      \Prob(\sigmatilde_1 > M_1 \mid \omega) = \int^{\infty}_{M_1} \tilde{p}_{\omega}(\sigmatilde) \dd \sigmatilde < \frac{1}{3}, \quad \Prob(\sigmatilde_1 \in (\lambdatilde_1, \lambdatilde_1 + \epsilon_3) \mid \omega) = \int^{\lambdatilde_1 + \epsilon_3}_{\lambdatilde_1} \tilde{p}_{\omega}(\sigmatilde) \dd \sigmatilde < \frac{1}{3}, \\
      \text{and} \quad \Prob(\sigmatilde_1 \in [\lambdatilde_1 + \epsilon_3, M_1] \mid \omega) = \int^{M_1}_{\lambdatilde_1 + \epsilon_3} \tilde{p}_{\omega}(\sigmatilde) \dd \sigmatilde > \frac{1}{3}.
    \end{gathered}
  \end{equation}
  
  To proceed with the construction of the event $A_4\subseteq A_3$, we will need the following lemma which is essentially used to control $\tilde{f}''_{\omega, a_k}(\tilde{\sigma})$ (c.f. (\ref{eq:derivatives_ftilde})). The control of $\tilde{f}''_{\omega, a_k}(\tilde{\sigma})$ will be needed to further control $\tilde{f}'_{\omega, a_k}(\tilde{\sigma})$, which will be necessary, as we mentioned at the beginning of the proof.  The  proof of the following lemma will be postponed to the end of this section. 
  \begin{lem} \label{lem:C_max}
  Let $L > 0$ be any constant independent of $N$. If $N$ is large enough, there is a constant $C_L$ such that for all $x \in [-L, L]$
  \begin{equation} \label{eq:est_square_inverse_sum}
    \E \left( \sum^N_{i = 1} \frac{1}{(\sigmatilde_i - x)^2 + 1} \right) < C_L.
  \end{equation}
\end{lem}
  
Let $L>0$ be as in Lemma \ref{lem:C_max}. Now we choose $L $ sufficiently large such that $[-L, L] \supseteq (c_1 - \epsilon_2, M_1)$, and define
  \begin{equation}
    A_4 = \left\{ \omega \in A_3 \mid \int^{\infty}_{\lambdatilde_1} \tilde{p}_{\omega}(\sigmatilde_1) \sum^N_{i = 1} \frac{1}{(\sigmatilde_i - c_1)^2 + 1} \dd\sigmatilde_1 < \frac{C_L}{\epsilon_1/8} \right\}.
  \end{equation}
  We have that $\Prob(A_4) > \epsilon_1/8$ for all large enough $N$. Otherwise,
  \begin{equation} \label{eq:ineq_avarage_at_c1}
    \begin{split}
      \E \left( \sum^N_{i = 1} \frac{1}{(\sigmatilde_i - c_1)^2 + 1} \right) = {}& \int p(\omega) \dd \omega \int^{\infty}_{\lambdatilde_1} \tilde{p}_{\omega}(\sigmatilde_1) \sum^N_{i = 1} \frac{1}{(\sigmatilde_i - c_1)^2 + 1} \dd\sigmatilde_1 \\
      \geq {}& \int_{A_3 \setminus A_4} p(\omega) \dd \omega \int^{\infty}_{\lambdatilde_1} \tilde{p}_{\omega}(\sigmatilde_1) \sum^N_{i = 1} \frac{1}{(\sigmatilde_i - c_1)^2 + 1} \dd\sigmatilde_1 \\
      \geq {}& \Prob(A_3 \setminus A_4) \frac{C_L}{\epsilon_1/8} \\
      \geq {}& C_L,
    \end{split}
  \end{equation}
  contradictory to \eqref{eq:est_square_inverse_sum}.
  
  Now consider the function $F_1 = \log(N^{1/3} \lvert x_{11} \rvert^2)$ defined by \eqref{20041121}. For a fixed $\omega \in A_4$, $F_1$ depends on $\sigma_1$, or equivalently $\sigmatilde_1$. Below we express it as $F_1(\sigmatilde_1; \omega)$ as a function of $\sigmatilde_1$. $F_1(\sigmatilde_1; \omega)$ is an increasing function as $\sigmatilde_1 \in (\lambdatilde_1, +\infty)$. In Theorem \ref{thm:main_thm_2} we have shown that the random variable $N^{1/3}|x_{11}|^2$ converges weakly to the random variable $(3\pi/2)^{1/3} \Xi_j^{(k)}(\totala; \infty)$, so there exists $M_2 > 1$ such that for large enough $N$,
  \begin{equation}
    \Prob(F_1 < -\log M_2 \text{ or } F_1 > \log M_2) < \frac{\epsilon_1}{16}.
  \end{equation}
  Then the event $A_5$ defined by
  \begin{align}
    A_5 = \left\{ \omega \in A_4 \mid [-\log M_2, \log M_2] \cap [F_1(\lambdatilde_1 + \epsilon_3; \omega), F_1(M_1; \omega)] \neq \emptyset \right\} \label{event A5}
  \end{align}
  satisfies $\Prob(A_5) > \epsilon_1/16$ for large enough $N$. Otherwise, we will have
  \begin{equation}
    \Prob(F_1 < -\log M_2 \text{ or } F_1 > \log M_2) > \int_{A_4 \setminus A_5} p(\omega) \dd \omega \Prob(\sigmatilde_1 \in [\lambdatilde_1 + \epsilon_3, M_1] \mid \omega) \geq \frac{\epsilon_1}{16} \frac{1}{3} = \frac{\epsilon_1}{48}.
  \end{equation}
  
  We have, by \eqref{eq:defn_P_GUE_tilde} and \eqref{eq:GUE_P_n_integral},
  \begin{equation}
    F_1(\sigmatilde; \omega) = \int_{\bigcup^N_{i = 2} (\sigmatilde_i, \lambdatilde_{i - 1}]} \frac{-1}{\sigmatilde - x} \dd x + \frac{1}{3} \log N.
  \end{equation}
  Hence with $\omega \in A_4$, we have
  \begin{equation} \label{eq:difference_F_upper}
    F_1(M_1; \omega) - F_1(\lambdatilde_1 + \epsilon_3; \omega) < \int_{(-\infty, \lambdatilde_1]} \left( \frac{1}{\lambdatilde_1 + \epsilon_3 - x} - \frac{1}{M_1 - x} \right) \dd x = \log \frac{M_1 - \lambdatilde_1}{\epsilon_3} < \log \frac{M_1 - c_1}{\epsilon_3},
  \end{equation}
  and for $\sigmatilde$ between $\lambdatilde_1 + \epsilon_3$ and $M_1$,
  \begin{equation} \label{eq:difference_F_lower}
    \frac{\dd}{\dd\sigmatilde} F_1(\sigmatilde; \omega) > \int_{[c_1 - \epsilon_2, c_1]} \frac{1}{(\sigmatilde - x)^2} \dd x = \frac{1}{M_1 - c_1} - \frac{1}{M_1 - c_1 + \epsilon_2},
  \end{equation}
  where we recalled the domain in \eqref{def of A1}.
  
  By \eqref{eq:difference_F_upper}, with $M_3 = M_2 \cdot ((M_1 - c_1)/\epsilon_3)$, we have that
  \begin{equation} \label{eq:defn_A_5}
    [F_1(\lambdatilde_1 + \epsilon_3; \omega), F_1(M_1; \omega)] \subseteq [-\log M_3, \log M_3], \quad  \text{if $\omega \in A_5$}.
  \end{equation}
  Hence we have for all $\omega \in A_5$,
  \begin{equation} \label{eq:A_5_prop}
    \begin{split}
      \Prob(-\log M_3 < F_1(\sigmatilde_1; \omega) < \log M_3 \mid \omega) \geq {}& \Prob(F_1(\lambdatilde_1 + \epsilon_3; \omega) < F_1(\sigmatilde_1; \omega) < F_1(M_1; \omega) \mid \omega) \\
      = {}& \Prob(\sigmatilde_1 \in [\lambdatilde_1 + \epsilon_3, M_1] \mid \omega) > \frac{1}{3}.
    \end{split}
  \end{equation}
  Therefore, choosing $A=A_5, \epsilon'=\frac13, \epsilon=\frac{\epsilon_1}{16}$, we conclude the proof of the first estimate in (\ref{032401}).
  
 Next, we prove the second estimate in (\ref{032401}).  We note that by \eqref{eq:derivatives_ftilde}, $\tilde{f}''_{\omega, a_k}$ is positive and decreasing on $(\lambdatilde_1, +\infty)$. For $\omega \in A_5$, we have that if $\sigmatilde > \lambdatilde_1$, then for large enough $N$,
  \begin{equation} \label{bound for f''}
    \begin{split}
      0 < \tilde{f}''_{\omega, a_k}(\sigmatilde) < {}& N^{-\frac{1}{3}} + \sum^N_{i = 2} \frac{1}{(\sigmatilde_i - c_1)^2} \\
      < {}& N^{-\frac{1}{3}} + \frac{1 + \epsilon^2_2}{\epsilon^2_2} \sum^N_{i = 2} \frac{1}{(\sigmatilde_i - c_1)^2 + 1} \\
      < {}& N^{-\frac{1}{3}} + \frac{1 + \epsilon^2_2}{\epsilon^2_2} \int^{\infty}_{\lambdatilde_1} \tilde{p}_{\omega}(\sigmatilde_1) \sum^N_{i = 1} \frac{1}{(\sigmatilde_i - c_1)^2 + 1} \dd \sigmatilde_1 \\
      < {}& \frac{8 C_L}{\epsilon_1 \epsilon^2_2},
    \end{split}
  \end{equation}
  where we use the property that $\sigmatilde_i + \epsilon_2 < c_1 < \sigmatilde$ ($i = 2, \dotsc, N$).
  
  Now with the aid of \eqref{bound for f''}, we claim that for all $\omega \in A_5$, if we take $C_A$ to be a constant bigger than both $\epsilon^{-1}_3 + 8 C_L(\epsilon_1 \epsilon^2_2)^{-1} (M_1 - c_1)$ and $16 C_L (\epsilon_1 \epsilon^2_2)^{-1}(M_1 - c_1)$, then
  \begin{equation} \label{eq:est_f_dev}
    \lvert \tilde{f}'_{\omega, a_k}(\sigmatilde) \rvert < C_A \quad \text{if $\sigmatilde \in [\lambdatilde_1 + \epsilon_3, M_1]$.}
  \end{equation}
  Since $\tilde{f}'_{\omega, a_k}(\sigmatilde)$ is increasing on $(\lambdatilde_1, +\infty)$, to prove \eqref{eq:est_f_dev}, it suffices to check that
  \begin{align}
    \tilde{f}'_{\omega, a_k}(\lambdatilde_1 + \epsilon_3) > {}& -C_A, \label{eq:f_deriv_lower} \\
    \tilde{f}'_{\omega, a_k}(M_1) < {}& C_A. \label{eq:f_deriv_upper}
  \end{align}
  If \eqref{eq:f_deriv_lower} does not hold, we have that, in light of \eqref{bound for f''},
  \begin{equation}
    \begin{split}
      \tilde{f}'_{\omega, a_k}(2M_1 - c_1) = {}& \tilde{f}'_{\omega, a_k}(\lambdatilde_1 + \epsilon_3) + \int^{2M_1 - c_1}_{\lambdatilde_1 + \epsilon_3} \tilde{f}''_{\omega, a_k}(\sigmatilde) \dd \sigmatilde \\
      \leq {}& \tilde{f}'_{\omega, a_k}(\lambdatilde_1 + \epsilon_3) + [(2M_1 - c_1) - (\lambdatilde_1 + \epsilon_3)] \frac{8 C_L}{\epsilon_1 \epsilon^2_2} \\
      \leq {}& -C_A + 2(M_1 - c_1)\frac{8 C_L}{\epsilon_1 \epsilon^2_2} < 0.\\
    \end{split}
  \end{equation}
  It implies that $\tilde{f}'_{\omega, a_k}(\sigmatilde)$ is negative on $(\lambdatilde_1, 2M_1 - c_1)$, and then $\tilde{f}_{\omega, a_k}(\sigmatilde)$ is decreasing there. Consequently, by \eqref{def of f_tilde}, $\tilde{p}_{\omega}(\sigmatilde)$ is increasing there. We hence have
  \begin{equation} \label{091201}
    \frac{\Prob(\sigmatilde_1 \in (M_1, +\infty) \mid \omega)}{\Prob(\sigmatilde_1 \in [\lambdatilde_1 + \epsilon_3, M_1] \mid \omega)} \geq \frac{\Prob(\sigmatilde_1 \in (M_1, 2M_1 - c_1) \mid \omega)}{\Prob(\sigmatilde_1 \in [\lambdatilde_1 + \epsilon_3, M_1] \mid \omega)} = \frac{\int^{2M_1 - c_1}_{M_1} \tilde{p}_{\omega}(\sigmatilde) {\rm d}\sigmatilde}{\int^{M_1}_{\lambdatilde_1 + \epsilon_3} \tilde{p}_{\omega}(\sigmatilde) {\rm d}\sigmatilde} > 1,
  \end{equation}
  which is contradictory to \eqref{eq:comparison_int_ptilde}. On the other hand, if \eqref{eq:f_deriv_upper} does not hold, then
  \begin{equation}
    \tilde{f}'_{\omega, a_k}(\lambdatilde_1) = \tilde{f}'_{\omega, a_k}(M_1) - \int^{M_1}_{\lambdatilde_1} \tilde{f}''_{\omega, a_k}(\sigmatilde) \dd \sigmatilde \geq \tilde{f}'_{\omega, a_k}(M_1) - (M_1 - \lambdatilde_1) \frac{8 C_L}{\epsilon_1 \epsilon^2_2} \geq C_A - (M_1 - c_1) \frac{8 C_L}{\epsilon_1 \epsilon^2_2} \geq \epsilon^{-1}_3.
  \end{equation}
  Hence we have, by the monotonicity of $\tilde{f}'_{\omega, a_k}(\sigmatilde)$, that $\tilde{f}'_{\omega, a_k}(\sigmatilde) \geq \epsilon^{-1}_3$ for all $\sigmatilde \in [\lambdatilde_1, M_1]$. Hence we have
  \begin{equation}
    \begin{split}
      \int^{\lambdatilde_1 + \epsilon_3}_{\lambdatilde_1} \tilde{p}_{\omega}(\sigmatilde) \dd \sigmatilde = {}& \frac{1}{\tilde{C}_{\omega, a_k}} \int^{\lambdatilde_1 + \epsilon_3}_{\lambdatilde_1} \exp(-\tilde{f}_{\omega, a_k}(\sigmatilde)) \dd \sigmatilde \\
      \geq {}& \frac{1}{\tilde{C}_{\omega, a_k}} \int^{\lambdatilde_1 + \epsilon_3}_{\lambdatilde_1} \exp[-\tilde{f}_{\omega, a_k}(\lambdatilde_1 + \epsilon_3) + \epsilon^{-1}_3(\lambdatilde_1 + \epsilon_3 - \sigmatilde)] \dd \sigmatilde \\
      = {}& (e - 1) \epsilon_3 \frac{\exp[-\tilde{f}_{\omega, a_k}(\lambdatilde_1 + \epsilon_3)]}{\tilde{C}_{\omega, a_k}},
    \end{split}
  \end{equation}
  and
  \begin{equation}
    \begin{split}
      \int^{M_1}_{\lambdatilde_1 + \epsilon_3} \tilde{p}_{\omega}(\sigmatilde) \dd \sigmatilde = {}& \frac{1}{\tilde{C}_{\omega, a_k}} \int^{M_1}_{\lambdatilde_1 + \epsilon_3} \exp(-\tilde{f}_{\omega, a_k}(\sigmatilde)) \dd \sigmatilde \\
      \leq {}& \frac{1}{\tilde{C}_{\omega, a_k}} \int^{M_1}_{\lambdatilde_1 + \epsilon_3} \exp[-\tilde{f}_{\omega, a_k}(\lambdatilde_1 + \epsilon_3) - \epsilon^{-1}_3(\sigmatilde - \lambdatilde_1 - \epsilon_3)] \dd \sigmatilde \\
      < {}& \epsilon_3 \frac{\exp[-\tilde{f}_{\omega, a_k}(\lambdatilde_1 + \epsilon_3)]}{\tilde{C}_{\omega, a_k}}.
    \end{split}
  \end{equation}
  Hence we have
  \begin{equation}
    \frac{\Prob(\sigmatilde_1 \in (\lambdatilde_1, \lambdatilde_1 + \epsilon_3) \mid \omega)}{\Prob(\sigmatilde_1 \in [\lambdatilde_1 + \epsilon_3, M_1] \mid \omega)} = \frac{\int^{\lambdatilde_1 + \epsilon_3}_{\lambdatilde_1} \tilde{p}_{\omega}(\sigmatilde) \dd\sigmatilde}{\int^{M_1}_{\lambdatilde_1 + \epsilon_3} \tilde{p}_{\omega}(\sigmatilde) \dd\sigmatilde} > 1,
  \end{equation}
  which is also contradictory to \eqref{eq:comparison_int_ptilde}. Hence we have \eqref{eq:est_f_dev}.
  
  We have that if $\omega \in A_5$ and $\sigmatilde_1 \in [\lambdatilde_1 + \epsilon_3, M_1]$, then 
  \begin{enumerate*}[label=(\roman*)]
  \item
    $\sigmatilde_1$ spreads fairly even due to \eqref{eq:est_f_dev}; and
  \item
     $F_1(\sigmatilde_1; \omega)$ varies monotonically and significantly, by \eqref{eq:difference_F_lower}.
  \end{enumerate*}
  We therefore conclude that for all $\omega \in A_5$, for some $\epsilon_4 > 0$ independent of $\omega$,
  \begin{equation} \label{eq:conditional_var_first_est}
    \var(F_1(\sigmatilde_1; \omega) \mid \omega \text{ and } \sigmatilde_1 \in [\lambdatilde_1 + \epsilon_3, M_1]) > \epsilon_4.
  \end{equation}
  Then using \eqref{eq:A_5_prop}, we have
  \begin{equation} \label{eq:final_nondegeneracy}
    \begin{split}
      & \var(F_1(\sigmatilde_1; \omega) \mid \omega \text{ and } -\log M_3 < F_1(\sigmatilde_1; \omega) < \log M_3) \\
      \geq {}& \frac{\Prob(\sigmatilde_1 \in [\lambdatilde_1 + \epsilon_3, M_1] \mid \omega)}{\Prob(-\log M_3 < F_1(\sigmatilde_1; \omega) < \log M_3 \mid \omega)} \var(F_1(\sigmatilde_1; \omega) \mid \omega \text{ and } \sigmatilde_1 \in [\lambdatilde_1 + \epsilon_3, M_1]) \\
      > & \frac{1}{3} \epsilon_4.
    \end{split}
  \end{equation}

  Taking $A = A_5$, $M = M_3$, $\epsilon = \epsilon_1/16$, $\epsilon' = 1/3$ and $\epsilon'' = \epsilon_4$, we see that Lemma \ref{lem:tough_lem} is verified in the $j = 1$ case by \eqref{eq:A_5_prop} and \eqref{eq:final_nondegeneracy}.

  \paragraph{The $j > 1$ case}

  Similar to $A_1$ in the proof of the $j = 1$ case, we can define
  \begin{equation}
    B_1 = \{ \lambdatilde_j \in (d_1, d_1 + \delta_2) \text{ and } \lambdatilde_{j - 1} \in (d_2 - \delta_2, d_2) \text{ and } \sigmatilde_{j + 1} < d_1 - \delta_2 \text{ and } \sigmatilde_{j - 1} > d_2 + \delta_2 \text{ and } \sigmatilde_1 < N_1 \},
  \end{equation}
  such that $\Prob(B_1) > \delta_1$ for large enough $N$, where $d_1, d_2, N_1$ are real numbers, $\delta_1, \delta_2$ are positive numbers, and $d_1 + \delta_2 < d_2 - \delta_2$, $d_2 + \delta_2 < N_1$. We additionally require that $\delta_2 < (d_2 - d_1)/6$ for later use.

  Next, let $p(\omega)$ be the marginal density of $\omega$, and the conditional density of $\sigma_j$, as $\omega$ is fixed, is
  \begin{equation}
    p_{\omega}(\sigma) = \frac{1}{C_{\omega}} \exp(-g_{\omega, \alpha_1}(\sigma))\Id(\lambda_j < \sigma < \lambda_{j - 1}) \quad \text{where} \quad g_{\omega, \alpha_1} = \frac{\sigma^2}{2} - \alpha_1 \sigma - \sum^N_{i = j + 1} \log(\sigma - \sigma_i) - \sum^{j - 1}_{i = 1} \log(\sigma_i - \sigma),
  \end{equation}
  for some constant $C_{\omega}$, or equivalently, the conditional density for $\sigmatilde_j$ is
  \begin{equation}
    \begin{gathered}
      \tilde{p}_{\omega}(\sigmatilde) = \frac{p_{\omega}(\sigma)}{N^{1/6}} = \frac{1}{\tilde{C}_{\omega, a_k}} \exp(-\tilde{g}_{\omega, a_k}(\sigmatilde)) \Id(\lambdatilde_j < \sigmatilde < \lambdatilde_{j - 1}), \\
      \text{where} \quad \tilde{g}_{\omega, a_k}(\sigmatilde) = N^{\frac{1}{3}} \sigmatilde - a_k\sigmatilde + N^{-\frac{1}{3}} \frac{\sigmatilde^2}{2} - \sum^N_{i = j + 1} \log(\sigmatilde - \sigmatilde_i) - \sum^{j - 1}_{i = 1} \log(\sigmatilde_i - \sigmatilde),
    \end{gathered}
  \end{equation}
  and $\tilde{C}_{\omega, a_k} = N^{-N/6} \exp(2N^{2/3} a_k) C_{\omega}$. Hence
  \begin{equation} \label{eq:derivatives_gtilde}
    \begin{gathered}
      \tilde{g}'_{\omega, a_k}(\sigmatilde) = N^{\frac{1}{3}} - a_k + N^{-\frac{1}{3}} \sigmatilde - \sum^N_{i = j + 1} \frac{1}{\sigmatilde - \sigmatilde_i} + \sum^{j - 1}_{i = 1} \frac{1}{\sigmatilde_i - \sigmatilde}, \quad \tilde{g}''_{\omega, a_k}(\sigmatilde) = N^{-\frac{1}{3}} + \sum_{i\in  \llbracket1,  N \rrbracket  \setminus \{ j \}} \frac{1}{(\sigmatilde - \sigmatilde_i)^2}, \\
      \text{and} \quad \int^{\lambdatilde_{j-1}}_{\lambdatilde_j} \tilde{p}_{\omega}(\sigmatilde) \dd \sigmatilde = 1.
    \end{gathered}
  \end{equation}
  Then analogous to $A_2$ and $A_3$ in the proof of the $j = 1$ case, we can define
  \begin{equation}
    B_3 = \left\{ \omega \in B_1 \mid \int^{\lambdatilde_{j - 1}}_{\lambdatilde_{j - 1} - \delta_3} \tilde{p}_{\omega}(\sigmatilde) \dd \sigmatilde < \frac{1}{3} \text{ and } \int^{\lambdatilde_j + \delta_3}_{\lambdatilde_j} \tilde{p}_{\omega}(\sigmatilde) \dd \sigmatilde < \frac{1}{3} \right\}, \label{092310}
  \end{equation}
  such that $\Prob(B_3) > \delta_1/4$ for large enough $N$, where $\delta_3 > 0$. We additionally require that $\delta_3 < \delta_2/2$ for later use. It is straightforward to see that for all $\omega \in B_3$,
  \begin{equation} \label{eq:1/3_est_sigmatilde_j}
    \Prob(\sigmatilde_j \in [\lambdatilde_j + \delta_3, \lambdatilde_{j - 1} - \delta_3]) > \frac{1}{3}.
  \end{equation}
  
  Analogous to the definition of $A_4$, we choose $L$ in Lemma \ref{lem:C_max} sufficiently large such that $[-L, L] \supseteq (d_1 - \delta_2, d_2 + \delta_2)$, and let
  \begin{subequations}
    \begin{align}
      B_4 = {}& \left\{ \omega \in B_3 \mid \int^{\lambdatilde_{j - 1}}_{\lambdatilde_j} \tilde{p}_{\omega}(\sigmatilde_j) \sum^N_{i = 1} \frac{1}{(\sigmatilde_i - d_1)^2 + 1}  \dd \sigmatilde_j < \frac{C_L}{\delta_1/16} \right. \label{eq:B4_first_ineq} \\
      {}& \text{ and } \left. \int^{\lambdatilde_{j - 1}}_{\lambdatilde_j} \tilde{p}_{\omega}(\sigmatilde_j) \sum^N_{i = 1} \frac{1}{(\sigmatilde_i - d_2)^2 + 1}  \dd \sigmatilde_j < \frac{C_L}{\delta_1/16} \right\}. \label{eq:B4_second_ineq}
    \end{align}
  \end{subequations}
  We have that $\Prob(B_4) > \delta_1/8$ for large enough $N$. Otherwise, we have that with probability no less than $ \delta_1/8$, at least one of the two inequalities in \eqref{eq:B4_first_ineq} and \eqref{eq:B4_second_ineq} fails. Without of loss of generality, we assume that in probability $\geq \delta_1/16$ inequality \eqref{eq:B4_first_ineq} fails. Then like \eqref{eq:ineq_avarage_at_c1}, we can derive $\E \left( \sum^N_{i = 1} \frac{1}{(\sigmatilde_i - d_1)^2 + 1} \right) \geq C_L$, contradictory to \eqref{eq:est_square_inverse_sum}.
  
  Now consider the function $F_j = \log(N^{1/3} \lvert x_{j1} \rvert^2)$ defined by \eqref{20041121}. For a fixed $\omega \in B_4$, $F_j$ depends on $\sigma_j$, or equivalently $\sigmatilde_j$. Below we express it  as a function of $\sigmatilde_j$ and further decompose it into two parts $F_j(\sigmatilde_j; \omega)=F_j^{(1)}(\sigmatilde_j; \omega)+F_j^{(2)}(\sigmatilde_j; \omega)$, where 
  \begin{equation}
    F_j^{(1)}(\sigmatilde_j; \omega):= \sum^{j - 1}_{i = 1} \log \frac{\reservesigma_j - \lambda_i}{\reservesigma_j - \reservesigma_i}, \quad F_j^{(2)}(\sigmatilde_j; \omega):=\sum^N_{i = j + 1} \log \frac{\reservesigma_j - \lambda_{i-1}}{\reservesigma_j - \reservesigma_i}+\frac13\log N.
  \end{equation}
  We also note that $F_j^{(2)}(\sigmatilde; \omega)$ is an increasing function of $\sigmatilde \in (\lambdatilde_j, \lambdatilde_{j - 1})$, analogous to $F_1(\sigmatilde; \omega)$ as $\sigmatilde \in (\lambdatilde_1, +\infty)$. Similar to $A_5$ in the proof of the $j = 1$ case, we have that for a large enough $N_2 > 0$, the event
  \begin{equation}
    B_5 = \left\{ \omega \in B_4 \mid [-\log N_2, \log N_2] \cap [F_j^{(2)}(\lambdatilde_j + \delta_3; \omega), F_j^{(2)}(\lambdatilde_{j - 1} - \delta_3; \omega)] \neq \emptyset \right\}
  \end{equation}
  satisfies $\Prob(B_5) > \delta_1/16$ for large enough $N$. 
  
  We have, by \eqref{eq:defn_P_GUE_tilde} and \eqref{eq:GUE_P_n_integral},
  \begin{equation}
    F_j^{(2)}(\sigmatilde; \omega) = \int_{\bigcup^N_{i = j + 1} (\sigmatilde_i, \lambdatilde_{i - 1}]} \frac{-1}{\sigmatilde - x} \dd x + \frac{1}{3} \log N.
  \end{equation}
  Hence with $\omega \in A_4$, we have, analogous to \eqref{eq:difference_F_upper} and \eqref{eq:difference_F_lower},
  \begin{equation} \label{eq:difference_F_upper_j}
    F_j^{(2)}(\lambdatilde_{j - 1} - \delta_3; \omega) - F_j^{(2)}(\lambdatilde_j + \delta_3; \omega) < \int_{(-\infty, \lambdatilde_j] } \left( \frac{1}{\lambdatilde_j + \delta_3 - x} - \frac{1}{\lambdatilde_{j - 1} - \delta_3 - x} \right) \dd x <  \log \frac{d_2 - d_1}{\delta_3},
  \end{equation}
  By \eqref{eq:difference_F_upper_j}, with $N_3 = N_2 \cdot ((d_2 - d_1)/\delta_3)$, we have, analogous to \eqref{eq:defn_A_5},
  \begin{equation}
    [F_j^{(2)}(\lambdatilde_j + \delta_3; \omega), F_j^{(2)}(\lambdatilde_{j - 1} - \delta_3; \omega)] \subseteq [-\log N_3, \log N_3], \quad  \text{if $\omega \in B_5$}.
  \end{equation}
  On the other hand, since $\lambdatilde_{j - 1} > d_2 - \delta_2$ and $\sigmatilde_1 < N_1$, we have
  \begin{equation} \label{092301}
    - \log \frac{N_1 - d_2 + \delta_2 + \delta_3}{\delta_3} < \log \frac{\lambdatilde_{j - 1} - \sigmatilde_j}{\sigmatilde_1 - \sigmatilde_j} < F_j^{(1)}(\sigmatilde_j; \omega) < 0, \quad \text{for all } \sigmatilde_j\in [\lambdatilde_{j}+\delta_3, \lambdatilde_{j-1}-\delta_3].
  \end{equation}
  Hence by setting $N_4=N_3 \cdot ((N_1 - d_2 + \delta_2 + \delta_3)/\delta_3)$, we have  for any $\sigmatilde\in [\lambdatilde_j + \delta_3, \lambdatilde_{j - 1} - \delta_3]$,
  \begin{equation}
    [F_j(\lambdatilde_j + \delta_3; \omega), F_j(\lambdatilde_{j - 1} - \delta_3; \omega)] \subseteq [-\log N_4, \log N_4], \quad  \text{if $\omega \in B_5$}.
  \end{equation}
  We then have, analogous to \eqref{eq:A_5_prop} in the $j = 1$ case,
  \begin{equation} \label{eq:within_M_B_part}
    \Prob(-\log N_4 < F_j(\sigmatilde_j; \omega) < \log N_4 \mid \omega) > \frac{1}{3}.
  \end{equation}

  Further, note that if $\sigmatilde \in (\lambdatilde_j, d_2 - \delta_2)$, we have
  \begin{equation}
    \begin{split}
      \frac{\dd}{\dd\sigmatilde} F_j(\sigmatilde; \omega) = {}& \int_{\bigcup^N_{i = j + 1} (\sigmatilde_i, \lambdatilde_{i - 1}]} \frac{\dd x}{(\sigmatilde - x)^2} - \int_{\bigcup^{j - 1}_{i = 1} (\lambdatilde_i, \sigmatilde_i]} \frac{\dd x}{(\sigmatilde - x)^2} \\
      > {}& \int_{[\lambdatilde_j - \delta_2, \lambdatilde_j]} \frac{\dd x}{(\sigmatilde - x)^2} - \int_{[d_2 - \delta_2, +\infty)} \frac{\dd x}{(\sigmatilde - x)^2} = \frac{\delta_2}{(\sigmatilde - \lambdatilde_j)(\sigmatilde - \lambdatilde_j + \delta_2)} - \frac{1}{d_2 - \delta_2 - \sigmatilde}.
    \end{split}
  \end{equation}
  By direct computation, and based on our assumptions on $d_1, d_2, \delta_2, \delta_3$, we have that for all $\omega \in B_5$
  \begin{equation} \label{092320}
    \frac{\dd}{\dd\sigmatilde} F_j(\sigmatilde; \omega) \geq \frac{1}{6\delta_2} > 0, \quad  \sigmatilde \in [\lambda_j + \delta_3, \lambda_j + \delta_2],
  \end{equation}
  which is a counterpart of \eqref{eq:difference_F_lower}.
  
  We note that by \eqref{eq:derivatives_gtilde}, $\tilde{g}''_{\omega, a_k}$ is positive and concave on $(\lambdatilde_j, \lambdatilde_{j - 1})$, so we have for $\sigmatilde \in (\lambdatilde_j, \lambdatilde_{j - 1})$,
  \begin{equation}
    0 < \tilde{g}''_{\omega, a_k}(\sigmatilde) < \max \{ \tilde{g}''_{\omega, a_k}(\lambdatilde_j), \tilde{g}''_{\omega, a_k}(\lambdatilde_{j - 1}) \},
  \end{equation}
  and then for large enough $N$, we can derive in a way parallel to \eqref{bound for f''}
  \begin{equation}
    0 < \tilde{g}''_{\omega, a_k}(\sigmatilde) < \frac{16 C_L}{\delta_1 \delta^2_2}.
  \end{equation}
  
  Moreover, we can show that for all $\omega \in B_5$, we can find a large enough $C_B > 0$ such that 
  \begin{equation} \label{eq:boundedness_g'}
    \lvert \tilde{g}'_{\omega, a_k}(\sigmatilde) \rvert < C_B \quad \text{if $\sigmatilde \in [\lambdatilde_j + \delta_3, \lambdatilde_{j - 1} - \delta_3]$}.
  \end{equation}
  The proof relies on that $\tilde{g}_{\omega, a_k}(\sigmatilde)$ is increasing on $(\lambdatilde_j, \lambdatilde_{j - 1})$ and we can show both $\tilde{g}'_{\omega, a_k}(\lambdatilde_j + \delta_3) > -C_N$ and $\tilde{g}'_{\omega, a_k}(\lambdatilde_{j - 1} - \delta_3) < C_N$. Since the proof techniques are similar to those used for \eqref{eq:f_deriv_lower} and \eqref{eq:f_deriv_upper}, we omit the detail.
  
  We have, analogous to \eqref{eq:conditional_var_first_est},
  \begin{equation}
    \var(F_j(\sigmatilde_j; \omega) \mid \omega \text{ and } \sigmatilde_1 \in [\lambdatilde_j + \delta_3, \lambdatilde_j + \delta_2]) > \epsilon_4.
  \end{equation}
  Also because of the probability inequality \eqref{eq:1/3_est_sigmatilde_j} and the boundedness of $\tilde{g}'_{\omega, a_k}(\sigmatilde)$ given in \eqref{eq:boundedness_g'}, we have that
  \begin{equation} \label{eq:1/3_est_sigmatilde_j_narrow}
    \Prob(\sigmatilde_j \in [\lambdatilde_j + \delta_3, \lambdatilde_j + \delta_2] \mid \omega) > \frac{\delta_5}{3}.
  \end{equation}
  for some $\delta_4 > 0$. Hence we have, analogous to \eqref{eq:final_nondegeneracy},
  \begin{equation} \label{eq:last_var_est}
    \begin{split}
      & \var(F_j(\sigmatilde_j; \omega) \mid \omega \text{ and } -\log N_4 < F_j(\sigmatilde_1; \omega) < \log N_4) \\
      \geq {}& \frac{\Prob(\sigmatilde_j \in [\lambdatilde_j + \delta_3, \lambdatilde_j + \delta_2] \mid \omega)}{\Prob(-\log N_4 < F_j(\sigmatilde_j; \omega) < \log N_4 \mid \omega)} \var(F_j(\sigmatilde_j; \omega) \mid \omega \text{ and } \sigmatilde_j \in [\lambdatilde_1 + \delta_3, \delta_2]) \\
      > & \frac{\delta_5}{3} \delta_4.
    \end{split}
  \end{equation}
  
  Taking $A = B_5$, $M = N_4$, $\epsilon = \delta_1/16$, $\epsilon' = 1/3$ and $\epsilon'' = \delta_4\delta_5/3$, we see that Lemma \ref{lem:tough_lem} is verified in the $j > 1$ case by \eqref{eq:within_M_B_part} and \eqref{eq:last_var_est}.
\end{proof}

Now, we prove Lemma \ref{lem:C_max}

\begin{proof}[Proof of Lemma \ref{lem:C_max}]
  The left-hand side of \eqref{eq:est_square_inverse_sum} can be expressed as
  \begin{equation} \label{eq:square_inverse_as_linear_stat}
    \E \left( \sum^N_{i = 1} \frac{1}{(\sigmatilde_i - x)^2 + 1} \right) = N^{-\frac{1}{3}} \int^{\infty}_{-\infty} \rho_N(t) \frac{1}{(t - (2\sqrt{N} + N^{-1/6} x))^2 + N^{-1/3}}\dd t,
  \end{equation}
  where $\rho_N(t)$ is the empirical density function of $\sigma_1, \dotsc, \sigma_N$. By the property of determinantal processes (cf.~\eqref{eq: mixed corr}), $\rho_N(t) = K^{0, 0}_{\GUEalpha}(t, t) = \K^{0, 0}_{\GUEalpha}(t, t)$, where $K^{0, 0}_{\GUEalpha}$ is defined in \eqref{eq:kernel_GUE} and represented by a double integral formula in \eqref{20041005}. The estimation of $\rho_N(t)$ with $t \in ((-2 + \epsilon)\sqrt{N}, 2\sqrt{N} - N^{-1/10})$ and $t \in [2\sqrt{N} - N^{-1/10}, 2\sqrt{N} - CN^{-1/6}]$ can be done by the same saddle point analysis method as we evaluate $\E(S_x + xN_x)$ in \eqref{eq:GUE_measure_mean_alt} and \eqref{eq:GUE_measure_mean_local_alt} respectively, since $\E(S_x + xN_x)$ is expressed by a very similar double integral formula in \eqref{eq:double_contour_GUE_mean}. For $t > 2\sqrt{N} - CN^{-1/6}$, where $C > 0$, $\rho_N(t)$ can be estimated by using (\ref{091501}) in Appendix \ref{sec:proofs_in_sec_2} and then apply the standard saddle point method to $H_{N, 0}$ and $J_{N, 0}$. For $t \leq (-2 + \epsilon)\sqrt{N}$, similar methods can be applied and we omit the detail. The estimate we need is that for large enough $N$:
  \begin{enumerate}[label=(\roman*)]
  \item 
    (The semicircle law) For $t \in (-2\sqrt{N} + N^{-1/10}, 2\sqrt{N} - N^{-1/10})$, $\rho_N(t) = \frac{1}{2\pi} \sqrt{4N - t^2} (1 + o(1))$. 
  \item
    For $t \in [2\sqrt{N} - N^{-1/10}, 2\sqrt{N} - C N^{-1/6})$ and $t \in (-2\sqrt{N} + C N^{-1/6}, -2\sqrt{N} + N^{-1/10}]$, $\rho_N(t) = \bigO(N^{1/4} (2\sqrt{N} - t)^{1/2})$ and $\rho_N(t) = \bigO(N^{1/4} (2\sqrt{N} + t)^{1/2})$ respectively.
  \item
    For $t \geq 2\sqrt{N} - C N^{-1/6}$ and $t \leq -2\sqrt{N} + C N^{-1/6}$, $\rho_N(t) = \bigO(N^{1/6} e^{-cN^{1/6}(t - 2\sqrt{N})})$ and $\rho_N(t) = \bigO(N^{1/6} e^{cN^{1/6}(2\sqrt{N} + t)})$, for some $c > 0$, respectively.
  \end{enumerate}
  The estimate of $\rho_N(t)$ above and the expression \eqref{eq:square_inverse_as_linear_stat} imply the desired boundedness.
\end{proof}

In the end of this section, we state some simulation results on the distribution of the eigenvector components; see Figures \ref{fig.density} and  \ref{fig.Tail} below. The simulation is done under the following setting for $G_{\totalalpha}$ in (\ref{20040301}): $N=1000$, $k=2$, $\alpha_1= \sqrt{N}$, $\alpha_2=\sqrt{N}-N^{1/6}$.  The simulation results are based on $6000$ replications. 
In Figure \ref{fig.density}, we plot the kernel density estimates (smooth approximations of histograms), $p_X$, for $X=N^{1/3}|x_{11}|^2$, $N^{1/3}|x_{12}|^2$, and $N|x_{13}|^2$.   In Figure \ref{fig.Tail}, we plot the negative logarithm of the tail function, $-\log \mathbb{P}(X>t), t\geq 0.5$,  for $X=N^{1/3}|x_{11}|^2$, $N^{1/3}|x_{12}|^2$, and $N|x_{13}|^2$. Observe that the limiting distribution of $N|x_{13}|^2$ shall be $\Exp(1)$, in light of Corollary \ref{cor:to_thm_2}. However, the yellow curve is apparently above $1$ at $t=0$ in Figure \ref{fig.density}. This is due to a finite $N$ effect, since according to our proof of Corollary \ref{cor:to_thm_2}, the difference between the distribution of  $N|x_{13}|^2$ and the limiting one, $\Exp(1)$, is of order $O(N^{-1/3})$. We also remark here that Figure \ref{fig.Tail} shows (numerically) the difference between the tail behavior of the laws in Theorem \ref{thm:main_thm_2} and that of $\Exp(1)$. A theoretical study of the tails of these laws will  be deferred to  future study. 

\begin{figure}[htb]
    \begin{minipage}[t]{0.48\linewidth}
      \centering
      \includegraphics[width=9.5cm,height=5.5cm]{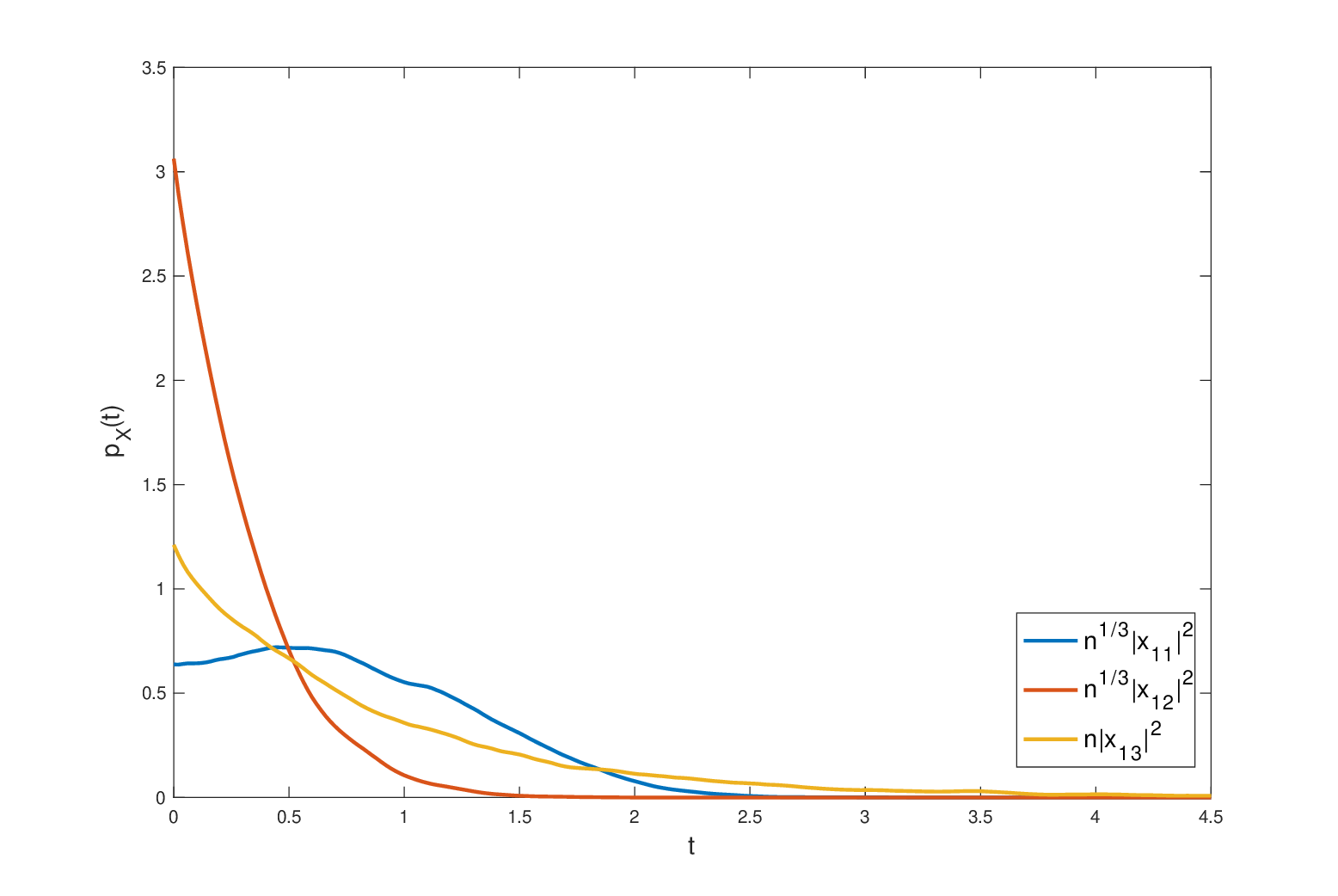}
      \caption{ Kernel Density Estimate}
      \label{fig.density}
    \end{minipage}
    \begin{minipage}[t]{0.48\linewidth}
      \centering
      \includegraphics[width=9.5cm,height=5.5cm]{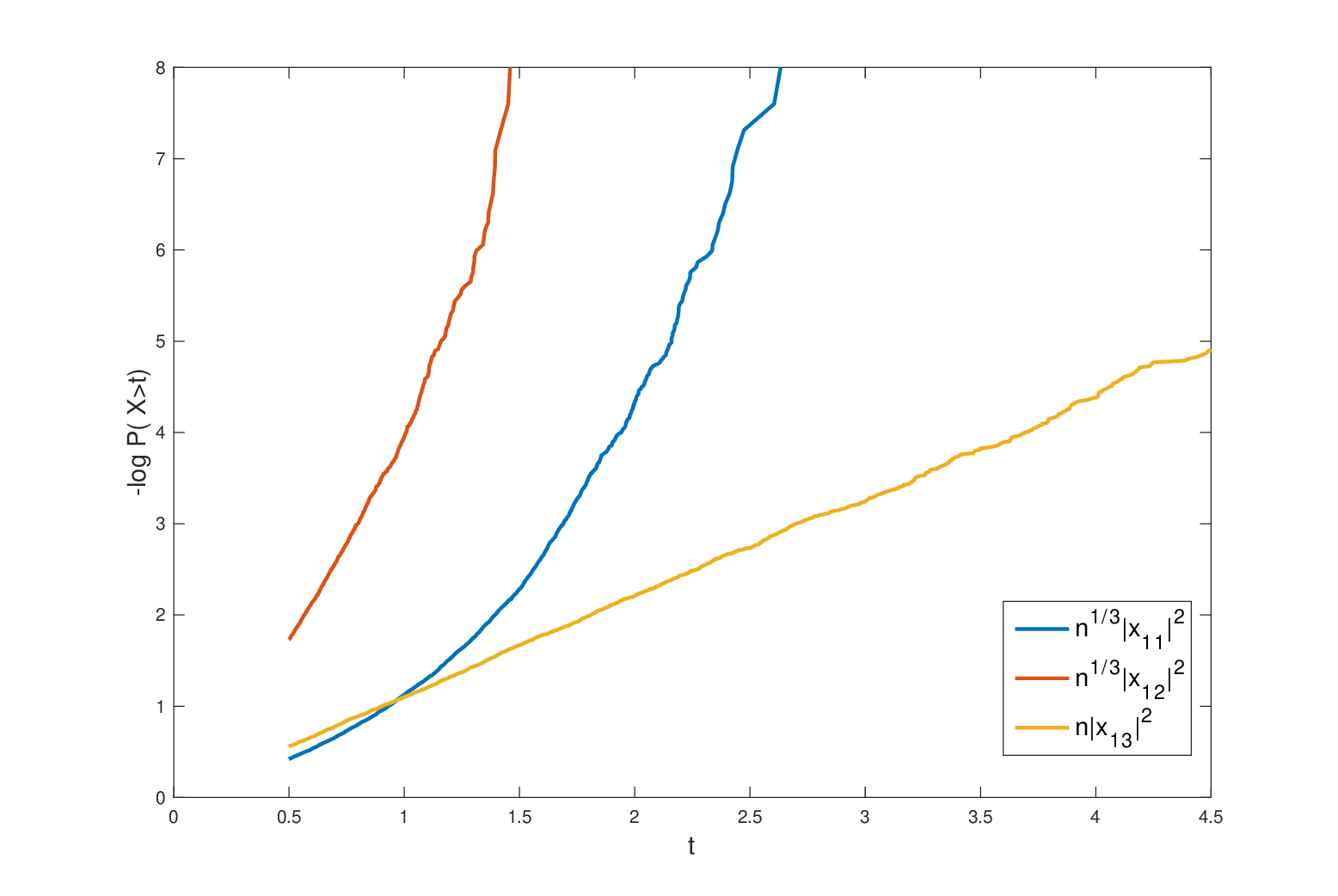}
      \caption{Negative Logarithm of Tail Probability}      
      \label{fig.Tail}
    \end{minipage}
  \end{figure}

\appendix

\section{Proof of results in Section \ref{s.pre}} \label{sec:proofs_in_sec_2}

\paragraph{Proof of Lemma \ref{lem:Airy_limit}}

We only give a sketch of the proof of Lemma \ref{lem:Airy_limit}, because the case that $j_1 = j_2$, and all $\alpha_j = \sqrt{N}$ (equivalent to $a_{k - j + 1} = 0$) for $j = 1, \dotsc, k$ has already been proved in \cite{Peche05} which follows closely the method used in \cite{Baik-Ben_Arous-Peche05}. Our proof is an adaption of that in \cite{Peche05}. The main difference between our lemma and the results in \cite{Baik-Ben_Arous-Peche05} and \cite{Peche05} is that we require $\epsilon$ to be large enough while in \cite{Baik-Ben_Arous-Peche05} and \cite{Peche05}, $\epsilon$ is only required to be positive, because they essentially assume all $a_j = 0$. Our assumption on $\epsilon$ implies that the operators $e^{\epsilon(x - y)} K^{k_1, k_2}_{\Airya}(x, y)$ and $e^{\epsilon(x - y)} K^{j_1, j_2}_{N, \scaled}(x, y)$ are both trace class.

\begin{proof}
  Below in this proof we are going to use notation in \cite{Peche05} that is quite different from the notation used elsewhere in our paper.

  We define, analogous to \cite[Formula (16)]{Peche05},
  \begin{equation} \label{eq:defn_K'_Njj}
    K'_{N, j_1, j_2}(x, y) = N^{-1/6} N^{(j_1 - j_2)/6} e^{N^{1/3}(y - x)} \K^{j_1, j_2}_{\GUEalpha}(2\sqrt{N} + N^{-1/6} y, 2\sqrt{N} + N^{-1/6} x).
  \end{equation}
  We only need to consider the convergence of $e^{\epsilon(y - x)} K'_{N, j_1, j_2}(x, y)$ to $e^{\epsilon(y - x)} \K^{k_1, k_2}_{\Airya}(y, x)$ (pointwise and in trace norm). Analogous to \cite[Formulas (18) and (19)]{Peche05}, we denote $F(z) = z^2/2 - 2z + \log z$ (see \cite[Formula (17)]{Peche05}), $\tilde{w}_c = 1 + \epsilon N^{-1/3}$ (see \cite[Formula (16)]{Peche05}), and define
  \begin{align}
    H_{N, j_2}(x) = {}& \frac{N^{\frac{1}{3}}}{2\pi} \int_{\Gamma} \left[ z^k \prod^k_{i = j_2 + 1} \frac{1}{z - \frac{\alpha_i}{\sqrt{N}}} \right] \exp(-N F(z)) \exp(N^{\frac{1}{3}} x(z - \tilde{w}_c)) \dd z, \\
    J_{N, j_1}(y) = {}& \frac{N^{\frac{1}{3}}}{2\pi} \int_{\gamma} \left[ w^{-k} \prod^k_{i = j_1 + 1} \left( w - \frac{\alpha_i}{\sqrt{N}} \right) \right] \exp(N F(w)) \exp(-N^{\frac{1}{3}} y(w - \tilde{w}_c)) \dd w,
  \end{align}
  where the contours $\Gamma$ and $\gamma$ are defined as in \cite[Formula (14)]{Peche05}. Then we have, analogous to \cite[Proposition 2.1]{Peche05},
  \begin{equation} \label{eq:decom_K'_Njj}
    N^{(j_1 - j_2)/3} e^{\epsilon(y - x)} K'_{N, j_1, j_2}(x, y) = -\int^{\infty}_0 H_{N, j_2}(x + t) J_{N, j_1}(y + t) \dd t.
  \end{equation}
  Next, analogous to \cite[Formulas (21) and (22)]{Peche05}, define
  \begin{align}
    H_{\infty, k_2}(x) = {}& \frac{\exp(-\epsilon x)}{2\pi} \int_{\Gamma_{\infty}} \exp \left( xz - \frac{z^3}{3} \right) \prod^{k_2}_{i = 1} \frac{1}{z - a_i} \dd z, \label{eq:defn_H_infty} \\
    J_{\infty, k_1}(y) = {}& \frac{\exp(\epsilon y)}{2\pi} \int_{\gamma_{\infty}} \exp \left( -yw + \frac{w^3}{3} \right) \prod^{k_1}_{i = 1} (w - a_i) \dd w, \label{eq:defn_J_infty}
  \end{align}
  where the contours $\Gamma_{\infty}$ and $\gamma_{\infty}$ are defined in \cite[Figure 1]{Peche05}. By the arguments in \cite[Sections 2.1 and 2.2]{Peche05}, we have, analogous to \cite[Proposition 2.2]{Peche05}, that for any fixed $y_0 \in \realR$, there exists $C > 0$, $c > 0$, an integer $N_0 > 0$ such that
  \begin{align}
    \left\lvert Z_{N, j_2} H_{N, j_2}(x) - H_{\infty, k_2}(x) \right\rvert \leq {}& \frac{C \exp(-cx)}{N^{1/3}}, && \text{for any $x > y_0$, $N \geq N_0$}, \label{eq:est_H_N} \\
    \left\lvert Z^{-1}_{N, j_1} J_{N, j_1}(y) - J_{\infty, k_1}(y) \right\rvert \leq {}& \frac{C \exp(-cy)}{N^{1/3}}, && \text{for any $y > y_0$, $N \geq N_0$}, \label{eq:est_J_N}
  \end{align}
  where $Z_{N, j} = N^{(j - k)/3} \exp(NF(1))$. On the other hand, we have
  \begin{equation} \label{eq:conv_expr_KAiry}
    e^{\epsilon(y - x)} \K^{k_2, k_1}_{\Airya}(y, x) = -\int^{\infty}_0 H_{\infty, k_2}(x + t) J_{\infty, k_1}(y + t) \dd t.
  \end{equation}
  The proof is finished by using the argument in \cite[Section 3.3]{Baik-Ben_Arous-Peche05}.
\end{proof}

\paragraph{Proof of Lemma \ref{lem:interlacing_Airy}}

By Lemma \ref{lem:Airy_limit}, for any $n$, the joint distribution of $\lambda^{(N - j)}_1, \lambda^{(N - j - 1)}_1, \lambda^{(N - j)}_2, \lambda^{(N - j - 1)}_2, \dotsc, \lambda^{(N - j)}_n$ converges weakly to that of $\xi^{(k - j)}_1, \xi^{(k - j - 1)}_1, \xi^{(k - j)}_2, \xi^{(k - j - 1)}_2, \dotsc, \xi^{(k - j)}_n$ up to a scaling transform. Hence we have that the interlacing inequality \eqref{eq:interlacing_GUE} implies the weak interlacing property
\begin{equation} \label{eq:weak_interlacing}
  +\infty > \xi^{(k - j)}_1 \geq \xi^{(k - j - 1)}_1 \geq \xi^{(k - j)}_2 \geq \xi^{(k - j - 1)}_2 \geq \dotsb \geq \xi^{(k - j)}_n.
\end{equation}
On the other hand, the determinantal structure requires that the point process consisting of $\xi^{(k - j)}_i$ and $\xi^{(k - j - 1)}_l$ is simple, so with probability $1$ the inequalities in \eqref{eq:weak_interlacing} are all strict. So with probability $1$ we have \eqref{eq:interlacing_Airy} by letting $n \to \infty$.

\paragraph{Proof of Lemma \ref{lem:non-asy_est_Airy}}

We prove the lemma in three steps: First the right tail estimate of $\xi^{(k)}_j$ in part \ref{enu:lem:non-asy_est_Airy:1}, then the left tail estimate of $\xi^{(k)}_j$, and at last we prove part \ref{enu:lem:non-asy_est_Airy:2} about the rigidity of $\xi^{(k)}_n$.
\begin{proof}[Proof of the right tail estimate of $\xi^{(k)}_j$]
  We note that
  \begin{equation} \label{eq:Airy_right_tail_j->1}
    \Prob(\xi^{(k)}_j > t) \leq \Prob(\xi^{(k)}_1 > t) \leq \E( \text{\# of } \xi^{(k)}_i \text{ on } [t, +\infty)) = \int^{+\infty}_t K^{k, k}_{\Airya}(x, x) \dd x.
  \end{equation}
  Then by \eqref{eq:form_ext_K},
  \begin{equation} \label{eq:int_density_Airy_k}
    \int^{+\infty}_t K^{k, k}_{\Airya}(x, x) \dd x = \frac{1}{(2\pi \ii)^2} \int_{\gamma} \dd u \int_{\sigma} \dd v \frac{e^{\frac{u^3}{3} - tu}}{e^{\frac{v^3}{3} - tv}} \frac{\prod^k_{j = 1} (u - a_j)}{\prod^k_{j = 1} (v - a_j)} \frac{1}{(u - v)^2}.
  \end{equation}
  Let $\gamma$ and $\sigma$ be deformed into $\gamma_{\standard}(\sqrt{t})$ and $\sigma_{\standard}(-\sqrt{t})$ (c.f. (\ref{eq:defn_standard_gamma_sigma})). By standard saddle point analysis, we find that as $t \to +\infty$, the integral \eqref{eq:int_density_Airy_k} concentrates on the region $u \in \gamma_{\standard}(\sqrt{t}) \cap \{ u - \sqrt{t} = \bigO(t^{-1/4}) \}$ and $v \in \sigma_{\standard}(-\sqrt{t}) \cap \{ v + \sqrt{t} = \bigO(t^{-1/4}) \}$. Then we conclude that as $t \to +\infty$,
  \begin{equation} \label{eq:est_KAiry_tail}
    \int^{+\infty}_t K^{k, k}_{\Airya}(x, x) \dd x = \bigO \left( t^{-3/2} \exp \left( -\frac{4}{3} t^{3/2} \right) \right).
  \end{equation}
  Hence by choosing $C$ properly, we have $\Prob(\xi^{(k)}_j > t) < \int^{+\infty}_t K^{k, k}_{\Airya}(x, x) dx < C e^{-t/C}$.
\end{proof}

\begin{proof}[Proof of the left tail estimate of $\xi^{(k)}_j$]
  We note that by the interlacing property in Lemma \ref{lem:interlacing_Airy}, $\Prob(\xi^{(k)}_j < -t) < \Prob(\xi^{(0)}_j < -t)$, where $\xi^{(0)}_j$ is the $j$-th particle in the determinantal point process defined by the Airy kernel \eqref{eq:defn_Airy_kernel}. Then by \cite{Tracy-Widom94}, with any $\lambda \in (0, 1)$, we have
  \begin{equation} \label{eq:several_particles_ineq_Airy}
    \Prob(\xi^{(0)}_j < -t) = \sum^{j - 1}_{n = 0} E(n; -t) < (1 - \lambda)^{1 - j} \sum^{\infty}_{n = 0} (1 - \lambda)^n E(n; -t) = 2^{j - 1} D(-t, \lambda),
  \end{equation}
  where $E(n; -t)$ is the probability that exactly $n$ particles are in $[-t, \infty)$ as denoted in \cite[Section ID]{Tracy-Widom94}, and $D(-t, \lambda)$ is defined by \cite[Formula (1.17)]{Tracy-Widom94} as
  \begin{equation} \label{eq:defn:D(s, lambda)}
    D(-t, \lambda) = \exp \left( - \int^{\infty}_{-t} (x + t) q(x; \lambda)^2 \dd x \right), \quad \text{where} \quad
    \begin{aligned}
      \frac{\dd q(s; \lambda)}{\dd s^2} = {}& sq(s; \lambda) + 2q^3(s; \lambda), \\
      q(s; \lambda) \sim {}& \sqrt{\lambda} \Ai(s) \text{ as } s \to \infty.
    \end{aligned}
  \end{equation}
  The function $q(s; \lambda)$ is the Ablowitz-Segur solution to the \Painleve\ II equation \cite{Ablowitz-Segur77}, \cite{Ablowitz-Segur81}, its asymptotics at $+\infty$ is given by the Airy function multiplied by constant $\sqrt{\lambda}$. The asymptotic behaviour of $q(s; \lambda)$ has been extensively studied, see \cite{Deift-Zhou95} for a rigorous and systematic discussion. We then derive the upper bound of $D(-t, \lambda)$ for large $t$ from the asymptotics of $q(s; \lambda)$, and finally justify the estimate $\Prob(\xi^{(k)}_j < -t) < C e^{-t/C}$ for some properly chosen $C$.
\end{proof}

\begin{proof}[Proof of the rigidity of $\xi^{(k)}_n$]
  We note that by the interlacing property \eqref{eq:interlacing_GUE}, for all $n > k$,
  \begin{multline}
    \Prob \left( \left\lvert \xi^{(k)}_n + \left( \frac{3\pi n}{2} \right)^{2/3} \right\rvert > n^{\frac{3}{5}} \right) \leq \Prob \left( \text{\# of $\xi^{(0)}_l$ in $\left( -\left( \frac{3\pi n}{2} \right)^{2/3} + n^{\frac{3}{5}}, \infty \right)$ is $\geq n - k$} \right) \\
    + \Prob \left( \text{\# of $\xi^{(0)}_l$ in $\left( -\left( \frac{3\pi n}{2} \right)^{2/3} - n^{\frac{3}{5}}, \infty \right)$ is $< n$} \right).
  \end{multline}
  Since $\xi^{(0)}_n$ are the $n$-th particle in the determinantal point process with the Airy kernel, so the problem is reduced to the rigidity of particles in this determinantal point process. The desired regidity can be deduced from the mean and variance of the number of $\xi^{(0)}_l$ in $(-T, \infty)$ and the Markov inequality. If we denote the number of $\xi^{(0)}_l$ in $(-T, \infty)$ as $v_1(T)$, in the notation of \cite{Soshnikov00b}, then 
  \begin{equation} \label{eq:Soshnikov_est}
    \E(v_1(T)) = 2T^{3/2}/(3\pi) + \bigO(1), \quad \text{and} \quad \var(v_1(T)) = \bigO(\log T)
  \end{equation}
  as $T \to +\infty$, see \cite[Theorem 1 and the paragraph above Theorem 1]{Soshnikov00b} \footnote{It is pointed out in \cite{Landon-Sosoe19} that \cite[Theorem 1]{Soshnikov00b} has a calculational error. See \cite[Theorem 6.2]{Landon-Sosoe19}. Since we only need the magnitude of the variance, this mistake does not affect our argument. We also note that the variance of $\Prob(\# \text{ of } \xi^{(0)}_l \text{ in } (-T, -\infty))$ as $T \to +\infty$ can be computed by the contour integral method that is used in the proof of our Proposition \ref{prop:Airy_measure}.}. That is enough to show that as $l \to \infty$,
  \begin{equation}
    \Prob \left( v_1 \left( \left( \frac{3\pi n}{2} \right)^{2/3} - n^{\frac{3}{5}} \right) \geq n - k \right) = \bigO \left( \frac{\log n}{n^{6/5}} \right), \quad \Prob \left( v_1 \left( \left( \frac{3\pi n}{2} \right)^{2/3} + n^{\frac{3}{5}} \right) < n \right) = \bigO \left( \frac{\log n}{n^{6/5}} \right).
  \end{equation}
  By choosing the constant $c$ properly, we obtain \eqref{eq:ineq:xi_fluc} for all $n \geq 2$.
\end{proof}

\paragraph{Proof of Lemma \ref{lem:non-asy_est_GUE}}

This lemma is analogous to Lemma \ref{lem:non-asy_est_Airy}. We prove it in four steps, with the first three steps parallel to those in the proof of Lemma \ref{lem:non-asy_est_Airy}: First, the right tail estimate of $\sigma_j$ (part \ref{enu:lem:non-asy_est_GUE:1}), next the left tail estimate of $\sigma_j$ (part \ref{enu:lem:non-asy_est_GUE:1}), and then the rigidity for $\sigma_n$ close to the edge (part \ref{enu:lem:non-asy_est_GUE:2}), and at last the rigidity of $\sigma_n$ in the bulk (part \ref{enu:lem:non-asy_est_GUE:3}). In part \ref{enu:lem:non-asy_est_GUE:1} we also need to consider $\sigma_N$, but we omit it, because the estimates for $\sigma_N$ are analogous to the estimate for $\sigma_1$.

\begin{proof}[Proof of the right tail estimate of $\sigma_j$]
  We use the same idea as in \eqref{eq:Airy_right_tail_j->1}, and write
  \begin{equation} \label{eq:GUE_right_tail_j->1}
    \Prob(\sigma_j > 2\sqrt{N} + tN^{-\frac{1}{6}}) \leq \Prob(\sigma_1 > 2\sqrt{N} + tN^{-\frac{1}{6}}) = \int^{\infty}_{2\sqrt{N} + t N^{-1/6}} K^{0, 0}_{\GUEalpha}(x, x) \dd x = \int^{\infty}_t K'_{N, 0, 0}(x, x) \dd x,
  \end{equation}
  where $K'_{N, 0, 0}(x, x)$ is defined in \eqref{eq:defn_K'_Njj}. Although we can evaluate the right-hand side of \eqref{eq:GUE_right_tail_j->1} like \eqref{eq:int_density_Airy_k}, we prefer an indirect method that relies on result and proof of Lemma \ref{lem:Airy_limit}. We recall that as a special case of \eqref{eq:decom_K'_Njj},
  \begin{equation}
    K'_{N, 0, 0}(x, x) = -\int^{\infty}_0 H_{N, 0}(x + t) J_{N, 0}(x + t) \dd t, \label{091501}
  \end{equation}
  and then by \eqref{eq:est_H_N} and \eqref{eq:est_J_N}, there exists $N_0 > 0$ and $C > 0$ such that for $x > 0$, $N > N_0$,
  \begin{equation}
    \left\lvert Z_{N, 0} H_{N, 0}(x) - H_{\infty, k}(x) \right\rvert \leq \frac{C \exp(-cx)}{N^{1/3}}, \quad \left\lvert Z^{-1}_{N, j_1} J_{N, 0}(x) - J_{\infty, k}(x) \right\rvert \leq \frac{C \exp(-cy)}{N^{1/3}},
  \end{equation}
  where $Z_{N, 0}= N^{-k/3} \exp(-3N/2)$ and $H_{\infty, k}, J_{\infty, k}$ are defined in \eqref{eq:defn_H_infty} and \eqref{eq:defn_J_infty}. Hence by the very rough estimate (whose proof is omitted) that $H_{\infty, k}(x) = \bigO(1)$ and $J_{\infty, k}(x) = \bigO(1)$ for all $x > 0$, and with the help of \eqref{eq:conv_expr_KAiry}, we have that
  \begin{equation}
    K'_{N, 0, 0}(x, x) - K^{k, k}_{\Airya}(x, x) = \bigO(N^{-1/3} e^{-cx}), \quad \text{for all $x > 0$ and $N > N_0$}.
  \end{equation}
  Therefore, the desired right tail estimate of $\sigma_j$ is implied by the estimate \eqref{eq:est_KAiry_tail} for the right tail estimate of $\xi^{(k)}_j$.
\end{proof}

\begin{proof}[Proof of the left tail estimate of $\sigma_j = \lambda^{(N)}_j$]
  We use the same idea as in the proof of the left tail estimate of $\xi^{(k)}_j$, that $\Prob(\sigma_j < 2\sqrt{N} - tN^{-1/6}) \leq \Prob(\lambda^{(N - k)}_j < 2\sqrt{N} - tN^{-1/6})$, so it is not hard to see that it suffices to prove that there exists $C > 0$ such that
  \begin{equation} \label{eq:largest_GUE_eigenvalues}
    \Prob(\lambda^{(N - k)}_j < 2\sqrt{N - k} - t(N - k)^{-\frac{1}{6}}) < C e^{-t/C}, \quad \text{for all } 2 \leq t \leq 2(N - k)^{2/3}.
  \end{equation}
  where $\lambda^{(N - k)}_j$ is the $j$-th largest eigenvalue of a GUE random matrix with dimension $N - k$. The $j = 1$ case of \eqref{eq:largest_GUE_eigenvalues} exists in literature, see \cite[Section 5.3, especially Formula (5.16)]{Ledoux07}, where a stronger version of \eqref{eq:largest_GUE_eigenvalues} is derived in a very accessible way. The $j > 1$ case of \eqref{eq:largest_GUE_eigenvalues} is not found in literature, to the best knowledge of the authors. However, we can extend the method in \cite[Section 5.3]{Ledoux07} to solve this case. To see it, we note that like \eqref{eq:several_particles_ineq_Airy}, with $\lambda \in (0, 1)$, we have
  \begin{equation}
    \begin{split}
      \Prob(\lambda^{(N - k)}_j < 2\sqrt{N - k} - t(N - k)^{-\frac{1}{6}}) = {}& \sum^{j - 1}_{n = 0} E(n; 2\sqrt{N - k} - t(N - k)^{-\frac{1}{6}}) \\
      < {}& (1 - \lambda)^{1 - j} \sum^{\infty}_{n = 0} (1 - \lambda)^n E(n; 2\sqrt{N - k} - t(N - k)^{-\frac{1}{6}}) \\
      = {}& (1 - \lambda)^{1 - j} \det \left( \Idmatrix - \lambda K \right),
    \end{split}
  \end{equation}
  where $K$ is the $N \times N$ matrix whose $(m, n)$ entry is
  \begin{equation}
    \langle P_{m - 1}, P_{n - 1} \rangle_{L^2((1 - \frac{t}{2}(N - k)^{-2/3}, \infty), \dd\mu)},
  \end{equation}
  such that the meanings of $P_m$ and $\dd\mu$ are the same as in \cite[Formula (1.11)]{Ledoux07}. Then by the same arguments that leads to \cite[Formula (5.14)]{Ledoux07}, we have
  \begin{equation} \label{eq:largest_several_GUE}
    \det \left( \Idmatrix - \lambda K \right) = \prod^N_{i = 1} (1 - \lambda \rho_i) \leq e^{-\frac{1}{2} \sum^N_{i = 1} \rho_i} = \exp \left( -\lambda N \mu^N \left( (1 - \frac{t}{2}(N - k)^{-2/3}, \infty) \right) \right),
  \end{equation}
  where $\rho_i$ are the eigenvalues of $K$, and $\mu^N$ is the measure defined in \cite[Formula (1.4)]{Ledoux07}. We note that if we let $\lambda = 1$ in \eqref{eq:largest_several_GUE}, then \eqref{eq:largest_several_GUE} is equivalent to \cite[Formula (5.14)]{Ledoux07}. At last, using the estimate of $\mu^N((1 - \frac{t}{2}(N - k)^{-2/3}, \infty))$ given in \cite[Section 5.3]{Ledoux07}, we derive an estimate of $\det \left( \Idmatrix - \lambda K \right)$, which yields the desired estimate of $\Prob(\lambda^{(N - k)}_j < 2\sqrt{N - k} - t(N - k)^{-1/6})$ and $\Prob(\sigma_j < 2\sqrt{N} - tN^{-1/6})$. Finally, we note that essentially the idea of the proof above is in \cite{Widom02}.
\end{proof}

\begin{proof}[Proof of the rigidity of $\sigma_n$ for $n \leq CN^{1/10}$]
  As in \eqref{eq:ineq:xi_fluc_GUE}, we note that \eqref{eq:ineq:xi_fluc_GUE} is analogous to \eqref{eq:ineq:xi_fluc}, and can be proved by an analogous argument. Instead of \eqref{eq:Soshnikov_est}, we have that if $v^{(n)}_1(T)$ is the number of eigenvalues of an $n$-dimensional GUE random matrix in the interval $(2\sqrt{n} - n^{-1/6}T, +\infty)$, then as $n \to \infty$, $T \geq T_0$ a positive constant, and $T/n = o(1)$, by the result of \cite{Gustavsson05} \footnote{The mean estimate is given in \cite[Lemma 2.2]{Gustavsson05}, and the variance estimate is given in \cite[Lemma 2.3]{Gustavsson05} under an additional condition that $T \to \infty$ as $n \to \infty$. However, as pointed out by \cite[Remark under Theorem 6.3]{Landon-Sosoe19}, if we only need a crude estimate as in \eqref{eq:Gustavsson_est}, then the argument in \cite{Gustavsson05} works for all $T > T_0 > 0$.}
  \begin{equation} \label{eq:Gustavsson_est}
    \E(v^{(n)}_1(T)) = 2T^{3/2}/(3\pi) + \bigO(1), \quad \text{and} \quad \var(v^{(n)}_1(T)) = \bigO(\log T).
  \end{equation}
  Then we prove \eqref{eq:ineq:xi_fluc_GUE} by the same argument as the proof of \eqref{eq:ineq:xi_fluc}.
\end{proof}

\begin{proof}[The proof of the regidity of $\sigma_n$ as in \eqref{eq:rigidity_Yau}]
  This is a direct consequence of the interlacing property $\lambda^{(N - k)}_n \leq \sigma_n \leq \lambda^{(N - k)}_{n - k}$ and the rigidity of GUE eigenvalues in \cite[Theorem 2.2]{Erdos-Yau-Yin12}, which states the rigidity of eigenvalues for Wigner matrices that of which the GUE random matrices are a special case.
\end{proof}

\end{document}